\documentclass{article}
\usepackage{blindtext}
\usepackage[a4paper, total={6in, 8in}]{geometry}
\usepackage{graphicx} 
\usepackage{amsthm}
\usepackage{amsmath}
\usepackage{amsfonts}
\usepackage{amsopn}
\usepackage{comment}
\usepackage{amssymb}
\usepackage{hyperref}
\usepackage{bussproofs}
\usepackage[all, 2cell]{xy}
\usepackage[all]{xy}
\usepackage{rotating}
\usepackage{lscape}
\usepackage{float}
\usepackage{dsfont}
\usepackage{multicol}
\usepackage{wasysym}
\usepackage[outline]{contour}
\usepackage{tikz}
\usepackage{tikz-cd}
\usepackage[title]{appendix}
\usepackage{mathbbol}
\usetikzlibrary{shapes.geometric, arrows, positioning, backgrounds}

\theoremstyle{definition}
\newtheorem{theorem}{Theorem}[section]

\theoremstyle{definition}
\newtheorem{definition}[theorem]{Definition}

\theoremstyle{definition}

\theoremstyle{definition}

\theoremstyle{definition}

\theoremstyle{definition}

\theoremstyle{definition}
\newtheorem{prop}[theorem]{Proposition}

\theoremstyle{definition}
\newtheorem{lemma}[theorem]{Lemma}

\theoremstyle{definition}

\theoremstyle{definition}

\theoremstyle{definition}

\theoremstyle{definition}
\newtheorem{example}[theorem]{Example}

\theoremstyle{definition}

\theoremstyle{definition}

\theoremstyle{question}

\theoremstyle{definition}
\newtheorem{convention}[theorem]{Convention}

\theoremstyle{definition}

\DeclareMathOperator{\bang}{!}

\title{Intuitionistic Linear Logic with Subexponentials: Type Theory, Categorical Models and Realisability
\footnote{The author is sincerely grateful to Grigory Kondyrev for providing quite a few 
insightful and original ideas and for inspiring the author to revise his approach in many aspects. 
The author is also thankful to Anton Ayzenberg, Murdoch James Gabbay and Nikita Repeev for remarks improving
the ideas of the paper.
The author also thanks Stepan Kuznetsov for the initial impulse resulted in this research.}}
\author{Daniel Rogozin}
\date{ }

\newcommand{\Ob}[1]{\operatorname{Ob}({\mathcal{#1}})}

\providecommand{\keywords}[1]
{
  \textbf{\textit{Keywords---}} #1
}

\usetikzlibrary{cd}
\makeatletter
\tikzcdset{
  eq node/.style={
    commutative diagrams/math mode=false, anchor=center},
  eq/.style={
    phantom,
    /tikz/every to/.append style={
      edge node={node[commutative diagrams/eq node]
        {\@eqnswtrue\make@display@tag\ltx@label{#1}}}}}}
\makeatother

\begin{document}  

\maketitle

\begin{abstract}
In this paper, we present a typed lambda calculus ${\bf SILL}(\lambda)_{\Sigma}$, a type-theoretic version of multiplicative intuitionistic linear logic with subexponentials,
that is, we have many comonadic resource modalities with some interconnections between them given by a subexponential signature $\Sigma$. 
We introduce the concept of a $\Sigma$-assemblage to characterise models of ${\bf SILL}(\lambda)_{\Sigma}$ by expanding the concept of a linear category where one has multiple resource comonads and symmetric lax monoidal comonad morphisms. 
We also generalise several known results from linear logic and show that every $\Sigma$-assemblage can be viewed as a symmetric monoidal closed category equipped with a family
of monoidal adjunctions and morphisms by modernising and generalising Benton's results by involving the formal theory of comonads in the fashion of Street. 
We give a stronger 2-categorical characterisation of $\Sigma$-assemblages and 
show that the 2-category of $\Sigma$-assemblages 1,2-fully faithfully embeds into the 2-category of particular families of monoidal adjunctions, their left morphisms and transformations, that is, polymodal expansions of linear-non-linear models.
In the final section, we describe realisability models for the particular case of a three-element subexponential signature
by describing BCI algebras with extra operators viewed as applicative morphisms and assemblies over them.
\end{abstract}


\section{Introduction}

The exponential modality $\bang$ is a comonadic operator allowing one to introduce the lacking weakening and contraction
inference rules in linear logic \cite{girard1987linear}. $\bang$ is given by the so-called promotion and dereliction rules
(aka introduction and elimination rules) alongside the weakening and contraction rules for modalised formulas. 
The standard presentation consists of the below Gentzen-style inference rules:

\begin{minipage}{0.45\textwidth}
  \begin{prooftree}
    \AxiomC{$\bang A_1, \ldots, \bang A_n \vdash A$}
    \RightLabel{$\bang_{\bf R}$}
    \UnaryInfC{$\bang A_1, \ldots, \bang A_n \vdash \bang A$}
  \end{prooftree}

  \begin{prooftree}
    \AxiomC{$\Gamma \vdash B$}
    \RightLabel{${\bf W}$}
    \UnaryInfC{$\Gamma, \bang A \vdash B$}
  \end{prooftree}
\end{minipage}
\hfill
\begin{minipage}{0.45\textwidth}
  \begin{tabular}{p{\textwidth}}
    \begin{prooftree}
      \AxiomC{$\Gamma, A \vdash B$}
      \RightLabel{$\bang_{\bf L}$}
      \UnaryInfC{$\Gamma, \bang A \vdash B$}
    \end{prooftree}

    \begin{prooftree}
      \AxiomC{$\Gamma, \bang A, \bang A \vdash B$}
      \RightLabel{${\bf C}$}
      \UnaryInfC{$\Gamma, \bang A \vdash B$}
    \end{prooftree}
  \end{tabular}
\end{minipage}

\vspace{\baselineskip}

Subexponentials are operators of the form $\bang_{s}$ that may allow (or not) weakening or contraction where $s$ ranges over some set of indices. 
We have a family of non-equivalent modal operators $\bang_{s}$ indexed by elements of a set $I$, and each $\bang_{s}$ has its own ability to introduce
this or that structural rule. Subexponential modalities were initially introduced in \cite{danos1993structure}
as part of ``multicolour linear logic'' and then developed in \cite{nigam2017subexponentials}
to propose a proof-theoretic framework for concurrency theory and linear logic programming.
We refer the reader interested in practical uses of the logical framework based on linear logic with subexponentials to \cite{nigam2009algorithmic}, where the authors
develop a simple specification language expressive enough to model data structures and iterative procedures.
Subexponentials in a non-commutative setting are also of interest in mathematical linguistics, see \cite{kanovich2019subexponentials} and \cite{mcpheat2023categorical}.
From a linguistic point of view, structural rules are generally pathological since they change the structure of a sentence,
so they might not stay grammatically correct. Still, the contraction rule allows dealing with such cases, where the localised
duplication is unavoidable, for example, in analysing so-called ``parasitic extractions'' as discussed in \cite{kanovich2017undecidability}.

Linear logic also provides a framework for functional programming with more flexible mechanisms of resource management
by adapting ideas of the Curry-Howard correspondence, see \cite{wadler1990linear} and \cite{girard1987linear}.
Some of those ideas have already found their practical implementation (to a certain extent) in such functional languages
as Haskell (\cite{bernardy2017linear}) and Idris 2 (\cite{brady2021idris}). One can transfer the same ideas
to other substructural logics, in particular affine logic admitting weakening and relevant logic admitting contraction.
Programs well-typed in affine lambda calculus, for example, are known to terminate in linear time (see \cite[§1]{curien2005introduction}),
whereas type systems based on relevant logic are of interest for compiler optimisations, see \cite{walker2005substructural}, and intersection types, see
\cite{bakel2011strict}.  

As regards the questions of semantics, linear logic has a number of semantic frameworks reflecting the intuition
that we interpret provability through resource-aware actions, not statements being valid in the sense of classical or intuitionistic logic,
see \cite[Part III, Part VI]{girard2011blind} for a more detailed discussion on geometry of interaction and coherence spaces.
Categorical semantics of linear logic, as it is presented, for example, in \cite{abramsky2011introduction} and \cite{mellies2009categorical},
generalises the formal properties of coherence spaces and allows interpreting proofs as morphisms in symmetric monoidal closed categories.

In this paper, we introduce a type-theoretic formulation of multiplicative intuitionistic linear logic
with subexponentials. We consider a linear type theory called ${\bf SILL}(\lambda)_{\Sigma}$, where $\Sigma$ is a so-called \emph{subexponential signature},
that is, a preorder \footnote{We do not require a subexponential signature to be finite since the cardinality aspects do not affect general semantic matters.} with two upward closed subsets $W$ and $C$, and we have a family of operators $(\bang_s)_{s \in \Sigma}$ with the following intuition.
If $s_1 \preceq s_2$, then $\bang_{s_2}$ is stronger than $\bang_{s_1}$ in the following sense: $\bang_{s_2} A \to \bang_{s_1} A$ for any $A$. 
If $s \in W$, then the corresponding $\bang_s$ admits weakening, i.e. the formula $\bang_s A \to \mathds{1}$ is provable for any $A$.
Similarly, if $s \in C$, then $\bang_s A \to \bang_s A \otimes \bang_s A$ is provable for any $A$. Type-theoretically, we have a similar intuition:
\begin{itemize}
  \item $M : \bang_s A$ for $s \in W \cap C$ means that $M$ is an element of type $A$ that we can use arbitrarily many times.
  So $\bang_s$ has the behaviour of the usual $\bang$ operator from linear logic.
  That is, $M$ is an element of an \emph{intuitionistic} (or \emph{normal}) type.
  \item $M : \bang_s A$ for $s \in C \setminus W$ means that $M$ is an element of $A$ that one should consume \emph{at least} once.
  In other words, $M$ is an element of a \emph{relevant} type.
  \item $M : \bang_s A$ for $s \in W \setminus C$ means that $M$ is an element of $A$ that one should consume \emph{at most} once.
  In other words, $M$ is an element of an \emph{affine} type.
\end{itemize}

We also consider the particular case ${\bf SILL}(\lambda)_{\Sigma_3}$ as it satisfies the principle introducing a naturally occurring policy
of how several modes of resource management are connected with each other. Informally, this principle can be formulated
with two (meta)implications (aka symmetric lax monoidal comonad morphisms in category-theoretic terms): $\bang_{\bf i} A \vdash \bang_{\bf r} A$ and $\bang_{\bf i} A \vdash \bang_{\bf a} A$.
That is, if one can duplicate (destroy) an object in a ``normal'' way,
then one can duplicate (or destroy) this entity as an object of a relevant (or affine) type. However, this is not the only way of expanding linear logic with modalities since 
in many cases the presence of many non-equivalent subexponentials with more complicated interaction between them
is also of interest, for example, in concurrent programming, see \cite{olarte2015subexponential}. 
In such generalisations, we consider the generalised form of the promotion rule as it is given in \cite{kanovich2019subexponentials}:
\begin{prooftree}
  \AxiomC{$\bang_{s_1} A_1, \ldots, \bang_{s_n} A_n \vdash A$}
  \AxiomC{$s \preceq s_1, \ldots, s_n$}
  \BinaryInfC{$\bang_{s_1} A_1, \ldots, \bang_{s_n} A_n \vdash \bang_s A$}
\end{prooftree}
which is equivalent to the principle $\bang_{s_1} A_1 \otimes \bang_{s_2} A_2 \vdash \bang_s (A_1 \otimes A_2)$ for $s \preceq s_1, s_2$.
We develop the Curry-Howard correspondence for the whole family of subexponential enrichments of intuitionistic linear logic alongside their categorical semantics
by interpreting the preordering relation $s_1 \preceq s_2$ as the symmetric lax monoidal comonad morphism $\bang_{s_2} \Rightarrow \bang_{s_1}$.

In categorical semantics of linear logic, the exponential modality corresponds to a comonad associating 
a cocommutative comonoid with every object of the underlying symmetric monoidal category, see \cite[§1.8.4]{abramsky2011introduction}, \cite{seely1989linear}, \cite[Chapter 12]{Troelstra1992-TROLOL} and \cite{benton1993linear}. 
In computational terms, $\bang$ allows relaxing the resource usage restrictions: by default anything can be used exactly once, but with $\bang$
we can locally address to a particular formula as many times as we want. Logically, $\bang$ allows embedding
intuitionistic logic into linear logic and manipulating duplications and deletions in a restricted, localised way.

Subexponentials have been already studied quite comprehensively from the aspects of proof theory, linguistic applications
and such topics in applied computer science as epistemic and spatial concurrent constraint programming (see \cite[§7-8]{nigam2017subexponentials}).
Semantic aspects play a smaller role in developing practical frameworks, the semantic analysis would allow us to equip subexponentials with mathematical content
and also provide the framework with several resource policies within a single system with the denotational semantics.
Although the semantic aspects of subexponentials have been overlooked for a long time, Ovalle in \cite{Ovalle2026} 
provides an ideologically similar categorical interpretation of subexponentials with comonads over enriched symmetric monoidal categories and their morphisms. 
Our approach is more concerned about the generalisation of Benton's linear-non-linear models with the 2-categorical concepts rather than studying symmetric monoidal categories
with comonadic modalities themselves.

\subsection{Map of the paper}
In Section~\ref{sill:lambda:syntax}, we axiomatise a family of natural deduction calculi ${\bf SILL}(\lambda)_{\Sigma}$ with proof-terms
and proof normalisation rules. We favour the Curry-Howard style presentation
of all those calculi for two reasons. First, given such a presentation, one can think of ${\bf SILL}(\lambda)_{\Sigma}$ as a type-theoretic foundation for a strongly linearly typed functional programming language with several policies
of how the user consumes data. Secondly, proof-terms in our approach come along with the corresponding proof-conversion rules,
so if proof-terms reflect morphisms in syntax, then the conversion rules syntactically represent the axioms of a $\Sigma$-assemblage.
Thus, the semantic interpretation is defined for both derivability and proof transformations.

The rest of the sections in the paper develop a semantic analysis for ${\bf SILL}(\lambda)_{\Sigma}$
and how categorical models of those systems can be represented. 
Section~\ref{preliminaries} contains all necessary preliminary concepts related to 2-categories and symmetric monoidal categories.
In Section~\ref{cocteau:categories:sem}, we introduce the notion of a $\Sigma$-assemblage, that is, a symmetric monoidal closed 
category with a family of symmetric lax monoidal comonads (with the resource policies determined by the corresponding indices) and relations between them. 
Further, we formulate the theorem that every $\Sigma$-assemblage is a model of
the many-sorted equational theory generated by ${\bf SILL}(\lambda)_{\Sigma}$-proof conversions. 

In Section~\ref{cocteau:marais}, we first elaborate on how one can characterise so-called relevant categories,
that is, symmetric monoidal categories where each object $A$ can be copied with a natural transformation $\operatorname{copy}$
with the components $\operatorname{copy}_A : A \to A \otimes A$. We treat relevant categories separately because of the important difference
between them and Cartesian categories in the following aspect. It is well known that a symmetric monoidal category $\mathcal{C}$
is Cartesian if and only if the forgetful functor ${\bf coMon}(\mathcal{C}) \to \mathcal{C}$ is an isomorphism, see, e.g., \cite[Corollary 19]{mellies2009categorical},
where ${\bf coMon}(\mathcal{C})$ is the category of cocommutative comonoids in $\mathcal{C}$.
As already discussed in \cite{rel:cat:nlab}, we only have a weaker criterion: a symmetric monoidal category
is relevant if and only if the forgetful functor ${\bf coSem}(\mathcal{C}) \to \mathcal{C}$ has a strict symmetric monoidal section that does not have 
to be an isomorphism. The aforementioned
nLab article gives neither a complete proof of this fact nor a reference, we give a proof, which we have not found in this form in the literature.
Further, we modify Benton's results from \cite{benton1994mixed} connecting
linear-non-linear models and linear categories. We introduce $\Sigma$-boutons
to characterise $\Sigma$-assemblages with symmetric monoidal categories 
equipped with resource modalities using the characterisation results on relevant categories from the previous section. 
Generally, we unveil each comonad as a monoidal adjunction with the corresponding Eilenberg-Moore categories of coalgebras. 
We provide the 2-categorical approach based on the formal theory of comonads
to characterise $\Sigma$-assemblages with monoidal adjunctions, 
but we also treat strict functors of $\Sigma$-assemblages and their natural transformations. 
We modernise the presentation of unveiling models of ${\bf SILL}(\lambda)_{\Sigma}$ with Benton-style adjunctions via 2-category theory
and the formal theory of comonads in the fashion of Street \cite{street1972formal} in order to track the relations between different resource policies as morphisms between adjunctions.

In Section~\ref{realisability}, we give the realisability interpretation of ${\bf SILL}(\lambda)_{\Sigma_3}$ in order to equip proofs in this type theory
with low-level computational data coming from substructural combinatory algebras. We define relational subexponential linear combinatory algebras, that is, BCI-algebras
equipped with three comonadic applicative morphisms encoding the behaviour of the exponential as well as the affine and the relevant subexponential. We show that
the categories of assemblies and modest sets over a relational subexponential linear combinatory algebra are $\Sigma_3$-assemblages, and therefore, models of ${\bf SILL}(\lambda)_{\Sigma_3}$.

Section~\ref{conclusion} is a brief summary with the discussion on further research directions.
\section{Syntax of Linear Type Theory with Subexponentials}~\label{sill:lambda:syntax}

\begin{definition}
  Let $(I, \preceq)$ be a preorder, a \emph{subexponential signature}
  is a quadruple $\Sigma = (I, \preceq, W, C)$ where $W$ (weakening) and $C$ (contraction)
  are upward closed subsets of $I$.
\end{definition}
For more convenience, we do not distinguish the carrier preorder of a subexponential signature from a subexponential signature itself,
so $s \in \Sigma$ means $s \in I$.

The intuitive idea behind the notion of a subexponential signature is that
we have a preordered family of indices of subexponential modal operators $\{ \bang_s \: | \: s \in \Sigma \}$
and $W$ and $C$ are subsets of those indices whose operators admit weakening and contraction respectively.
The preorder relation on subexponential indices has the following informal semantics. If $s_1 \preceq s_2$,
the operator $\bang_{s_2}$ is stronger than $\bang_{s_1}$, that is, $\bang_{s_2} A \vdash \bang_{s_1} A$.
This is why we require $W$ and $C$ to be upward closed: the property of admitting weakening or contraction
is inherited by stronger modalities from weaker ones. We capture this idea rigorously below. 

The definition of a well-formed formula is specified by the corresponding grammar.

\begin{definition}~\label{formula:def}
  Let $\Sigma = (I, \preceq, W, C)$ be a subexponential signature and let $\{ p_i \: | \: i < \omega \}$ 
  be a countable set of propositional variables. The set of $\Sigma$-formulas is generated by the following grammar:
  \[
    A, B ::= p_i \: | \: \mathds{1} \: | \: (A \multimap B) \: | \: (A \otimes B) \: | \: (\bang_s A)_{s \in \Sigma}
  \]
\end{definition}

By default, subexponentials bind more strongly than binary connectives, whereas the tensor product connective has a higher priority than linear implication.

If $s \in W \: (C)$, then $\bang_s$ is an \emph{affine} (\emph{relevant}) subexponential. $\bang_s$ is \emph{exponential}
if $s \in W \cap C$.

The formalism we present below is inspired by \cite{troelstra1995natural} and \cite{benton1993linear}. In fact, we can think of it as the expansion of the 
linear type theory with additional subexponentials and relations between them defined by the structure of the underlying preorder. Let us define linear terms first.

\begin{definition}[Preterms]

  Let $\Sigma = (I, \preceq, W, C)$ be a subexponential signature and let $\operatorname{Var} = \{ x_i \: | \: i < \omega \}$ be a countable alphabet of variables. 
  The set of preterms is generated by the following BNF-grammar.
  \[
  \begin{array}{lll}
  & M, N ::= x_i \: | \: {\bf 1} \: | \: {\bf let} \: {\bf 1} = M \: {\bf in} \: N \: | \: \lambda x. M \: |  \: M N \: | \: M \otimes N \: | \: {\bf let} \: x \otimes y \: = \:M \: {\bf in} \: N \: | & \\
  & \:\:\:\: (\bang_s M \: {\bf with} \: \vec{x} = \vec{N})_{s \in \Sigma} \: | \: ({\bf der}_s(M))_{s \in \Sigma} \: | \: ({\bf del}_{s}(M; N))_{s \in W} \: | \: ({\bf let}_s \: x, y = M \: {\bf in} \: N)_{s \in C} \: & 
  \end{array}
  \]
\end{definition}

A preterm $M$ is a \emph{linear lambda term} (or simply \emph{term}) if it is well-typed, that is, there is a context $\Gamma$ and a type $A$ such that $\Gamma \vdash M : A$
is derivable by the typing rules from Figure~\ref{fig:1}. So we assume that every free variable occurs exactly once as in the term assignment for the multiplicative-exponential fragment of intuitionistic linear logic. 
We do not bother ourselves with the rigorous definitions of free variables and linear substitution, the reader can transfer them from \cite[Definition 3.14, Definition 3.15]{UCAM-CL-TR-346}.

A \emph{context} $\Gamma$ is a finite multiset of variable declarations of the form $x : A$, where $x$ is a variable and 
$A$ is a formula. A \emph{typing judgement} is a typing declaration $M : A$ for a term $M$.
$\Gamma \vdash M : A$ stands for ``a typing judgement $M : A$ is provable from $\Gamma$'' by the inference rules from Figure~\ref{fig:1}.
Such a term assignment directly generalises the linear typed lambda calculus for intuitionistic linear logic from,
for example, \cite[Chapter 6]{Troelstra1992-TROLOL} for an arbitrary subexponential signature $\Sigma$.

\begin{figure}[H]
  \hrule
  \caption{System ${\bf SILL}(\lambda)_{\Sigma}$}~\label{fig:1}

  \begin{minipage}[b]{0.33333\textwidth}
    \begin{small}
      \begin{prooftree}
        \AxiomC{$ $}
        \RightLabel{$\mathds{1}{\bf I}$}
        \UnaryInfC{$\vdash {\bf 1} : \mathds{1}$}
      \end{prooftree}
    \end{small}
    \end{minipage}%
    \begin{minipage}[b]{0.33333\textwidth}
      \begin{small}
        \begin{prooftree}
          \AxiomC{$ $}
          \RightLabel{\bf ax}
          \UnaryInfC{$x : A \vdash x : A$}
        \end{prooftree}
        \end{small}
    \end{minipage}%
    \begin{minipage}[b]{0.33333\textwidth}
    \begin{small}
      \begin{prooftree}
        \AxiomC{$\Gamma \vdash M : \mathds{1}$}
        \AxiomC{$\Delta \vdash N : B$}
        \RightLabel{$\mathds{1}{\bf E}$}
        \BinaryInfC{$\Gamma, \Delta \vdash {\bf let} \: {\bf 1} = M \: {\bf in} \: N : B$}
      \end{prooftree}
    \end{small}
    \end{minipage}
  \hfill
  \begin{minipage}{0.45\textwidth}
    \begin{small}
    \begin{prooftree}
      \AxiomC{$\Gamma, x : A \vdash M : B$}
      \RightLabel{$\multimap{\bf I}$}
      \UnaryInfC{$\Gamma \vdash \lambda x. M : A \multimap B$}
    \end{prooftree}

    \begin{prooftree}
      \AxiomC{$\Gamma \vdash M : A$}
      \AxiomC{$\Delta \vdash N : B$}
      \RightLabel{$\otimes{\bf I}$}
      \BinaryInfC{$\Gamma, \Delta \vdash M \otimes N : A \otimes B$}
    \end{prooftree}
  \end{small}
  \end{minipage}
  \begin{minipage}{0.45\textwidth}
    \begin{tabular}{p{\textwidth}}
      \begin{small}
      \begin{prooftree}
        \AxiomC{$\Gamma \vdash M : A \multimap B$}
        \AxiomC{$\Delta \vdash N : A$}
        \RightLabel{$\multimap{\bf E}$}
        \BinaryInfC{$\Gamma, \Delta \vdash M N : B$}
      \end{prooftree}

      \begin{prooftree}
        \AxiomC{$\Gamma \vdash M : A \otimes B$}
        \AxiomC{$\Delta, x : A, y : B \vdash N : C$}
        \RightLabel{$\otimes{\bf E}$}
        \BinaryInfC{$\Delta, \Gamma \vdash {\bf let} \: x \otimes y \: = \:M \: {\bf in} \: N : C$}
      \end{prooftree}
    \end{small}
    \end{tabular}
  \end{minipage}
  \vspace{\baselineskip}

  \begin{minipage}{0.7\textwidth}
    \begin{centering}
    \begin{small}
    \begin{prooftree}
      \AxiomC{$\Gamma_1 \vdash M_1 : \bang_{s_{1}} A_1, \ldots, \Gamma_n \vdash M_n : \bang_{s_{n}} A_n$}
      \AxiomC{$x_1 : \bang_{s_{1}} A_1, \ldots, x_n : \bang_{s_{n}} A_n \vdash N : A$}
      \RightLabel{${\bang_s}_I$}
      \BinaryInfC{$\Gamma_1, \ldots, \Gamma_n \vdash \bang_s \: N \: {\bf with} \: \vec{x} = \vec{M}: \bang_s A$}
    \end{prooftree}
    \end{small}
    \end{centering}
    where $s \preceq s_{1}, \ldots, s_{n}$.
  \end{minipage}
  \hfill
  \begin{minipage}{0.2\textwidth}
    \begin{centering}
    \begin{small}
    \begin{prooftree}
      \AxiomC{$\Gamma \vdash M : \bang_s A$}
      \RightLabel{${\bang_s}_E$}
      \UnaryInfC{$\Gamma \vdash {\bf der}_{s} \: (M) : A$}
    \end{prooftree}
    for each $s \in \Sigma$.
    \end{small}
  \end{centering}
  \end{minipage}
  \vspace{\baselineskip}

  \begin{minipage}{0.45\textwidth}
    \begin{small}
      \begin{prooftree}
        \AxiomC{$\Gamma \vdash M : \bang_{s} A$}
        \AxiomC{$\Delta \vdash N : B$}
        \RightLabel{${\bf W}$}
        \BinaryInfC{$\Gamma, \Delta \vdash {\bf del}_{s}(M; N) : B$}
      \end{prooftree}
      for $s \in W$.
    \end{small}
  \end{minipage}
  \hfill
  \begin{minipage}{0.45\textwidth}
    \begin{tabular}{p{\textwidth}}
    \begin{small}
      \begin{prooftree}
        \AxiomC{$\Gamma \vdash M : \bang_{s} A$}
        \AxiomC{$\Delta, x : \bang_{s} A, y : \bang_{s} A \vdash N : B$}
        \RightLabel{${\bf C}$}
        \BinaryInfC{$\Gamma, \Delta \vdash {\bf let}_{s} \: x, y = M \: {\bf in} \: N : B$}
      \end{prooftree}
    \end{small}
      for $s \in C$.
    \end{tabular}
  \end{minipage}
  \hrule
\end{figure}

As in other term assignments for constructive and substructural logic, we have to detect the patterns of proof simplification given as
term rewriting to equip the typing inference rules with rules allowing one to reduce the depth of an inference tree.

\begin{definition}~\label{conversion:def}
  Let $M, N$ be terms, $\Gamma$ a context and $A$ a type, a \emph{conversion in context} is a judgement of the form
  $\Gamma \vdash M \triangleleft N : A$ that we read as $M$ converts to $N$ in a context $\Gamma$ and both $M$ and $N$
  have type $A$. The conversion rules are specified in Appendix~\ref{appendix}. $\trianglelefteq$ (i.e. the \emph{multistep conversion relation}) is the least preorder relation that contains $\triangleleft$.
\end{definition}

We have the following facts about proof conversions and, in particular, $\beta$-reduction.

\begin{theorem}[Normalisation and Confluence]~\label{proof:normalisation}
  \begin{enumerate}
  \item Let $\Gamma \vdash M : A$ be a typing judgement provable in ${\bf SILL}(\lambda)_{\Sigma}$ and $\Gamma \vdash M \trianglelefteq N : A$,
  then the ${\bf SILL}(\lambda)_{\Sigma}$-derivation of $\Gamma \vdash M : A$ can be transformed into the ${\bf SILL}(\lambda)_{\Sigma}$-derivation of
  $\Gamma \vdash N : A$.
  \item The $\beta$-reduction relation $\triangleleft_{\beta}$, that is, $\triangleleft$ restricted to the $\beta$-reduction rules, is strongly normalisable.
  \item Let $M, N, P$ be terms such that $M \trianglelefteq_{\beta} N$ and $M \trianglelefteq_{\beta} P$,
  then there is a term $P'$ such that $N \trianglelefteq_{\beta} P'$ and $P \trianglelefteq_{\beta} P'$,
  where $\trianglelefteq_{\beta}$ is the reflexive transitive closure of $\triangleleft_{\beta}$.
  As a corollary, a normal form is unique up to $\alpha$-equivalence.
  \end{enumerate}
\end{theorem}

The proof normalisation part is proved by induction on the generation of $M \triangleleft N$.
The strong normalisation and confluence parts are shown by the Tait-Girard technique 
(in particular, we refer to the proof based on candidates for the linear term calculus as in Bierman's PhD thesis in \cite[§3.5]{UCAM-CL-TR-346}), so we omit the detailed proof.

\section{Category-theoretic Preliminaries}~\label{preliminaries}

In this section, we outline the required concepts from 2-category theory and refine the formal theory of comonads.

We assume that the reader is familiar with the underlying notions of category theory such as (1-)categories, functors, natural transformations and adjunctions. In further sections, we will obtain some results about categories of generally large categories and such objects
cannot be defined in the standard set-theoretic foundations of mathematics such as ZFC or NBG, so we accept the foundations based
on Grothendieck universes as they are described in \cite[§1.2.15]{lurie2009higher} or \cite[§8]{shulman2008set}, but we do not bother ourselves with a deeper elaboration 
on this matter.

We assume that such concepts as 2-category, (strict) 2-functors, strict natural transformations, modifications are known to the reader.
Otherwise, we refer to \cite[Chapter 2]{johnson20212} and \cite{lack20092} for a more thorough and systematic introduction.

We fix ${\bf SMC}$ as the 2-category of symmetric monoidal categories, symmetric lax monoidal functors
and symmetric lax monoidal natural transformations. ${\bf SMC}_{\operatorname{str}}$ is a 2-subcategory of ${\bf SMC}$ where 1-cells are strict symmetric monoidal functors.
${\bf SMCC}_{str}$ is the full 1,2-subcategory of ${\bf SMC}_{str}$ restricted to those SMCs which are closed.

Let us recall how one treats adjunctions in a 2-categorical setting, the reader can find more details in \cite[§2]{kelly2006review}
and also \cite[Section 3]{kaledin2020adjunction} and \cite{riehl2016homotopy} for more advanced aspects of 2-categorical 
adjunctions.
\begin{definition}~\label{2-cat:adj}
Let $\mathcal{C}$ be a 2-category and let $A, B \in \operatorname{Ob}(\mathcal{C})$ be 0-cells, an \emph{adjunction} in $\mathcal{C}$, denoted as $\eta, \varepsilon: f^{\natural} \dashv f^{\sharp} : A \to B$ 
(or just $f^{\natural} \dashv f^{\sharp}$, when it is clear from the context what other components of an adjunction are),
consists of a pair of 1-cells $f^{\sharp} : A \to B$ and $f^{\natural} : B \to A$ and a pair of 2-cells $\eta : 1_B \Rightarrow f^{\sharp} \circ f^{\natural}$
and $\varepsilon : f^{\natural} \circ f^{\sharp} \Rightarrow 1_A$ such that the following conditions are satisfied:
\[
\begin{tikzcd}
  A \ar[to=B1, "f^{\sharp}"'] \ar[rr, "1_{A}", ""{name=ida}] && |[alias=A]|A \ar[to=B2, "f^{\sharp}"] \\
                                                                      &&&& = & A \ar[rr, bend left, "f^{\sharp}", ""{name=U,inner sep=1pt,below}] \ar[rr, bend right, "f^{\sharp}"{below}, ""{name=D,inner sep=1pt}] && B \ar[Rightarrow, from=U, to=D, "1_{f^{\sharp}}"]\\
  & |[alias=B1]| B \ar[rr, "1_{B}"', ""{name=idb}] \ar[to=A, "f^{\natural}"'] && |[alias=B2]| B 
  \arrow[Rightarrow, from=B1, to=ida, shorten >=2em,shorten <=2em, "\varepsilon"]
  \arrow[Rightarrow, from=idb, to=A, shorten >=2em,shorten <=2em, "\eta"']
\end{tikzcd}
\]
\[
\begin{tikzcd}
  B \ar[to=A1, "f^{\natural}"']  \ar[rr, "1_B", ""{name=idb}] && |[alias=B]| B \ar[to=A2, "f^{\natural}"]\\
                                                                      &&&& = & B \ar[rr, bend left, "f^{\natural}", ""{name=U,inner sep=1pt,below}] \ar[rr, bend right, "f^{\natural}"{below}, ""{name=D,inner sep=1pt}] && A \ar[Rightarrow, from=U, to=D, "1_{f^{\natural}}"] \\
  & |[alias=A1]| A \ar[to=B, "f^{\sharp}"'] \ar[rr, "1_A"', ""{name=ida}] && |[alias=A2]| A
  \arrow[Rightarrow, from=idb, to=A1, shorten >=2em,shorten <=2em, "\eta"]
  \arrow[Rightarrow, from=B, to=ida, shorten >=2em,shorten <=2em, "\varepsilon"]
\end{tikzcd}
\]

We say that $f^{\natural}$ is \emph{left adjoint} to $f^{\sharp}$ and $f^{\sharp}$ is \emph{right adjoint} to $f^{\natural}$.
\end{definition}

In particular, a \emph{symmetric monoidal adjunction} can be defined as an adjunction in the 2-category ${\bf SMC}$.

\begin{definition}~\label{left:adj:def}
  Let $\mathcal{C}$ be a 2-category, let $A,B$ be 0-cells and let $\eta, \varepsilon: f^{\natural} \dashv f^{\sharp} : A \to B$ and
  $\eta', \varepsilon': {f'}^{\natural}\dashv {f'}^{\sharp} : A' \to B'$ be adjunctions in $\mathcal{C}$, then a \emph{left morphism of adjunctions} consists of
  a pair of 1-cells $g_1 : A \to A'$ and $g_2 : B \to B'$ such that $g_1 \circ f^{\natural} = {f'}^{\natural} \circ g_2$.
\end{definition}

Note that there are different non-equivalent definitions of an adjunction morphism. For example,
the definition of an adjunction morphism given in \cite[Exercise 4.2.v]{riehl2017category} extends our definition of left 
adjunction morphism with a similar extra equality of composites with right adjoints. 
The definition of an adjunction morphism from \cite[Definition 2.2.2]{zaganidis2017classifiers} contains only the equality
with right adjoints. So we coined the term ``left adjunction morphism'' to avoid further terminological confusion.

Let $\mathcal{C}, \mathcal{D}$ be 2-categories, then, by default, ${\bf Fun}(\mathcal{C}, \mathcal{D})$ means
the 2-category of strict 2-functors, their strict natural transformations and modifications.

Let $[1]$ denote a two-element category with two objects $0$ and $1$ and a single non-identity morphism $0 \to 1$.
\begin{definition}
Given a 2-category $\mathcal{C}$, a 2-category ${\bf Adj}^{l}(\mathcal{C})$
of adjunctions in $\mathcal{C}$ and their left morphisms is defined as 
the full 1,2-subcategory of ${\bf Fun}([1], \mathcal{C})$ spanned by 1-cells in $\mathcal{C}$ having a right adjoint.
\end{definition}

Comonads in 2-categories are defined by the dualisation of the corresponding definition of a monad in 2-categories from \cite{street1972formal}.
The reader can find the explicit presentation of the formal theory of comonads in \cite[§2.3]{zwanziger2024natural}.
\begin{definition}~\label{2-comonad}
  Let $\mathcal{C}$ be a 2-category, a \emph{comonad} in $\mathcal{C}$ consists of a 0-cell $A \in \operatorname{Ob}(\mathcal{C})$ along with
  a 1-cell $k : A \to A$ and 2-cells $\delta : k \Rightarrow k \circ k$ (\emph{comultiplication}) and $\varepsilon : k \Rightarrow 1_A$
  (\emph{counit}) such that the following diagrams commute:
  \[
  \begin{tikzcd}
    & k \ar[ddl, "1_k"'] \ar[ddr, "1_k"] \ar[dd, "\delta"] && k \ar[rr, "\delta"] \ar[dd, "\delta"'] && k k \ar[dd, "k \delta"] \\
    \\
  k & k k \ar[l, "k\varepsilon"] \ar[r, "\varepsilon k"'] & k      &  k k \ar[rr, "\delta k"'] && k k  k 
  \end{tikzcd}
  \]
\end{definition}

Most of the time, we will identify a comonad $(k, \delta, \varepsilon)$ over $A \in \operatorname{Ob}(\mathcal{C})$
with the underlying 1-cell $k$. A usual (co)monad over a category $\mathcal{C}$ is a (co)monad in the 2-category ${\bf Cat}$
of categories, functors and natural transformations. It is well known
that every adjunction in a 2-category gives a comonad. Indeed, let
$\eta, \varepsilon: f^{\natural} \dashv f^{\sharp} : A \to B$ be an adjunction in $\mathcal{C}$,
then the composite $f^{\natural} \circ f^{\sharp} : A \to A$ forms a comonad on $A$
with the counit 2-cell $\varepsilon : f^{\natural} \circ f^{\sharp} \Rightarrow 1_A$
and the comultiplication 2-cell $f^{\natural} \eta f^{\sharp} : f^{\natural} \circ f^{\sharp} \Rightarrow f^{\natural} \circ f^{\sharp} \circ f^{\natural} \circ f^{\sharp}$.

Let $\mathcal{C}$ be a 2-category, then one can describe the 2-category of comonads in $\mathcal{C}$, that is, comonads as 0-cells, comonad morphisms as 1-cells
and comonad morphism transformations as 2-cells.
The reader can find more details describing the category of monads in \cite[§8.3]{lack20092} or \cite{street1972formal} 
and apply them to comonads by duality. 

\begin{definition}~\label{comonad:morph}
\begin{itemize}
\item Let $(A, k)$ and $(B, l)$ be comonads in a 2-category $\mathcal{C}$, a \emph{comonad morphism}
$(f, \mu) : (A, k) \to (B, l)$ consists of a 1-cell $f : A \to B$ and a 2-cell $\mu : f k \Rightarrow l f$
such that the following conditions are satisfied:
\[
\begin{tikzcd}
  f k \ar[rr, "\mu"] \ar[d, "f \delta"'] && l  f \ar[d, "\delta f"]                & f k \ar[rr, "\mu"] \ar[dr, "f \varepsilon"'] && l f \ar[dl, "\varepsilon  f"] \\
  f k  k \ar[r, "\mu k"'] & l f k \ar[r, "l \mu"'] & l l f & & f 
\end{tikzcd}
\]

\item Let $(A, k_1)$ and $(B, k_2)$ be comonads in a 2-category $\mathcal{C}$
  and let $(f_1, \mu_1), (f_2, \mu_2) : (A, k_1) \to (B, k_2)$ be comonad morphisms, 
  a \emph{comonad morphism transformation} is a 2-cell $\varphi : f_1 \Rightarrow f_2$ such that the following square commutes:
  \[
  \begin{tikzcd}
    f_1 k_1 \ar[d, "\mu_1"'] \ar[rr, "\varphi k_1"] && f_2 k_1 \ar[d, "\mu_2"] \\
    k_2 f_1 \ar[rr, "k_2 \varphi"'] && k_2 f_2
  \end{tikzcd}
  \]
\end{itemize}

The 2-category ${\bf coMnd}(\mathcal{C})$ consists of all comonads in a 2-category $\mathcal{C}$,
their morphisms and morphism transformations.
\end{definition}

Turning to the example with ${\bf SMC}$, a \emph{symmetric lax monoidal comonad morphism} of symmetric lax monoidal comonads is a comonad morphism in ${\bf coMnd}({\bf SMC})$. One could also
define this concept explicitly as a comonad morphism respecting the structures of a symmetric lax monoidal functor.

Generally, we characterise ${\bf coMnd}(\mathcal{C})$
with the 2-category ${\bf Fun}_{oplax}(W_{cmd}, \mathcal{C})$ of strict 2-functors, \emph{oplax} natural transformations and their modifications,
where $W_{cmd}$ (aka the ``walking comonad'') is a 2-category with a single object $\bullet$. 
1-cells are generated by the arrow $k : \bullet \to \bullet$, 2-cells are generated by
$\varepsilon : k \Rightarrow {\bf id}_{\bullet}$ and $\delta : k \Rightarrow k \circ k$
along with the axioms from Definition~\ref{2-comonad}.

The next step is to an alternative characterisation of ${\bf coMnd}(\mathcal{C})$ with the concept of a left adjunction morphism. For that, let us recall the 2-categorical notion of a mate.
Let $\eta, \varepsilon: f^{\natural} \dashv f^{\sharp} : A \to B$ and $\eta', \varepsilon': {f'}^{\natural} \dashv {f'}^{\sharp} : A' \to B'$ be adjunctions in $\mathcal{C}$.
Let $g_1 : A \to A'$ and $g_2 : B \to B'$, then there is a bijection between 2-cells 
${f'}^{\natural} \circ g_2 \Rightarrow g_1 \circ f^{\natural}$ and $g_2 \circ f^{\sharp} \Rightarrow {f'}^{\sharp} \circ g_1$, so squares of the form
\[
\begin{tikzcd}
  A \ar[rr, "f^{\sharp}", ""{name=fnat}] \ar[dd, "g_1"', ""{name=g1}] && |[alias=B]| B \ar[dd, "g_2", ""{name=g2}] \\
  \\
  |[alias=A']| A' \ar[rr, "{f'}^{\sharp}"',  ""{name=fnat'}] && B'
  \arrow[from=fnat, to=fnat', Rightarrow, shorten=6mm]
\end{tikzcd}
\]
correspond one-to-one to squares of the form
\[
\begin{tikzcd}
  A \ar[dd, "g_1"', ""{name=g1}] && |[alias=B]| B \ar[dd, "g_2", ""{name=g2}] \ar[ll, "f^{\natural}"', ""{name=fsharp}] \\
  \\
  |[alias=A']| A' && B' \ar[ll, "{f'}^{\natural}", ""{name=fsharp'}]
  \arrow[from=fsharp, to=fsharp', Rightarrow, shorten=6mm]
\end{tikzcd}
\]
So those pairs of 2-cells are called \emph{mates}, see \cite[§2]{lack20092} and \cite[§2]{kelly2006review}.

\begin{prop}~\label{2:comonad:morphish:prop} Let $\mathcal{C}$ be a 2-category, and let $\eta, \varepsilon: f^{\natural} \dashv f^{\sharp} : A \to B$ and
  $\eta', \varepsilon': {f'}^{\natural} \dashv {f'}^{\sharp} : A' \to B'$ be adjunctions in $\mathcal{C}$
  and let $g_1 : A \to A'$ and $g_2 : B \to B'$ be 1-cells such that a pair $(g_1, g_2)$ is a left morphism
  of the corresponding adjunctions. Then $(g_1, g_2)$ induces a morphism of comonads $f^{\natural} \circ f^{\sharp} \Rightarrow {f'}^{\natural} \circ {f'}^{\sharp}$.
\end{prop}

\begin{proof}
  The proof is similar to \cite[Proposition 2.2.4]{zaganidis2017classifiers},
  but let us just demonstrate how exactly a comonad morphism is given.
  Let $\beta : g_2 f^{\sharp} \Rightarrow {f'}^{\sharp} g_1$ be the mate of the identity $g_1 f^{\natural} = {f'}^{\natural}  g_2$, that is,
  the following composite:
  \[
  \begin{tikzcd}
  A \ar[rr, "f^{\sharp}"] \ar[ddrr, "1_A"', ""{name=id1}]  && |[alias=B]| B \ar[rr, "g_2"] \ar[dd, "f^{\natural}"] && |[alias=B']| B' \ar[dd, "{f'}^{\natural}"] \ar[ddrr, "1_{B'}", ""{name=id2}] \\
  \\
  && |[alias=A]| A \ar[rr, "g_1"'] && |[alias=A']| A' \ar[rr, "{f'}^{\sharp}"'] && B'
  \arrow[from=B, to=id1, Rightarrow, shorten=2mm, "\varepsilon"]
  \arrow[from=id2, to=A', Rightarrow, shorten=2mm, "\eta'"]
  \arrow[from=A, to=B', equal, shorten=1mm]
  \end{tikzcd}
  \]

  Then a pair $(g_1, {f'}^{\natural} \beta)$ is a comonad morphism $(A, f^{\natural} f^{\sharp}) \to (A', {f'}^{\natural} {f'}^{\sharp})$.
\end{proof}

\begin{prop}~\label{2:cell:adj:comonad}
  Let $\eta, \varepsilon: f^{\natural} \dashv f^{\sharp}$ and $\eta', \varepsilon': {f'}^{\natural} \dashv {f'}^{\sharp}$ 
  be adjunctions between $A$ and $B$ and $A'$ and $B'$ respectively in a 2-category $\mathcal{C}$. Consider the induced 
  comonads $f^{\natural} f^{\sharp}$ and $f'^{\natural} f'^{\sharp}$ on $A$ and $A'$ respectively. Let $(g_1, g_2)$ and $(g'_1, g'_2)$
  be left adjunction morphisms from $f^{\natural} \dashv f^{\sharp}$ to ${f'}^{\natural}\dashv f'^{\sharp}$ 
  inducing comonad morphisms $(g_1, \mu) : f^{\natural} f^{\sharp} \Rightarrow {f'}^{\natural} {f'}^{\sharp}$ and $(g'_1, \mu') : f^{\natural} f^{\sharp} \Rightarrow {f'}^{\natural} {f'}^{\sharp}$ respectively and 
  let $\alpha_1 : g_1 \Rightarrow g'_1$ and $\alpha_2 : g_2 \Rightarrow g'_2$ be 2-cells in $\mathcal{C}$ such that
  $\alpha_1 f^{\natural} = {f'}^{\natural} \alpha_2$. Then $\alpha_1$ and $\alpha_2$ induce a comonad morphism transformation
  \[
  \begin{tikzcd}
    f^{\natural} f^{\sharp} \ar[rrr, bend left=15, "{(g_1, \mu)}", ""{name=mu1}] \ar[rrr, bend right=15, "{(g'_1, \mu')}"', ""{name=mu2}] &&& {f'}^{\natural} {f'}^{\sharp}
    \arrow[from=mu1, to=mu2, Rightarrow]
  \end{tikzcd}
  \]
\end{prop}

\begin{proof}
  Let us show that the following square of 2-cells commutes:
  \begin{equation}~\label{goal:adj}
  \begin{tikzcd}
    g_1 f^{\natural} f^{\sharp} \ar[rr, "\alpha_1 f^{\natural} f^{\sharp}"] \ar[d, "\mu"'] && g'_1 f^{\natural} f^{\sharp} \ar[d, "\mu'"]\\
    {f'}^{\natural} {f'}^{\sharp} g_1 \ar[rr, "{f'}^{\natural} {f'}^{\sharp} \alpha_1"'] && {f'}^{\natural} {f'}^{\sharp} g'_1
  \end{tikzcd}
  \end{equation}
  By Proposition~\ref{2:comonad:morphish:prop}, we can unveil $\mu$ and $\mu'$ as
  ${f'}^{\natural} \beta : {f'}^{\natural} g_2 f^{\sharp} \Rightarrow {f'}^{\natural} {f'}^{\sharp} g_1$ and  
  ${f'}^{\natural} \beta' : {f'}^{\natural} g'_2 f^{\sharp} \Rightarrow {f'}^{\natural} {f'}^{\sharp} g'_1$. So let us rewrite~(\ref{goal:adj})
  as
  \begin{equation}~\label{goal:adj:2}
    \begin{tikzcd}
    g_1 f^{\natural} f^{\sharp} \ar[rrr, "\alpha_1 f^{\natural} f^{\sharp}"] \ar[d, equal]  &&& g'_1 f^{\natural} f^{\sharp} \ar[d, equal] \\
    {f'}^{\natural} g_2 f^{\sharp} \ar[d, "{f'}^{\natural} \beta"'] \ar[rrr, "{f'}^{\natural} \alpha_2 f^{\sharp}"] &&& {f'}^{\natural} g'_2 f^{\sharp} \ar[d, "{f'}^{\natural} \beta'"] \\
    {f'}^{\natural} {f'}^{\sharp} g_1 \ar[rrr, "{f'}^{\natural} {f'}^{\sharp} \alpha_1"'] &&& {f'}^{\natural} {f'}^{\sharp} g'_1
    \end{tikzcd}
  \end{equation}
  Observe that the top square commutes by $\alpha_1 f^{\natural} = {f'}^{\natural} \alpha_2$. The bottom one is unveiled as follows and it holds 
  by applying the assumption of the proposition as well:
  \[
  \begin{tikzcd}
    |[alias=A1]| A & |[alias=B1]| B & |[alias=B1']| B'     &&&   |[alias=A2]| A & |[alias=B2]| B & |[alias=B2']| B' \\
    & |[alias=A3]| A & |[alias=A'1]| A' & |[alias=B3']| B' & = & |[alias=A4]| A & |[alias=B3]| B & |[alias=B4']| B' \\
    & |[alias=A5]| A & |[alias=A'2]| A' & |[alias=B5']| B' &&& |[alias=A6]| A & |[alias=A'3]| A' & |[alias=B6']| B'
    \ar[from=A3, to=A5, equal]
    \ar[from=A2, to=A4, equal]
    \ar[from=A'1, to=A'2, equal]
    \ar[from=B3', to=B5', equal]
    \ar[from=B2, to=B3, equal]
    \ar[from=B2', to=B4', equal]
    \ar[from=A3, to=A'1, "g_1", ""{name=g1}]
    \ar[from=A5, to=A'2, "g'_1"', ""{name=g1'}]
    \ar[from=g1, to=g1', Rightarrow, shorten=3mm, "\alpha_1"]
    \ar[from=A'1, to=B3', "{f'}^{\sharp}"']
    \ar[from=A'2, to=B5', "{f'}^{\sharp}"']
    \ar[from=A1, to=A3, "1_A"', ""{name=1A}]
    \ar[from=B1, to=1A, Rightarrow, "\varepsilon"]
    \ar[from=B1, to=A3, "f^{\natural}"]
    \ar[from=B1', to=A'1, "{f'}^{\natural}"']
    \ar[from=A1, to=B1, "f^{\sharp}"]
    \ar[from=B1, to=B1', "g_2"]
    \ar[from=B1', to=B3', "1_{B'}", ""{name=1B'}]
    \ar[from=1B', to=A'1, Rightarrow, "\eta'"]
    \ar[from=A2, to=B2, "f^{\sharp}"]
    \ar[from=A4, to=B3, "f^{\sharp}"]
    \ar[from=B2, to=B2', "g_2", ""{name=g2}]
    \ar[from=B3, to=B4', "g_2'", ""{name=g2'}]
    \ar[from=g2, to=g2', Rightarrow, shorten=3mm, "\alpha_2"]
    \ar[from=A4, to=A6, "1_A"', ""{name=1A'}]
    \ar[from=A4, to=B3, "f^{\sharp}"]
    \ar[from=B3,to=A6, "f^{\natural}"]
    \ar[from=B3, to=1A', Rightarrow, "\varepsilon"]
    \ar[from=A6, to=A'3, "g'_1"]
    \ar[from=B4',to=A'3, "{f'}^{\natural}"']
    \ar[from=A'3,to=B6', "{f'}^{\sharp}"']
    \ar[from=B4', to=B6', "1_{B'}", ""{name=1B''}]
    \ar[from=1B'', to=A'3, Rightarrow, "\eta'"]
  \end{tikzcd}
  \]
\end{proof}

So we have a 2-functor from ${\bf Adj}^l(\mathcal{C})$ to ${\bf coMnd}(\mathcal{C})$ by 
Proposition~\ref{2:comonad:morphish:prop} and Proposition~\ref{2:cell:adj:comonad}.
But there is the other way round 2-functor as well, but for 2-categories admitting the construction of coalgebras,
so we need several more definitions.

We assume that the classical notion of a coalgebra over a comonad is already known to the reader,
but we refer to \cite[Chapter 3]{barr2000toposes} to have an introduction to algebras over a monad, which is applicable
to comonads and coalgebras by duality.

\begin{definition}
  Let $\mathcal{C}$ be a 2-category and let $(A, k)$ be a comonad in $\mathcal{C}$. An \emph{object of coalgebras} (also 
  called a \emph{co-Eilenberg-Moore object}) is a 0-cell $A^k \in \Ob{C}$ (if it exists) such that for each
  0-cell $B$:
  \begin{equation}~\label{coalg:iso}
    \mathcal{C}(B, A^k) \cong {\bf coMnd}(\mathcal{C})((B, 1_B), (A, k)).
  \end{equation}

  $\mathcal{C}$ \emph{admits the construction of coalgebras} if every comonad has an object of coalgebras.
\end{definition}

Fix a comonad $(A, k)$ in a 2-category $\mathcal{C}$ with an object of coalgebras $A^k$. We have a comonad morphism
$(k, \delta) : (A, 1_A) \to (A, k)$, so we obtain \emph{the cofree morphism} $v : A \to A^k$, a 1-cell in $\mathcal{C}$,
by transposing $(k, \delta)$ through~(\ref{coalg:iso}). \emph{The coforgetful morphism} $u : A^k \to A$ is obtained by transposing the identity
morphism $1_{A^k} : A^k \to A^k$. Thus we have an adjunction $u \dashv v : A \to A^k$ in $\mathcal{C}$ which is called
the \emph{cofree-coforgetful decomposition}. Thus if $\mathcal{C}$ admits the construction of coalgebras,
then every comonad has cofree-coforgetful decomposition. The important observation is that ${\bf SMC}$ is such a 2-category.

Therefore, we have (the complete proof can be adapted from \cite[Theorem 2]{street1972formal}):

\begin{theorem}~\label{coalgebra:theorem} 
  Let $\mathcal{C}$ be a 2-category that admits the construction of coalgebras, then
${\bf coMnd}(\mathcal{C})$ fully faithfully embeds into ${\bf Adj}^l(\mathcal{C})$.
\end{theorem}

Let us instantiate the above theorem for the 2-category ${\bf SMC}$ explicitly:
\begin{prop}~\label{coalgebra:theorem:pr}
  ${\bf coMnd}({\bf SMC})$, the 2-category of symmetric lax monoidal comonads over SMCs, their symmetric lax monoidal morphisms and transformations,
  fully faithfully embeds into ${\bf Adj}^{l}({\bf SMC})$, the 2-category of symmetric monoidal adjunctions in ${\bf SMC}$,
  their left morphisms and modifications.
\end{prop}

\subsection{On Cosemigroups, Comonoids and related concepts}

We also recall underlying definitions and facts related to symmetric monoidal categories and comonads over them. 
We assume that the reader is already familiar with the concepts of a symmetric (closed) monoidal category,
(lax, strong, strict) symmetric monoidal functors. Otherwise, we refer to \cite[Chapter 2]{etingof2015tensor} and \cite[Part I]{aguiar2010monoidal}.

Let us recall explicitly the concepts of comonoid and cosemigroup objects to keep the paper more self-contained as we will be addressing
them quite often. Cosemigroups are less studied, but they are of some interest in the study of Banach algebras, see \cite{poinsot2022multipliers}. This
definition will be of particular interest for studying relevant subexponentials, that is, those ones that admit at least the contraction rule.

\begin{definition} Let $(\mathcal{C}, \otimes, \mathds{1})$ be a monoidal category, then
  \begin{itemize}
  \item A \emph{cosemigroup object} in $\mathcal{C}$ is an object
  $A \in \Ob{C}$ with a choice of a morphism $\gamma_A : A \to A \otimes A$ called \emph{copying} such that the following is satisfied:
  \[
  \begin{tikzcd}
  A \ar[rrr, "\gamma_A"] \ar[d, "\gamma_A"']                                  &&& A \otimes A \ar[d, "1_A \otimes \gamma_A"] \\
  A \otimes A \ar[rrr, "\alpha_{A,A,A} \circ (\gamma_A \otimes 1_A)"'] &&& A \otimes A \otimes A
  \end{tikzcd}
  \]

\item A \emph{morphism of cosemigroups} $f : (A, \gamma_A) \to (B, \gamma_B)$ is an arrow $f : A \to B$ in $\mathcal{C}$ making the below square commute:
\[
\begin{tikzcd}
A \ar[r, "f"] \ar[d, "\gamma_A"'] & B \ar[d, "\gamma_B"] \\
A \otimes A \ar[r, "f \otimes f"'] & B \otimes B
\end{tikzcd}
\]

\item Let $(\mathcal{C}, \otimes, \mathds{1})$ be an SMC, then a cosemigroup $(A, \gamma_A)$ is \emph{cocommutative} if
the extra axiom is satisfied:
\[
\begin{tikzcd}
  & A \ar[dl, "\gamma_A"'] \ar[dr, "\gamma_A"] \\
  A \otimes A  \ar[rr, "\sigma_{A,A}"'] && A \otimes A 
\end{tikzcd}
\]
${\bf coSem}(\mathcal{C})$ is the category of cocommutative cosemigroups in $\mathcal{C}$ and their morphisms.
\end{itemize}
\end{definition}

\begin{definition}
A \emph{comonoid} is a cosemigroup object $(A, \gamma)$ along with a choice of a morphism $\iota_A : A \to \mathds{1}$ (\emph{counit})
commuting with formal diagonals in the following way:
\[
\begin{tikzcd}
  & A \ar[d, "\gamma_A"] \ar[dl, "\lambda_A^{-1}"'] \ar[dr, , "\rho_A^{-1}"] \\
\mathds{1} \otimes A & A \otimes A \ar[l, "\iota_A \otimes 1_A"] \ar[r, "1_A \otimes \iota_A"'] & A \otimes \mathds{1}
\end{tikzcd}
\]

A \emph{comonoid morphism} is a cosemigroup morphism $f : (A, \gamma_A) \to (B, \gamma_B)$ preserving the counit:
\[
\begin{tikzcd}
  A \ar[rr, "f"] \ar[dr, "\iota_A"'] && B \ar[dl, "\iota_B"] \\
  & \mathds{1}
\end{tikzcd}
\]

A comonoid is \emph{cocommutative} if its copying operation is commutative.
${\bf coMon}(\mathcal{C})$ denotes the category of cocommutative comonoids and their morphisms.
\end{definition}

The following statement is folklore:
\begin{theorem}~\label{cosemi:monoidal}
  Let $(\mathcal{C}, \otimes, \mathds{1})$ be a symmetric monoidal category, then
  the categories of cocommutative cosemigroups and comonoids ${\bf coSem}(\mathcal{C})$ and ${\bf coMon}(\mathcal{C})$
  are symmetric monoidal with the following structure:
  \begin{itemize}
    \item The tensor product of cosemigroups $(A, \gamma_A)$ and $(B, \gamma_B)$ is the cosemigroup $(A \otimes B, \gamma_{A \otimes B})$
    where $\gamma_{A \otimes B} : A \otimes B \to (A \otimes B) \otimes (A \otimes B)$ is given by the composite:
    \[
    \begin{tikzcd}
    A \otimes B \ar[rr, "\gamma_{A \otimes B}"] \ar[d, "\gamma_A \otimes \gamma_B"'] && (A \otimes B) \otimes (A \otimes B) \\
    A \otimes A  \otimes B \otimes B \ar[d, "\alpha_{A, A, B \otimes B}"'] && A \otimes (B \otimes (A \otimes B)) \ar[u, "\alpha^{-1}_{A, B, A \otimes B}"'] \\
    A \otimes (A \otimes B \otimes B) \ar[dr, "1_A \otimes \alpha^{-1}_{A, B, B}"'] && A \otimes ((B \otimes A) \otimes B) \ar[u, "1_A \otimes \alpha_{B, A, B}"'] \\
    & A \otimes ((A \otimes B) \otimes B) \ar[ur, "1_A \otimes (\sigma_{A, B} \otimes 1_B)"']
    \end{tikzcd}
    \]

    In ${\bf coMon}(\mathcal{C})$, the tensor product of comonoids $(A, \gamma_A, \iota_A)$ and $(B, \gamma_B, \iota_B)$ is
    the comonoid $(A \otimes B, \gamma_{A \otimes B}, \iota_{A \otimes B})$ where the counit $\iota_{A \otimes B}$
    is given by
    \[
    \begin{tikzcd}
    A \otimes B \ar[r, "\iota_A \otimes \iota_B"] & \mathds{1} \otimes \mathds{1} \ar[r, "\lambda = \rho"] & \mathds{1}.
    \end{tikzcd}
    \]
    where $\lambda$ and $\rho$ are unitors.
    \item The monoidal unit in ${\bf coSem}(\mathcal{C})$ is the cosemigroup
    $(\mathds{1}, \gamma_{\mathds{1}})$ where $\gamma_{\mathds{1}}$ is just $\rho^{-1}$ (the inverse of the unitor $\rho : \mathds{1} \otimes A \to A$). 
    In the category of comonoids,
    the cosemigroup $(\mathds{1}, \gamma_{\mathds{1}})$ is equipped with the identity map 
    $1_{\mathds{1}} : \mathds{1} \to \mathds{1}$ as the counit morphism.
  \end{itemize}
  Moreover, the forgetful functors ${\bf coSem}(\mathcal{C}) \to \mathcal{C}$
  and ${\bf coMon}(\mathcal{C}) \to \mathcal{C}$ are strict monoidal and symmetric.
\end{theorem}
The notions of a monoid and a semigroup are obtained by the dualisation of the above definitions,
see \cite{porst2008categories} for a more thorough introduction to categorical (co)monoids.
\section{Assemblages and Semantic Interpretation of ${\bf SILL}(\lambda)_{\Sigma}$}~\label{cocteau:categories:sem}

In this section, we construct the categorical denotational semantics for the type theory we introduced earlier.
We introduce $\Sigma$-assemblages to have an adequate semantic interpretation of ${\bf SILL}(\lambda)_{\Sigma}$. 

In categorical semantics of linear logic as it is studied in, for example, \cite{abramsky2011introduction} and \cite{mellies2009categorical},
one can consider so-called \emph{linear categories}, symmetric monoidal categories equipped with a symmetric lax
monoidal comonad that, informally speaking, assign a cocommutative comonoid to each object of the underlying category.
The heuristic is that we can associate the data type $\bang A$ with each $A$ such that elements of $\bang A$ can be 
both copied and destroyed. Such comonads allow interpreting the exponential modality semantically. 
We below suggest the notion of a $\Sigma$-assemblage expanding linear categories by equipping the underlying SMCC with a family of
symmetric lax monoidal comonads whose resource policies are described by what their indices belong to: $W$, $C$, both or neither of them.

\begin{definition}~\label{copy:del:def}
Let $\mathcal{C}$ be a symmetric monoidal category, then:
  \begin{itemize}
    \item An \emph{exponential} comonad is a symmetric lax monoidal comonad $(\mathsf{K}, \varepsilon, \delta)$ on $\mathcal{C}$ with a pair
    of symmetric monoidal natural transformations
    $\mathsf{c} : \mathsf{K} \Rightarrow  \mathsf{K}(.) \otimes \mathsf{K}(.)$ and $\mathsf{d} : \mathsf{K} \Rightarrow (X \mapsto \mathds{1})$ such that
    the following is satisfied:
    \begin{itemize}
      \item $(\mathsf{K} A, \mathsf{c}_A, \mathsf{d}_A)$ is a cocommutative comonoid for each $A \in \mathcal{C}$,
      \item $\delta_A$ is a comonoid morphism, i.e., the following axioms are satisfied:
      \[
      \begin{tikzcd}
        \mathsf{K} A \ar[rr, "\delta_A"] \ar[d, "\mathsf{c}_A"'] &&  \mathsf{K}^2 A \ar[d, "\mathsf{c}_{\mathsf{K} A}"']  & \mathsf{K} A \ar[r, "\delta_A"] \ar[dr, "\mathsf{d}_A"'] &  \mathsf{K}^2 A \ar[d, "\mathsf{d}_{\mathsf{K} A}"] \\
        \mathsf{K} A \otimes \mathsf{K} A \ar[rr, "\delta_A \otimes \delta_A "'] && \mathsf{K}^2 A \otimes \mathsf{K}^2 A && \mathds{1} 
      \end{tikzcd}
      \]
      \item $\mathsf{c}_A$ and $\mathsf{d}_A$ are coalgebra morphisms, i.e., the following axioms are satisfied:
      \[
        \begin{tikzcd}  
          \mathsf{K} A \ar[rr, "\delta_A"] \ar[d, "\mathsf{c}_A"'] && \mathsf{K}^2 A \ar[d, "\mathsf{K}\mathsf{c}_A"]                                                                                                     & \mathsf{K} A \ar[d, "\mathsf{d}_A"'] \ar[r, "\delta_A"] & \mathsf{K}^2 A \ar[d, "\mathsf{K}(\mathsf{d}_A)"] \\
          \mathsf{K} A \otimes \mathsf{K} A \ar[r, "\delta_A \otimes \delta_A"'] & \mathsf{K}^2 A \otimes \mathsf{K}^2 A \ar[r, "\mathsf{m}_{\mathsf{K}A, \mathsf{K}A}"'] & \mathsf{K}(\mathsf{K} A \otimes \mathsf{K} A) & \mathds{1} \ar[r, "\mathsf{m}_{\mathds{1}}"'] & \mathsf{K} \mathds{1}
        \end{tikzcd}
      \]
    \end{itemize}
  \item A \emph{relevant} comonad is a symmetric lax monoidal comonad $\mathsf{K}$ with a
  symmetric monoidal natural transformation $\mathsf{c} : \mathsf{K} \Rightarrow  \mathsf{K}(.) \otimes \mathsf{K}(.)$ 
  such that a pair $(\mathsf{K} A, \mathsf{c}_A)$ is a cocommutative cosemigroup,
  $\delta_A$ is a cosemigroup morphism and $\mathsf{c}_A$ is a coalgebra morphism.
  \item $((\mathsf{K}, \mathsf{m}), \varepsilon, \delta, \mathsf{d})$ is an \emph{affine} comonad, where $\mathds{d}$
  is a symmetric monoidal natural transformation with components
  $\mathsf{d}_A : \mathsf{K} A \to \mathds{1}$ such that the following square commutes:
  \[
  \begin{tikzcd}
    \mathsf{K} A \ar[r, "\delta_A"] \ar[d, "\mathsf{d}_A"'] & \mathsf{K}^2 A \ar[d, "\mathsf{K}(\mathsf{d}_A)"] \\
    \mathds{1} \ar[r, "\mathsf{m}_{\mathds{1}}"'] & \mathsf{K} A
  \end{tikzcd}
  \]
  \end{itemize}
\end{definition}

\begin{definition}
  Let $\Sigma = (I, \preceq, W, C)$ be a subexponential signature, 
  a \emph{$\Sigma$-assemblage} on an SMCC $\mathcal{C}$ is a structure $(\mathcal{C}, (\bang_s)_{s \in \Sigma}, (\mu_{s_1,s_2})_{s_1 \preceq s_2})$ such that:
  \begin{itemize}
    \item The mapping $s \mapsto \bang_s$ assigns a symmetric lax monoidal comonad on $\mathcal{C}$ such that
    if $s_1 \preceq s_2$, then there is a symmetric lax monoidal comonad morphism $\mu_{s_1, s_2} : \bang_{s_2} \Rightarrow \bang_{s_1}$ such that
    $\mu_{s, s} = 1$ and $\mu_{s_2, s_3} \circ \mu_{s_1, s_2} = \mu_{s_1, s_3}$,
    \item If $s \in W$, then $\bang_s$ is an affine comonad,
    \item If $s \in C$, then  $\bang_s$ is a relevant comonad,
    \item If $s \in W \cap C$, then $\bang_s$ is an exponential comonad,
    \item If $s_1 \in C$ and $s_1 \preceq s_2$, then for each $A \in \mathcal{C}$, ${\mu_{s_1, s_2}}_{A} : \bang_{s_2} A \to \bang_{s_1} A$
    induces a cocommutative cosemigroup morphism.
    \item If both $s_1 \in W \cap C$ and $s_1 \preceq s_2$, then ${\mu_{s_1, s_2}}_A : (\bang_{s_2} A, {\mathsf{c}_{s_2}}_A, {\mathsf{d}_{s_2}}_A) \to (\bang_{s_1} A, {\mathsf{c}_{s_1}}_A, {\mathsf{d}_{s_1}}_A)$
    induces a cocommutative comonoid morphism for each $A \in \mathcal{C}$.
    \item If $s_1 \in W$ and $s_1 \preceq s_2$, then the following triangle commutes for each $A \in \mathcal{C}$:
    \[
    \begin{tikzcd}
      \bang_{s_2} A \ar[rr, "{\mu_{s_1, s_2}}_{A}"] \ar[dr, "{\mathsf{d}_{s_2}}_A"'] && \bang_{s_1} A \ar[dl, "{\mathsf{d}_{s_1}}_A"] \\
      & \mathds{1}
    \end{tikzcd}
    \]
  \end{itemize}
\end{definition}

\begin{convention}
In order to distinguish similar components of different comonads in a $\Sigma$-assemblage, we will sometimes label
them with an upper index as follows: for example, $\varepsilon^{s}$ will stand for
counit in a comonad with an index $s$, etc.
\end{convention}

\begin{example}~\label{quantale:example} One can extract a rather natural example of a $\Sigma$-assemblage from quantales,
  see \cite{eklund2018semigroups} for a more general context.

  A \emph{quantale} is a structure $\mathcal{Q} = (Q, \cdot, 1, \bigvee)$ where $(Q, \bigvee)$ is a complete lattice
  and $(Q, \cdot, 1)$ is a monoid such that the multiplication operation preserves suprema in each coordinate. Quantales
  are commutative by default in this example.

  A \emph{quantic conucleus} on a quantale $\mathcal{Q}$ is a coclosure operator $g : \mathcal{Q} \to \mathcal{Q}$
  such that $g(a) \cdot g(b) \leq g(a \cdot b)$ and $g(1) = 1$ for each $a, b \in \mathcal{Q}$.
  It is a well-known fact that there is a bijection between subquantales of a quantale
  and quantic conuclei: if $g : \mathcal{Q} \to \mathcal{Q}$ is a quantic conucleus, then the set 
  $\mathcal{Q}_g = \{ a \in \mathcal{Q} \: | \: g(a) = a \}$ forms a subquantale of $\mathcal{Q}$, see, for example, \cite[Theorem 3.1.3]{570645}. 
  
  Let ${\bf SubQuant}(\mathcal{Q})$ denote the poset of all subquantales of a commutative quantale $\mathcal{Q}$ ordered by inclusion and 
  let $(\Sigma, \preceq)$ be any subexponential signature. Given an antitone map $\nu : \Sigma \to {\bf SubQuant}(\mathcal{Q})$, let us define
  a family of conuclei $\{ \bang_s : \mathcal{Q} \to \mathcal{Q} \: | \: s \in \Sigma \}$ in the following way, for $q \in Q$:
  \[
    \bang_s(q) = 
  \begin{cases}
    \bigvee \{ a \in \nu(s) \: | \: a \leq q \: \& \: a \leq 1 \}, \:\: \text{if $s \in W \setminus C$} \\
    \bigvee \{ a \in \nu(s) \: | \: a \leq q \: \& \: a \leq a \cdot a \}, \:\: \text{if $s \in C \setminus W$} \\
    \bigvee \{ a \in \nu(s) \: | \: a \leq q \: \& \: a \leq a \cdot a \: \& \: a \leq 1 \} \:\: \text{if $s \in C \cap W$} \\
    \bigvee \{ a \in \nu(s) \: | \: a \leq q \}, \:\: \text{otherwise} \\
  \end{cases}
  \]

  It is immediate that each $\Box_s$ is a quantic conucleus for any $s$, so any commutative quantale along with a family
  $\{ \Box_s : \mathcal{Q} \to \mathcal{Q} \: | \: s \in \Sigma \}$ forms a $\Sigma$-assemblage. 
\end{example}

There are quite a few examples of assemblages for the three-element signature $\Sigma_3$ consisting of 
three elements $I = \{ {\bf i}, {\bf r}, {\bf a} \}$ such that $W = \{ {\bf a}, {\bf i }\}$ and $C = \{ {\bf r}, {\bf i}\}$ and
the preorder is given as ${\bf r}, {\bf a} \preceq {\bf i}$. In this signature, we have three modalities $\bang_{\bf i}$, $\bang_{\bf r}$
and $\bang_{\bf a}$ with the extra principles $\bang_{\bf i} A \vdash \bang_{\bf r} A$ and $\bang_{\bf i} A \vdash \bang_{\bf a} A$. In other words,
if we can duplicate something in the intuitionistic sense, then we can duplicate in the relevant setting (but not vice versa). Similarly,
if we can destroy an entity intuitionistically, then it is destroyable in the affine setting.

\begin{example}~\label{lafont:category}
One can extract several curious examples of $\Sigma_3$-assemblages from presentably symmetric monoidal categories.
We refer the reader to \cite{Adamek_Rosicky_1994} and \cite[Chapter 2]{makkai1989accessible} for a more systematic study
of presentable categories.

Let $\mathcal{C}$ be a locally small category, then $\mathcal{C}$ is \emph{presentable} if
$\mathcal{C}$ has all small colimits and there is a \emph{set} $S$ of
$\kappa$-compact objects that generate $\mathcal{C}$ under $\kappa$-filtered colimits for some regular cardinal $\kappa$.
A presentably symmetric monoidal category is a presentable SMC such that $\otimes$ preserves all small colimits in each argument.

If $(\mathcal{C}, \otimes, \mathds{1})$ is a presentably symmetric monoidal category,
then $\mathcal{C}$ is also closed, which follows from the Adjoint Functor theorem.
Moreover, $\mathcal{C}/\mathds{1}$, ${\bf coMon}(\mathcal{C})$ and 
${\bf coSem}(\mathcal{C})$ are presentable; this can be shown similarly to \cite[Theorem 10]{Harington2025}.
Besides, the forgetful functors $\mathsf{U}_{\bf i} : {\bf coMon}(\mathcal{C}) \to \mathcal{C}$,
$\mathsf{U}_{\bf r} : {\bf coSem}(\mathcal{C}) \to \mathcal{C}$ and $\mathsf{U}_{\bf a} : \mathcal{C}/\mathds{1} \to \mathcal{C}$
have right adjoints. The functors ${\bf coMon}(\mathcal{C}) \to {\bf coSem}(\mathcal{C})$ and ${\bf coMon}(\mathcal{C}) \to \mathcal{C}/\mathds{1}$
sending a cocommutative comonoid in $\mathcal{C}$ to the cosemigroup (by dropping a counit) and to the corresponding object of $\mathcal{C}/\mathds{1}$ (by dropping a comultiplication)
are also strict symmetric monoidal and accessible.
Thus, the composites of the forgetful functors with those right adjoints induce a $\Sigma_3$-assemblage by Proposition~\ref{cocteau:bouton}, which we will prove below.

There are quite a few examples of categories satisfying this observation.
The first example is the category $\mbox{R-{\bf Mod}}$ of modules over a commutative ring $R$.
Thus, the category $\operatorname{Vect}_k$ of vector spaces over a field $k$
and the category of Abelian groups viewed as $\mathbb{Z}$-modules are also $\Sigma_3$-assemblages.
The category of $R$-modules $\mbox{R-{\bf Mod}}$ is equivalent to the category 
${\bf QCoh}(\mathcal{O}(\operatorname{Spec}R))$ of quasi-coherent $\mathcal{O}(\operatorname{Spec}(R))$-modules over the affine scheme $\operatorname{Spec}(R)$
\cite[Corollary II.5.5]{hartshorne2013algebraic}, so ${\bf QCoh}(\mathcal{O}(\operatorname{Spec}(R)))$
is a $\Sigma_3$-assemblage. There are also several examples of presentably SMC's (and thus $\Sigma_3$-assemblages) from algebraic topology
such as the category ${\bf Ch}(\mathcal{A})$ of chain complexes in a Grothendieck Abelian category $\mathcal{A}$ 
(see \cite[Proposition 3.10]{beke2000sheafifiable}). Thus all the aforementioned categories are models
of the system ${\bf SILL}(\lambda)_{\Sigma_3}$. There are already some results studying the semantics of (intuitionistic) linear logic
in structures from algebraic geometry as in \cite{mellies2022functorial}, so our examples demonstrate the alignment between
algebraic geometry and linear logic once more.
\end{example}
\begin{example}
One can also extract an example of an assemblage over the three-element signature with the trivial preordering obtained from equality in the study of resource modalities over the category of coherence spaces as they are described in \cite[Section 8.10]{mellies2009categorical},
where the exponential modality is decomposed into the relevant and the affine modalities obeying the Beck distributivity law. In that text, those modalities
are called the duplication and the suspension modalities respectively. Similarly, there is an example of an assemblage of the same similarity type
if we take the category of so-called Conway games with the exponential modality composed similarly from the affine and the relevant modalities in \cite{mellies2010resource}.
\end{example}
\begin{example}
There are also examples of assemblages in the coeffect calculus, a type-theoretic representation of resource-aware comonadic coeffects. In particular, \cite{brunel2014core}
introduces the coeffect calculus as an extension of intuitionistic linear type theory with a family of comonad modalities graded with elements of a preordered semiring. 
Semantically, this coeffect calculus is modelled with so-called exponential actions and bounded exponential actions, that is, comonadic actions of a bimonoidal category imposing weakening
and contraction from the structure of an acting bimonoidal category. Similar structures were studied in \cite{vollmer2024mixed} for describing models 
of a mixed linear and graded logic. 

In particular, if we take a preordered idempotent semiring, say a Boolean algebra, then exponential actions and bounded exponential actions over it will give examples of assemblages.
\end{example}

Now we formulate the theorem of semantic adequacy of ${\bf SILL}(\lambda)_{\Sigma}$ with respect to $\Sigma$-assemblages for a subexponential signature $\Sigma$.

Let $(\mathcal{C}, (\bang_s)_{s \in \Sigma}, (\mu_{s_1, s_2})_{s_1 \preceq s_2})$ be a $\Sigma$-assemblage on an SMCC $\mathcal{C}$ and let $\mathcal{I}$ be a function mapping propositional variables to objects of $\mathcal{C}$.
The interpretation $[\![.]\!]$ assigns an object of $\mathcal{C}$ to every type in an obvious way. 

The interpretation of the inference rules, in turn, is standard and is similar, for example, to \cite[§1.1.6]{abramsky2011introduction}, but the interpretation of the promotion rule
requires some attention as it involves symmetric lax monoidal comonad morphisms. Given a $\Sigma$-assemblage $(\mathcal{C}, (\bang_s)_{s \in \Sigma}, (\mu_{s_1,s_2})_{s_1 \preceq s_2})$ and let $s \preceq s_1, \ldots, s_n$,
the interpretation is given in the following way.
    \begin{prooftree}
      \AxiomC{$\Gamma_1 \xrightarrow{M_1} \bang_{s_1} A_1, \ldots, \Gamma_n \xrightarrow{M_n} \bang_{s_n} A_n$}
      \AxiomC{$\bang_{s_1} A_1 \otimes \ldots \otimes \bang_{s_n} A_n \xrightarrow{N} A$}
      \BinaryInfC{$\otimes_k \Gamma_k \xrightarrow{\bang_s(N) \circ \mathsf{m}^{s}_{\bang_{s_1} A_1, \ldots, \bang_{s_n} A_n} \circ \otimes_k {\mu_{s, s_k}}_{\bang_{s_k} A_k} \circ \otimes_k \delta^{s_k}_{A_k} \circ \otimes_k M_k} \bang_s A$}
    \end{prooftree}
  
    Let us visualise the above derivation with a diagram for $n = 2$:
    \[
    \begin{tikzcd}
      \Gamma_1 \otimes \Gamma_2 \ar[rrrr] \ar[d, "M_1 \otimes M_2"'] &&&& \bang_s A \\
      \bang_{s_1} A_1 \otimes \bang_{s_2} A_2 \ar[d, "\delta^{s_1}_{A_1} \otimes \delta^{s_2}_{A_2}"']     &&&& \bang_s (\bang_{s_1} A_1 \otimes \bang_{s_2} A_2) \ar[u, "\bang_s N"'] \\
      \bang^2_{s_1} A_1 \otimes \bang^2_{s_2} A_2 \ar[rrrr, "{\mu_{s, s_1}}_{\bang_{s_1} A_1} \otimes {\mu_{s, s_2}}_{\bang_{s_2} A_2}"'] &&&& \bang_s \bang_{s_1} A_1 \otimes \bang_s \bang_{s_2} A_2 \ar[u, "\mathsf{m}_{\bang_{s_1} A_1, \bang_{s_2} A_2}"'] 
    \end{tikzcd}
    \]
  
  \emph{An equality in context} $\Gamma \vdash M \equiv N : A$ stands for both $M$ and $N$ have type $A$ in $\Gamma$ and 
  they are equivalent modulo the smallest equivalence relation containing $\triangleleft$ from Definition~\ref{conversion:def}.

  \begin{theorem}~\label{assemblage:completeness}
  Every $\Sigma$-assemblage is a model of ${\bf SILL}(\lambda)_{\Sigma}$ and its ${\bf SILL}(\lambda)_{\Sigma}$-equalities in context.
  Moreover, there exists a $\Sigma$-assemblage $\mathcal{C}_{\Sigma}$ with the following property:
  if $\Gamma \vdash M : A$ and $\Gamma \vdash N : A$ such that $[\![M]\!]_{\mathcal{C}_{\Sigma}} = [\![N]\!]_{\mathcal{C}_{\Sigma}}$, then
  the equality in context $\Gamma \vdash M \equiv N : A$ is provable in
  the typed equational logic generated by the proof conversion rules from
  Appendix~\ref{appendix}.
  \end{theorem}

  \begin{proof}  
  Soundness is proved by induction on the generation of $\Gamma \vdash M \equiv N : A$. 
  The completeness part is proved by constructing the free $\Sigma$-assemblage from the typing rules and
  proof conversions from ${\bf SILL}(\lambda)_{\Sigma}$. The proof is standard for categorical logic, so we omit it.
  \end{proof}

  Note that there are several other approaches to categorical semantics of intuitionistic linear logic such as Lafont categories (\cite{lafont1988logiques}) and new-Seely semantics 
  (introduced in \cite{seely1989linear} and further refined in \cite{bierman1995categorical}),
  but it seems that not all of those approaches are equally scalable to the polymodal resource policies in a robust way. 
  In particular, a new-Seely category is an SMCC $\mathcal{C}$ with finite products and a strong symmetric monoidal 
  $\bang : (\mathcal{C}, \times, \top) \to (\mathcal{C}, \otimes, \mathds{1})$ with the copying and deletion natural transformations satisfying the 
  extra coherence diagram. We have several problems if we try to generalise it for the polymodal case. First, it seems we need to add extra monoidal structures
  for each subexponential, so we could get symmetric lax monoidal functors $\bang_s : (\mathcal{C}, \otimes_s, \mathds{1}) \to (\mathcal{C}, \otimes, \mathds{1})$,
  where $\otimes_s$ is the extra tensor structure reflecting the resource policy prescribed by $s$. If we also need relations between comonads 
  $\bang_{s_2}$ and $\bang_{s_1}$ for $s_1 \preceq s_2$, then we need to add extra monoidal functors from $(\mathcal{C}, \otimes_{s_2}, \mathds{1})$ to $(\mathcal{C}, \otimes_{s_1}, \mathds{1})$.
  Therefore, we believe that the new-Seely approach to the polymodal linear logic would not be as idiomatic as the expansion of linear categories we are suggesting in this paper.

  The concept of a Lafont category allows thinking of the semantics of intuitionistic multiplicative linear logic in terms of the cofree construction
  from an SMCC $\mathcal{C}$ to its category of cocommutative comonoids. The approach in the fashion of Lafont can be directly transferred to the
  $\Sigma_3$-assemblages, if we assume that there are cofree constructions not only from $\mathcal{C}$ to the category of cocommutative comonoids,
  but also to the category of cosemigroups and $\mathcal{C}/\mathds{1}$. It remains obscure to the author what the cofree construction might look like
  in a more general case in order to cover all possible $\Sigma$'s.

\section{$\Sigma$-assemblages and Resource Modalities}~\label{cocteau:marais}

In this section, we discuss how one can view comonads from $\Sigma$-assemblages
as monoidal adjunctions. To be more precise, we would like to discuss whether such
comonads can be materialised as resource modalities. The term ``resource modality'' was coined by Melli{\`e}s and Tabareau in \cite{mellies2007resource}
to categorise affine and relevant subexponential modalities in the fashion 
of Benton's linear-non-linear models \cite{benton1994mixed}. There is a well-known result in categorical semantics
of linear logic allowing one to unveil $\bang$ as a symmetric monoidal adjunction (see \cite{benton1994mixed}),
so in this section we extend this result to ${\bf SILL}(\lambda)_{\Sigma}$ for any $\Sigma$. Further, we show that scaling up the 
concept of a linear-non-linear model for multiple resource policies can be achieved by more idiomatic means if we involve the formal theory of comonads 
based on 2-categories. This will allow us to formulate and prove the representation theorem embedding the whole 2-category of all $\Sigma$-assemblages
into the 2-category of strict 2-functors of a particular kind that we will specify later in this section. 

We start by discussing relevant categories, the concept we need to treat comonads admitting the contraction rule as monoidal adjunctions.

\subsection{On Relevant Categories}~\label{relevant:cat:section}

This is an auxiliary subsection where we discuss some properties of relevant categories, that is,
symmetric monoidal categories where every object can be duplicated. The name ``relevant category'' comes from the fact
that such categories are models for substructural logics admitting the contraction rule, that is, relevant logics.
We will not be going into the categorical analysis of relevant logic and related substructural logics, 
but we refer the reader to \cite{petric2002coherence} and \cite{jacobs1994semantics}.

\begin{definition}~\label{relevant:cat}
  A symmetric monoidal category $\mathcal{C}$ admits \emph{duplication} (or \emph{copying}) if
  there is a symmetric monoidal natural transformation $\mathsf{copy} : 1_{\mathcal{C}} \Rightarrow 1_{\mathcal{C}} \otimes 1_{\mathcal{C}}$
  such that the following is satisfied:
  \[
  \begin{tikzcd}
  A \ar[rrr, "\mathsf{copy}_A"] \ar[d, "\mathsf{copy}_A"'] &&& A \otimes A \ar[d, "1_A \otimes \mathsf{copy}_A"] & A \otimes A \ar[rr, "\sigma_{A, A}"] && A \otimes A \\
  A \otimes A \ar[rrr, "\alpha_{A,A,A} \circ \mathsf{copy}_A \otimes 1_A"'] &&& A \otimes A \otimes A                    && A \ar[ul, "\mathsf{copy}_A"] \ar[ur, "\mathsf{copy}_A"']
  \end{tikzcd}
  \]
  We also call such categories \emph{relevant} categories.
\end{definition}

There is a more non-trivial equivalent characterisation of relevant categories which one can think of as
a version of the Eckmann-Hilton argument \cite{eckmann1962group} for cocommutative cosemigroups.

\begin{lemma}~\label{cosem:relevant} Let $(\mathcal{C}, \otimes, \mathds{1})$ be a symmetric monoidal category.
  \begin{enumerate}
  \item The category ${\bf coSem}(\mathcal{C})$ of cocommutative cosemigroups is relevant. 
  \item The endofunctor ${\bf coSem} : {\bf SMC}_{\operatorname{str}} \to {\bf SMC}_{\operatorname{str}}$ forms 
  a comonad over ${\bf SMC}_{\operatorname{str}}$, the category of symmetric monoidal categories and \emph{strict} symmetric monoidal functors (regarded here as a 1-category).
  \item $\mathcal{C}$ is relevant iff the forgetful functor $\mathsf{U}_{\mathcal{C}} : {\bf coSem}(\mathcal{C}) \to \mathcal{C}$ has a strict symmetric monoidal section 
  $\mathsf{V}_{\mathcal{C}} : \mathcal{C} \to {\bf coSem}(\mathcal{C})$.
  \end{enumerate}
\end{lemma}

\begin{proof} Fix cocommutative cosemigroups $\mathsf{A} = (A, \gamma_A)$ and $\mathsf{B} = (B, \gamma_B)$ in $\mathcal{C}$.
 We define the duplication operation $(A, \gamma_A) \xrightarrow{\mathsf{copy}_{(A, \gamma_A)} } (A \otimes A, \gamma_{A \otimes A})$
by $\gamma_A : A \to A \otimes A$. 
Further, let us show that $\mathsf{copy}_{(A, \gamma_A)} : (A, \gamma_A) \to (A \otimes A, \gamma_{A \otimes A})$ is natural in $(A, \gamma_A)$.
  Indeed, if $f : (A, \gamma_A) \to (B, \gamma_B)$ is a cocommutative cosemigroup morphism, then the diagram in ${\bf coSem}(\mathcal{C})$:
  \[
  \begin{tikzcd}
    (A, \gamma_A) \ar[rr, "f"] \ar[d, "\mathsf{copy}_{(A, \gamma_A)}"'] && (B, \gamma_B) \ar[d, "\mathsf{copy}_{(B, \gamma_B)}"] \\
    (A \otimes A, \gamma_{A \otimes A}) \ar[rr, "f \otimes f"'] && (B \otimes B, \gamma_{B \otimes B})
  \end{tikzcd}
  \]
  translates to the following diagram in $\mathcal{C}$:
  \[
  \begin{tikzcd}
    A \ar[d, "\gamma_A"'] \ar[rr, "f"] && B \ar[d, "\gamma_B"] \\
    A \otimes A  \ar[d, "{\gamma_A} \otimes \gamma_A"'] \ar[rr, "f \otimes f"'] && B \otimes B \ar[d, "{\gamma_B} \otimes \gamma_B"] \\
    A \otimes A \otimes A \otimes A \ar[rr, "f^{\otimes 4} "']  && B \otimes B \otimes B \otimes B
  \end{tikzcd}
  \]
  The top diagram commutes by the definition of a cosemigroup morphism. The bottom diagram also commutes 
  by the definition of a cosemigroup morphism combined with the functoriality of tensor product.
  The axioms of a relevant category follow directly from the definition of a cocommutative cosemigroup.

  The forgetful functor $\mathsf{U}_{\mathcal{C}}$ is obviously strict and symmetric, so let us check that
  ${\bf coSem}$ is a comonad. The components of the counit $\varepsilon : {\bf coSem} \Rightarrow 1$ are given by
  the forgetful functors $\mathsf{U}_{\mathcal{C}} : {\bf coSem}(\mathcal{C}) \to \mathcal{C}$
  for each symmetric monoidal category $\mathcal{C}$. The naturality of $\varepsilon$ is immediate. We specify the components of comultiplication
  $\delta_{\mathcal{C}} : {\bf coSem}(\mathcal{C}) \to {\bf coSem}^2(\mathcal{C})$
  for an SMC $\mathcal{C}$ as follows. Let $(A, \gamma_A) \in {\bf coSem}(\mathcal{C})$ be a cocommutative cosemigroup
  in $\mathcal{C}$, then we put $\delta_{\mathcal{C}} : (A, \gamma_A) \mapsto ((A, \gamma_A), \gamma_{(A, \gamma_A)})$, where
  $\gamma_{(A, \gamma_A)} : (A, \gamma_A) \to (A, \gamma_A) \otimes (A, \gamma_A) = (A \otimes A, \gamma_{A \otimes A})$
  where the right-hand side is a cocommutative cosemigroup from the previous part of this lemma. Note that
  $(A, \gamma_A)$ should be necessarily cocommutative so we can conclude that $(A, \gamma_A) \otimes (A, \gamma_A)$
  exists by Theorem~\ref{cosemi:monoidal}. $\delta_{\mathcal{C}}$ is defined on morphisms similarly.
  The rest is to check naturality for $\delta$,
  that is, the following square commutes for SMC's $\mathcal{C}$ and $\mathcal{D}$ and a strict symmetric monoidal functor 
  $\mathsf{F} : \mathcal{C} \to \mathcal{D}$: 
  \[
  \begin{tikzcd}
    {\bf coSem}(\mathcal{C}) \ar[d, "\delta_{\mathcal{C}}"']  \ar[rr, "{\bf coSem}(\mathsf{F})"] && {\bf coSem}(\mathcal{D}) \ar[d, "\delta_{\mathcal{D}}"'] \\
    {\bf coSem}^2(\mathcal{C}) \ar[rr, "{\bf coSem}^2(\mathsf{F})"] && {\bf coSem}^2(\mathcal{D})
  \end{tikzcd}
  \]
  which is routine. Further, one can show that the axioms of comonads are satisfied as an easy exercise:
  \[
    \begin{tikzcd}
      {\bf coSem}(\mathcal{C}) \ar[rr, "\delta_{\mathcal{C}}"] \ar[d, "\delta_{\mathcal{C}}"'] && {\bf coSem}^2(\mathcal{C}) \ar[d, "\delta_{{\bf coSem}\mathcal{C}}"] & {\bf coSem}(\mathcal{C}) \ar[rr, "\delta_{\mathcal{C}}"] \ar[drr, "1_{{\bf coSem}(\mathcal{C})}"] \ar[d, "\delta_{\mathcal{C}}"'] && {\bf coSem}^2(\mathcal{C}) \ar[d, "\varepsilon_{{\bf coSem}(\mathcal{C})}"] \\
      {\bf coSem}^2(\mathcal{C}) \ar[rr, "{\bf coSem}(\delta_{\mathcal{C}})"'] && {\bf coSem}^3(\mathcal{C})                                      & {\bf coSem}^2(\mathcal{C}) \ar[rr, "{\bf coSem}(\varepsilon_{\mathcal{C}})"'] && {\bf coSem}(\mathcal{C}) \\
    \end{tikzcd}
  \]

  Further, if $\mathcal{C}$ is relevant, then there is a natural transformation 
  $\mathsf{copy} : 1_{\mathcal{C}} \Rightarrow 1_{\mathcal{C}} \otimes 1_{\mathcal{C}}$ such that each component $A \in \mathcal{C}$
  is endowed with the structure of a cocommutative cosemigroup with the copying operation given by
  the corresponding component $\mathsf{copy}_A : A \to A \otimes A$.
  Put $\mathsf{V}_{\mathcal{C}} : A \mapsto (A, \mathsf{copy}_A)$. Then it is readily checked that $\mathsf{V}_{\mathcal{C}}$ is the required strict monoidal section.

  Now assume that the forgetful functor $\mathsf{U}_{\mathcal{C}} : {\bf coSem}(\mathcal{C}) \to \mathcal{C}$
  has a strict section $\mathsf{V}_{\mathcal{C}} : \mathcal{C} \to {\bf coSem}(\mathcal{C})$. In the previous part of the lemma, we showed that ${\bf coSem}(\mathcal{C})$ is relevant,
  so for each $A \in \mathcal{C}$ there is an arrow 
  $\mathsf{copy}_{\mathsf{V}_{\mathcal{C}} A} : \mathsf{V}_{\mathcal{C}} A \to (\mathsf{V}_{\mathcal{C}} A) \otimes (\mathsf{V}_{\mathcal{C}} A)$ in ${\bf coSem}(\mathcal{C})$ and, therefore,
  one has $\mathsf{U}_{\mathcal{C}}(\mathsf{copy}_{\mathsf{V}_{\mathcal{C}} A}) : A \to A \otimes A$ in $\mathcal{C}$,
  so we let $\gamma_A := \mathsf{U}_{\mathcal{C}}(\mathsf{copy}_{\mathsf{V}_{\mathcal{C}} A})$.
\end{proof}

We can also characterise relevant categories as coalgebras in the Eilenberg-Moore category 
${\bf SMC}_{\operatorname{str}}^{\bf coSem}$ consisting of coalgebras
$(\mathcal{C}, h_{\mathcal{C}})$ where $h_{\mathcal{C}} : \mathcal{C} \to {\bf coSem}(\mathcal{C})$
is a strict symmetric monoidal functor.

\begin{theorem}~\label{coalgebra:relevant} Let $\mathcal{C}$ be an SMC, then $\mathcal{C}$ is relevant iff $(\mathcal{C},h_{\mathcal{C}}) \in {\bf SMC}_{\operatorname{str}}^{\bf coSem}$ for
    some strict symmetric monoidal functor $h_{\mathcal{C}} : \mathcal{C} \to {\bf coSem}(\mathcal{C})$.
\end{theorem}

\begin{proof} If $\mathcal{C}$ is relevant, then the forgetful functor $\mathsf{U}_{\mathcal{C}} : {\bf coSem}(\mathcal{C}) \to \mathcal{C}$
    has a strict section $\mathsf{V}_{\mathcal{C}} : \mathcal{C} \to {\bf coSem}(\mathcal{C})$ by Lemma~\ref{cosem:relevant}. 

    Let us check that a pair $(\mathcal{C}, \mathsf{V}_{\mathcal{C}})$ is the required ${\bf coSem}$-coalgebra assuming that $\mathsf{V}_{\mathcal{C}}$
    is a strict symmetric monoidal section of $\mathsf{U}_{\mathcal{C}}$. That is, the following diagrams commute:
    \[
    \begin{tikzcd}
      \mathcal{C} \ar[r, "\mathsf{V}_{\mathcal{C}}"] \ar[dr, "1_\mathcal{C}"'] & {\bf coSem}(\mathcal{C}) \ar[d, "\mathsf{U}_{\mathcal{C}}"] & \mathcal{C} \ar[rr, "\mathsf{V}_{\mathcal{C}}"] \ar[d, "\mathsf{V}_{\mathcal{C}}"'] && {\bf coSem}(\mathcal{C}) \ar[d, "\delta_{\mathcal{C}}"] \\
      & \mathcal{C}                                                                                                        & {\bf coSem}(\mathcal{C}) \ar[rr, "{\bf coSem}(\mathsf{V}_{\mathcal{C}})"'] && {\bf coSem}^2(\mathcal{C})
    \end{tikzcd}
    \]
    The left-hand side triangle commutes automatically since $\mathsf{V}_{\mathcal{C}}$ is a strict symmetric monoidal section of the 
    forgetful functor. The right-hand side square commutes since $\mathsf{V}_{\mathcal{C}}$ is a section of $\mathsf{U}_{\mathcal{C}}$,
    so $\mathsf{V}_{\mathcal{C}}(A) = (A, \gamma_A)$.
    Hence ${\bf coSem}(\mathsf{V}_{\mathcal{C}})(A, \gamma_A) = ((A, \gamma_A), \gamma_{(A, \gamma_A)}) = \delta_{\mathcal{C}}(A, \gamma_A)$.
    Here $((A, \gamma_A), \gamma_{(A, \gamma_A)})$ is the cosemigroup structure on $(A, \gamma_A)$ in ${\bf coSem}(\mathcal{C})$.
    
    The other implication is immediate: if we have a coalgebra $h_{\mathcal{C}} : \mathcal{C} \to {\bf coSem}(\mathcal{C})$,
    then each $A \in \mathcal{C}$ is automatically endowed with the structure of a cocommutative cosemigroup.
\end{proof}

\subsection{The concept of a $\Sigma$-bouton and the 2-categorical approach}

Let $\mathcal{C}$ be a symmetric monoidal category and let $(\mathsf{K}, \mathsf{m}) : \mathcal{C} \to \mathcal{C}$
be a lax symmetric monoidal functor with the structure of a symmetric lax monoidal comonad. There is the
cofree-coforgetful adjunction between the underlying category $\mathcal{C}$ and the Eilenberg-Moore category 
$\mathcal{C}^{\mathsf{K}}$, which lifts to a symmetric monoidal adjunction if we equip the Eilenberg-Moore category
with a symmetric monoidal structure. Let $(A, h_A)$ and $(B, h_B)$ be $\mathsf{K}$-coalgebras in $\mathcal{C}^{\mathsf{K}}$.
The tensor product of $(A, h_A)$ and $(B, h_B)$ is defined as a coalgebra $(A \otimes B, h_{A \otimes B})$ where 
the coalgebra action $h_{A\otimes B}$
is given by the composite:
\[
\begin{tikzcd}
  A \otimes B \ar[rr, "h_A \otimes h_B"] && \mathsf{K} A \otimes \mathsf{K} B \ar[r, "\mathsf{m}_{A,B}"] & \mathsf{K} (A \otimes B)
\end{tikzcd}
\]
And the unit coalgebra is defined as $(\mathds{1}, \mathsf{m}_{\mathds{1}})$. Besides, the Eilenberg-Moore category $\mathcal{C}^{\mathsf{K}}$ is an SMC, see
\cite{lack2005limits} for more details and further generalisations.

Let $\mathcal{C}$ be a symmetric monoidal (closed) category, then a \emph{modality} on $\mathcal{C}$ is given by a symmetric monoidal adjunction
\[
\begin{tikzcd}
\mathcal{D} \ar[rr, shift left=1.25ex, "\mathsf{F}^{\natural}", ""{name=Fl}] && \mathcal{C} \ar[ll, shift left=1.25ex, "\mathsf{F}^{\sharp}", ""{name=Fr}]
\arrow[phantom, from=Fr, to=Fl, "\dashv" rotate=-90]
\end{tikzcd}
\]
for a symmetric monoidal category $\mathcal{D}$.
The notation is $((\mathsf{F}^{\natural}, \mathsf{n}), (\mathsf{F}^{\sharp}, \mathsf{o}), \mathcal{D}) : \mathcal{C} \to \mathcal{C}$,
where $\mathsf{n}$ and $\mathsf{o}$ are components of $\mathsf{F}^{\natural}$ and $\mathsf{F}^{\sharp}$ respectively 
\footnote{Sometimes we would just write $(\mathsf{F}^{\natural}, \mathsf{F}^{\sharp}, \mathcal{D}) : \mathcal{C} \to \mathcal{C}$ assuming that
the letters $\mathsf{n}$ and $\mathsf{o}$ are already reserved for the components of a left and a right adjoint respectively.}.

The key observation is that the composite $\mathsf{F}^{\natural} \circ \mathsf{F}^{\sharp} : \mathcal{C} \to \mathcal{C}$ is a symmetric lax monoidal comonad. 
Although the notion of a modality coincides with the notion of a monoidal adjunction,
we proceed with the term ``modality'' inspired by Yetter \cite{yetter1990quantales}, who used this name for a quantic conucleus 
over a (commutative) quantale, whose formal properties coincide with the properties of (symmetric) lax monoidal comonads.
For comparison, Melli{\`e}s and Tabareau \cite{mellies2007resource, mellies2010resource} call such adjunctions 
``resource modalities'', but further we will consider a broader class of comonadic modal operators, so we reserve the term
``resource modality'' only for the modalities dealing with resource management policies.

As Benton showed in \cite[Theorem 3, Theorem 8]{benton1994mixed}, every
linear category induces a symmetric monoidal adjunction between an SMCC and a Cartesian closed category.
Such adjunctions are called \emph{linear-non-linear models}, and they are an instance of the concept of a modality.
However, as it was discussed in \cite{mellies2009categorical}, one can weaken the definition of a linear-non-linear model and construct
the desired adjunction between an SMCC and a Cartesian category, not necessarily admitting exponential objects. 
We are in favour of the latter way since it is less restrictive. Benton built his linear-non-linear models by taking the adjunction between 
an SMCC $\mathcal{C}$ and the co-Kleisli category $\mathcal{C}_{\bang}$, which happens to be Cartesian closed. But this approach is not transferable
for a non-exponential $\bang_s$ since the co-Kleisli category $\mathcal{C}_{\bang_s}$ is not necessarily an SMC generally. 
Thus, Eilenberg-Moore categories have more favourable properties for us.

In this section we will show that every $\Sigma$-assemblage can be also viewed as an SMCC equipped with symmetric monoidal adjunctions of a particular kind and left morphisms between them.

In the previous section, we have already done some preliminary work on how one can characterise relevant categories
to use those results further to represent relevant subexponentials as a monoidal adjunction. Let us discuss
what monoidal categories we need to obtain a similar characterisation for affine subexponentials further.

\begin{definition}
  A symmetric monoidal category $\mathcal{C}$ is \emph{semicartesian}
  if the unit object is terminal. That is, there is a unique $\top_{A} : A \to \mathds{1}$ for each $A \in \mathcal{C}$. 
\end{definition}

Note that the tensor product operation admits projections in every semicartesian category:
\begin{center}
  $\pi_1 : A \otimes B \xrightarrow{1_A \otimes \top_{B}} A \otimes \mathds{1} \xrightarrow{\rho_A} A$

  $\pi_2 : A \otimes B \xrightarrow{\top_{A} \otimes 1_B} \mathds{1} \otimes B \xrightarrow{\lambda_B} B$
\end{center}

Let $(\mathcal{C}, \otimes, \mathds{1})$ be any monoidal category, then the slice category $\mathcal{C}/\mathds{1}$ becomes semicartesian with the tensor product 
operation defined as follows for $f, g \in \mathcal{C}/\mathds{1}$. 
\[
\begin{tikzcd}
  A \otimes B \ar[r, "f \otimes g"] & \mathds{1} \otimes \mathds{1} \ar[r, "\lambda = \rho"] & \mathds{1}
\end{tikzcd}
\]
$\mathcal{C}/\mathds{1}$ is semicartesian since the identity $1_{\mathds{1}} : \mathds{1} \to \mathds{1}$ is terminal
in $\mathcal{C}/\mathds{1}$.
One can also observe that a symmetric monoidal category is semicartesian iff the forgetful functor
$\mathsf{U} : \mathcal{C}/\mathds{1} \to \mathcal{C}$ is an isomorphism.

Now we give the definition of a resource modality.
\begin{definition}
  Let $\mathcal{C}$ be an SMCC and let $\mathcal{D}$ be an SMC, then a modality $(\mathsf{F}^{\natural}, \mathsf{F}^{\sharp}, \mathcal{D}) : \mathcal{C} \to \mathcal{C}$ is
  \emph{affine} if $\mathcal{D}$ is semicartesian,
  \emph{relevant} if $\mathcal{D}$ is relevant, \emph{exponential} if $\mathcal{D}$ is Cartesian.

  A modality is called a \emph{resource modality} if it satisfies any of the above conditions.
\end{definition}

The following concept expands Benton's linear-non-linear adjunctions for an arbitrary subexponential signature $\Sigma$. The idea behind this definition is that we choose a symmetric monoidal closed category and then we choose a family of symmetric monoidal categories parametrised
by $\Sigma$ materialising the structure of the underlying preorder with the corresponding morphisms of symmetric monoidal adjunctions.

\begin{definition}
  Let $\Sigma = (I, \preceq, W, C)$ be a subexponential signature, a \emph{$\Sigma$-bouton}\footnote{The word ``bouton'' means a flower bud in French. The choice of name is a metaphorical description of a category equipped with some adjunctions viewed as petals.} on a symmetric monoidal closed 
  category $\mathcal{C}$ consists of the following data:
  \begin{itemize}
    \item For each $s \in \Sigma$, there is a modality $(\mathsf{F}_s^{\natural}, \mathsf{F}_s^{\sharp}, \mathsf{T}_s) : \mathcal{C} \to \mathcal{C}$,
    \item If $s \in W$ ($s \in C$, $s \in W \cap C$), then the resource modality 
    $(\mathsf{F}_s^{\natural}, \mathsf{F}_s^{\sharp},\mathsf{T}_s)$ is affine (relevant, exponential),
    \item For $s_1 \preceq s_2$, there is a strict symmetric monoidal functor $\mathsf{T}_{s_1,s_2} : \mathsf{T}_{s_2} \to \mathsf{T}_{s_1}$ such that the following triangle commutes:
    \[
    \begin{tikzcd}
      \mathsf{T}_{s_2} \ar[dr, "{\mathsf{F}^{\natural}_{s_2}}"'] \ar[rr, "\mathsf{T}_{s_1, s_2}"] && \mathsf{T}_{s_1} \ar[dl, "{\mathsf{F}^{\natural}_{s_1}}"] \\
      & \mathcal{C}
    \end{tikzcd}
    \]
    \item If $s_1 \preceq s_2$ and $s_1 \in C$, then the following triangle commutes for each $A \in \mathsf{T}_{s_2}$:
      \[
      \begin{tikzcd}
        \mathsf{T}_{s_1,s_2}(A) \ar[drr, "\operatorname{copy}^{s_1}_{\mathsf{T}_{s_1,s_2} A}"'] \ar[rr, "\mathsf{T}_{s_1,s_2}(\operatorname{copy}^{s_2}_A)"] && \mathsf{T}_{s_1,s_2}(A \otimes A) \ar[d, equal] \\
        && \mathsf{T}_{s_1,s_2}(A) \otimes \mathsf{T}_{s_1,s_2}(A)
      \end{tikzcd}
      \]
      In other words, the copying operations, whenever they are given, are preserved homomorphically.
  \item If $s_1 \preceq s_2$ and $s_1 \in W$, then the following triangle commutes for each $A \in \mathsf{T}_{s_2}$:
      \[
      \begin{tikzcd}
        \mathsf{T}_{s_1,s_2}(A) \ar[drr, "\top^{s_1}_{\mathsf{T}_{s_1,s_2} A}"'] \ar[rr, "\mathsf{T}_{s_1,s_2}(\top^{s_2}_{A})"] && \mathsf{T}_{s_1,s_2}(\mathds{1}) \ar[d, equal] \\
        && \mathds{1}
      \end{tikzcd}
      \]
  \end{itemize}
\end{definition}

The ultimate goal of this subsection is to show that the 2-category of all $\Sigma$-assemblages 1,2-fully faithfully into the category of all $\Sigma$-boutons to strengthen
Benton's result for all possible subexponential signatures by using the formal theory of comonads.
The first step is to define the required 2-categories accurately.

\begin{definition}~\label{2:cat:cocteau} Let $\Sigma$ be any subexponential signature.
${\bf Assemblage}_{\Sigma}$ is the 2-category of all $\Sigma$-assemblages consisting of the data:
\begin{itemize}
  \item The class of 0-cells is the class of all $\Sigma$-assemblages,
  \item Let $(\mathcal{C}, (\bang_s)_{s \in \Sigma}, (\mu_{s_1, s_2})_{s_1 \preceq s_2} )$ and
  $(\mathcal{D}, (\bang'_s)_{s \in \Sigma}, (\mu'_{s_1, s_2})_{s_1 \preceq s_2} )$ be $\Sigma$-assemblages, 
  the class of 1-cells ${\bf Assemblage}_{\Sigma}(\mathcal{C}, \mathcal{D})$ consists of strict symmetric monoidal functors
  $\mathsf{G} : \mathcal{C} \to \mathcal{D}$ such that:
  \begin{itemize}
  \item $\mathsf{G}$ induces symmetric lax monoidal comonad morphisms $(\mathsf{G}, \widetilde{\mathsf{G}_s}) : (\mathcal{C}, \bang_s) \Rightarrow (\mathcal{D}, \bang'_s)$ for every $s$
   with the components $\widetilde{\mathsf{G}_s}: \mathsf{G} \bang_s \Rightarrow \bang'_s \mathsf{G}$ in the sense of Definition~\ref{comonad:morph},
  \item For each $s \in C$, the functor $\mathsf{G}_{s} : \mathcal{C}^{\bang_s} \to \mathcal{D}^{\bang'_{s}}$
 of the corresponding Eilenberg-Moore categories preserves the copying arrow for each coalgebra $(A, h_A) \in \mathcal{C}^{\bang_{s}}$:
  \[
    \begin{tikzcd} 
     \mathsf{G} A \ar[d, equal] \ar[rr, "\mathsf{G}_{s}(\gamma_A)" ] && \mathsf{G} A \otimes \mathsf{G} A \ar[d, equal] \\
     \mathsf{G} A \ar[rr, "\gamma_{\mathsf{G}_{s}(A)}"'] && \mathsf{G} A \otimes \mathsf{G} A
    \end{tikzcd}
  \] 
  Similarly, if $s \in W$, the functor $\mathsf{G}_s : \mathcal{C}^{\bang_{s}} \to \mathcal{D}^{\bang'_{s}}$
  preserves the deletion arrow as follows for each coalgebra $(A, h_A) \in \mathcal{C}^{\bang_{s}}$:
  \[
  \begin{tikzcd}
   \mathsf{G} A \ar[d, equal] \ar[rr, "\mathsf{G}_{s}(\iota_A)"] && \mathds{1} \ar[d, equal]\\
   \mathsf{G} A \ar[rr, "\iota_{\mathsf{G}_{s}(A)}"'] && \mathds{1}
  \end{tikzcd}
  \]
  \item the following squares commute for any $s_1 \preceq s_2$ in $\Sigma$:
  \begin{equation}~\label{cocteau:1cell:axiom1}
  \begin{tikzcd}
  \mathsf{G} \bang_{s_2} \ar[rr, "\mathsf{G}_{s_2}"] \ar[d, "\mathsf{G} \mu_{s_1,s_2}"'] && \bang'_{s_2} \mathsf{G} \ar[d, "\mu'_{s_1,s_2} \mathsf{G}"] \\
  \mathsf{G} \bang_{s_1} \ar[rr, "\mathsf{G}_{s_1}"'] && \bang'_{s_1} \mathsf{G}  \\
  \end{tikzcd}
  \end{equation}
  \end{itemize}

    \item Let $\mathsf{G}, \mathsf{H} \in {\bf Assemblage}_{\Sigma}(\mathcal{C}, \mathcal{D})$
    be 1-cells, the class ${\bf Assemblage}_{\Sigma}(\mathsf{G}, \mathsf{H})$ of 2-cells from $\mathsf{G}$ to $\mathsf{H}$
    consists of symmetric monoidal natural transformations $\theta : \mathsf{G} \Rightarrow \mathsf{H}$ inducing symmetric lax monoidal comonad morphism 
    transformations between $(\mathsf{G}, \widetilde{\mathsf{G}_s})$ and $(\mathsf{H}, \widetilde{\mathsf{H}_s})$ for any $s$ and the following condition is satisfied for $s_1 \preceq s_2$:
      \begin{equation}~\label{final:axiom}
      \begin{tikzcd}
        \mathsf{G} \bang_{s_2} \ar[rr, "\theta_{s_2}"] \ar[d, "\mathsf{G} \mu_{s_1,s_2}"']  && \bang'_{s_2} \mathsf{H} \ar[d, "\mu'_{s_1,s_2} \mathsf{H}"] \\
        \mathsf{G} \bang_{s_1} \ar[rr, "\theta_{s_1}"'] &&  \bang'_{s_1} \mathsf{H} \\
      \end{tikzcd}
    \end{equation}
    where $\theta_{s}$ for any $s$ is defined as
    \[
    \begin{tikzcd}
      \mathsf{G} \bang_{s} \ar[rr, "\theta \bang_{s}"] \ar[drr, dashrightarrow, "\theta_{s}"] \ar[d, "\widetilde{\mathsf{G}_{s}}"'] && \mathsf{H} \bang_{s} \ar[d, "\widetilde{\mathsf{H}_{s}}"] \\
      \bang'_{s} \mathsf{G} \ar[rr, "\bang'_{s} \theta"'] && \bang'_{s} \mathsf{H}
    \end{tikzcd}
    \]
\end{itemize}
\end{definition}

Prima facie, the reader might find the definition of 2-cells incomplete since 1-cells induce
strict functors of Eilenberg-Moore categories that preserve the copying and deletion
operations, but this is not reflected at the level of 2-cells. Let $\theta : \mathsf{G}
\Rightarrow \mathsf{H}$ be a 2-cell in ${\bf Assemblage}_{\Sigma}$. Recall that $\mathsf{G}$
and $\mathsf{H}$ induce comonad morphisms $(\mathcal{C}, \bang_{s}) \Rightarrow (\mathcal{D},
\bang'_{s})$ with the components $\widetilde{\mathsf{G}}_{s}$ and
$\widetilde{\mathsf{H}}_{s}$, so we have the induced functors of Eilenberg-Moore categories
$\mathsf{G}_{s}, \mathsf{H}_{s} : \mathcal{C}^{\bang_{s}} \to \mathcal{D}^{\bang'_{s}}$,
given by $\mathsf{G}_{s} : (A, h_A) \mapsto (\mathsf{G} A,
(\widetilde{\mathsf{G}}_{s})_{A} \circ \mathsf{G}(h_A))$ and $\mathsf{H}_{s} : (A, h_A)
\mapsto (\mathsf{H} A, (\widetilde{\mathsf{H}}_{s})_{A} \circ \mathsf{H}(h_A))$. Then
$\theta : \mathsf{G} \Rightarrow \mathsf{H}$ extends to a natural transformation
$\widehat{\theta}_{s} : \mathsf{G}_{s} \Rightarrow \mathsf{H}_{s}$ between strict symmetric
monoidal functors, so the following diagram automatically commutes for each coalgebra
$(A, h_A) \in \mathcal{C}^{\bang_{s}}$:
\[
\begin{tikzcd}
  \mathsf{G} A \ar[d, "\widetilde{\theta}_A"'] \ar[rr, "\widetilde{\mathsf{G}_{s}}{(\gamma_A)}"] && \mathsf{G} A \otimes \mathsf{G} A   \ar[d, "\widetilde{\theta}_A \otimes \widetilde{\theta}_A"] \ar[rr, equal]  && \mathsf{G} (A \otimes A)  \ar[d, "\widetilde{\theta}_{A \otimes A}"]  \\
  \mathsf{H} A \ar[rr, "\widetilde{\mathsf{H}_{s}}{(\gamma_A)}"'] && \mathsf{H} A \otimes \mathsf{H}A \ar[rr, equal] && \mathsf{H} (A \otimes A)
\end{tikzcd}
\]

Now we would like to fully faithfully embed the whole 2-category ${\bf Assemblage}_{\Sigma}$ into some 2-category of 2-functors, their strict natural transformations
and modifications similarly by using Proposition~\ref{coalgebra:theorem:pr} that characterises single comonads. Before specifying such a 2-category, let us show the following fact.
It is a folklore fact from basic category theory that the composition of adjoint functors is a comonad. 
However, the concept of a $\Sigma$-bouton also has strict functors to reflect the corresponding comonad morphism.
Therefore, we need a slightly more complicated proof compared to a standard argument in linear logic
showing how one can obtain a linear category from a linear-non-linear model as, for example, in 
\cite[§7.4]{mellies2009categorical}.

\begin{prop}~\label{cocteau:bouton} 
  Let $\Sigma$ be a subexponential signature. Every $\Sigma$-bouton induces a $\Sigma$-assemblage, and, conversely, 
  every $\Sigma$-assemblage induces a $\Sigma$-bouton.
\end{prop}

\begin{proof}
  Given a $\Sigma$-bouton $(\mathcal{C}, (\mathsf{F}^{\natural}_s \dashv \mathsf{F}^{\sharp}_s)_{{s \in \Sigma}}, {\mathsf{T}_{s_1,s_2}}_{s_1 \preceq s_2})$.
  The complete proof of the claim that a $\Sigma$-bouton induces an assemblage is routine, but let us write down a proof of how one obtains a relevant modality
  from a monoidal adjunction with a relevant category and how one can get a morphism of relevant modalities from a corresponding left adjunction morphism.

  Let $(\mathsf{F}_s^{\natural}, \mathsf{F}_s^{\sharp}, \mathsf{T}_s) : \mathcal{C} \to \mathcal{C}$ for $s \in C$
be a relevant resource modality on $\mathcal{C}$, let us show that the composite $\bang_{s} = \mathsf{F}_s^{\natural} \circ \mathsf{F}_s^{\sharp} : \mathcal{C} \to \mathcal{C}$
forms a relevant comonad on $\mathcal{C}$. The components are given as follows:
\begin{itemize}
  \item $\mathsf{m}^{s}_{A, B} := \mathsf{F}_s^{\natural} \mathsf{F}_s^{\sharp} A \otimes \mathsf{F}_s^{\natural} \mathsf{F}_s^{\sharp} B \xrightarrow{\mathsf{n}^s_{\mathsf{F}_s^{\sharp}A, \mathsf{F}_s^{\sharp} B}} \mathsf{F}_s^{\natural} (\mathsf{F}_s^{\sharp} A \otimes \mathsf{F}_s^{\sharp}B) \xrightarrow{\mathsf{F}_s^{\natural} (\mathsf{o}^s_{A, B})} \mathsf{F}_s^{\natural} \mathsf{F}_s^{\sharp} (A \otimes B)$,
  \item $\mathsf{m}^{s}_{\mathds{1}} := \mathds{1} \xrightarrow{\mathsf{n}^{s}_{\mathds{1}}}  \mathsf{F}_s^{\natural} \mathds{1} \xrightarrow{\mathsf{F}_s^{\natural}(\mathsf{o}^s_{\mathds{1}})} \mathsf{F}_s^{\natural} \mathsf{F}_s^{\sharp} \mathds{1}$,
  \item $\varepsilon^{s}_A : \mathsf{F}_s^{\natural} \mathsf{F}_s^{\sharp} A \to A$,
  \item $\delta^{s}_A : \mathsf{F}_s^{\natural} \mathsf{F}_s^{\sharp} A \xrightarrow{\mathsf{F}_s^{\natural}(\eta^s_{\mathsf{F}_s^{\sharp} A}) }  \mathsf{F}_s^{\natural} \mathsf{F}_s^{\sharp} \mathsf{F}_s^{\natural} \mathsf{F}_s^{\sharp} A$.
\end{itemize}

A similar statement holds for affine and exponential resource modalities. The copying natural transformation
  $\mathsf{c}^{s}$ is given by the components $\mathsf{c}^{s}_A : \bang_{s} A \to \bang_s A \otimes \bang_s A$, each of which is defined by the following composite:
  \[
    \bang_{s} A = \mathsf{F}_s^{\natural} \mathsf{F}_s^{\sharp} A \xrightarrow{\mathsf{F}_s^{\natural}(\mathsf{copy}_{{\mathsf{F}_s^{\sharp} A}} ) } \mathsf{F}_s^{\natural} (\mathsf{F}_s^{\sharp} A \otimes  \mathsf{F}_s^{\sharp} A) \xrightarrow{\mathsf{p}_{\mathsf{F}_s^{\sharp} A, \mathsf{F}_s^{\sharp} A }} (\mathsf{F}_s^{\natural} \mathsf{F}_s^{\sharp} A)^{\otimes 2} = \bang_s A \otimes \bang_s A.
  \]
  where
  \[\mathsf{p} : \mathsf{F}^{\natural}_s(A \otimes B) \to \mathsf{F}^{\natural}_s A \otimes \mathsf{F}^{\natural}_s B
  \]
  is a component of $\mathsf{F}^{\natural}$, which is symmetric oplax monoidal by \cite[Proposition 12]{mellies2009categorical}. The complete check of the relevant comonad conditions is routine.

  Now let $s_1 \preceq s_2$ and $s_1 \in C$. We need to construct a comonad morphism $\mu_{s_1,s_2} : \bang_{s_2} \Rightarrow \bang_{s_1}$
  from $\mathsf{T}_{s_1,s_2} : \mathsf{T}_{s_2} \to \mathsf{T}_{s_1}$ such that the following triangle commutes:
  \[
  \begin{tikzcd}
    \mathsf{T}_{s_2} \ar[dr, "\mathsf{F}_{s_2}^{\natural}"']  \ar[rr, "\mathsf{T}_{s_1,s_2}"] && \mathsf{T}_{s_1} \ar[dl, "\mathsf{F}_{s_1}^{\natural}"] \\
    & \mathcal{C}
  \end{tikzcd}
  \]
  So we let:
  \[
    \mu_{s_1,s_2} := \mathsf{F}_{s_2}^{\natural} \mathsf{F}_{s_2}^{\sharp} = \mathsf{F}_{s_1}^{\natural} \mathsf{T}_{s_1,s_2} \mathsf{F}_{s_2}^{\sharp} \xrightarrow{\mathsf{F}_{s_1}^{\natural}(\eta^s  \mathsf{T}_{s_1,s_2} \mathsf{F}_{s_2}^{\sharp})} \mathsf{F}_{s_1}^{\natural} \mathsf{F}_{s_1}^{\sharp} \mathsf{F}_{s_1}^{\natural}  \mathsf{T}_{s_1,s_2} \mathsf{F}_{s_2}^{\sharp} = \mathsf{F}_{s_1}^{\natural} \mathsf{F}_{s_1}^{\sharp} \mathsf{F}_{s_2}^{\natural} \mathsf{F}_{s_2}^{\sharp} \xrightarrow{\mathsf{F}_{s_1}^{\natural} \mathsf{F}_{s_1}^{\sharp} \varepsilon^{s_2}} \mathsf{F}_{s_1}^{\natural} \mathsf{F}_{s_1}^{\sharp}
  \]
  Note that a pair of functors $(1_{\mathcal{C}}, \mathsf{T}_{s_1,s_2})$ is a left adjunction morphism, then by Proposition~\ref{2:comonad:morphish:prop},
  $\mu_{s_1,s_2}$ is a symmetric lax monoidal comonad morphism. The rest is to check that $\mu_{s_1,s_2}$ preserves copying, that is, we must show that the following diagram commutes:
  \begin{equation}~\label{goal:2}
    \begin{tikzcd}
      \bang_{s_2} A \ar[rr, "\mathsf{c}^{s_2}_A"] \ar[d, "{\mu_{s_1,s_2}}_A"'] && (\bang_{s_2} A)^{\otimes 2} \ar[d, "{{\mu_{s_1,s_2}}_A} \otimes {{\mu_{s_1,s_2}}_A}"] \\
      \bang_{s_1} A \ar[rr, "\mathsf{c}^{s_1}_A"] && (\bang_{s_1} A)^{\otimes 2}\\
    \end{tikzcd}
  \end{equation}

  Let us unveil the above diagram:
    \[
    \begin{tikzcd}
      |[alias=node11]| \mathsf{F}_{s_1}^{\natural}  \mathsf{T}_{s_1,s_2} \mathsf{F}_{s_2}^{\sharp} A \ar[dd, equal] \ar[rrr, "\mathsf{F}_{s_1}^{\natural}  \mathsf{T}_{s_1,s_2} (\operatorname{copy}^{s_2}_{\mathsf{F}_{s_2}^{\sharp} A})"] &&& |[alias=node31]| \mathsf{F}_{s_1}^{\natural}  \mathsf{T}_{s_1,s_2} (\mathsf{F}_{s_2}^{\sharp} A)^{\otimes 2} \ar[dd, equal] \ar[rrrr, "\mathsf{p}^{s_2}_{\mathsf{T}_{s_2}^{\sharp} A, \mathsf{T}_{s_2}^{\sharp} A}"] &&&& (\mathsf{F}_{s_1}^{\natural}  \mathsf{T}_{s_1,s_2} \mathsf{F}_{s_2}^{\sharp} A)^{\otimes 2} \ar[dddd, "(\mathsf{F}_{s_1}^{\natural}(\eta^{s_1}_{ \mathsf{T}_{s_1,s_2} \mathsf{F}_{s_2}^{\sharp} A}))^{\otimes 2}"'] \\
      \\
      |[alias=node12]| \mathsf{F}_{s_1}^{\natural}  \mathsf{T}_{s_1,s_2} \mathsf{F}_{s_2}^{\sharp} A \ar[rrr, "\mathsf{F}_{s_1}^{\natural} (\mathsf{copy}_{ \mathsf{T}_{s_1,s_2} \mathsf{F}_{s_2}^{\sharp} A})"] \ar[dd, "\mathsf{F}_{s_1}^{\natural}(\eta^{s_1}_{ \mathsf{T}_{s_1,s_2} \mathsf{F}_{s_2}^{\sharp} A})"'] &&& |[alias=node21]| \mathsf{F}_{s_1}^{\natural}  \mathsf{T}_{s_1,s_2} (\mathsf{F}_{s_2}^{\sharp} A)^{\otimes 2} \ar[dd, "\mathsf{F}_{s_1}^{\natural}((\eta^{s_1}_{ \mathsf{T}_{s_1,s_2} \mathsf{T}_{s_2}^{\sharp} A})^{\otimes 2})"] \\
      \\
      |[alias=node13]| \mathsf{F}_{s_1}^{\natural} \mathsf{F}_{s_1}^{\sharp} \mathsf{F}_{s_1}^{\natural}  \mathsf{T}_{s_1,s_2} \mathsf{F}_{s_2}^{\sharp} A \ar[dd, "\mathsf{F}_{s_1}^{\natural} \mathsf{F}_{s_1}^{\sharp} (\varepsilon^{s_1}_A)"'] \ar[rrr, "\mathsf{F}_{s_1}^{\natural}(\mathsf{copy}_{\mathsf{F}_{s_1}^{\sharp} \mathsf{F}_{s_1}^{\natural}  \mathsf{T}_{s_1,s_2} \mathsf{F}_{s_2}^{\sharp} A} )"] &&& |[alias=node22]| \mathsf{F}_{s_1}^{\natural} (\mathsf{F}_{s_1}^{\sharp} \mathsf{F}_{s_1}^{\natural}  \mathsf{T}_{s_1,s_2} \mathsf{F}_{s_2}^{\sharp} A)^{\otimes 2} \ar[dd, "\mathsf{F}_{s_1}^{\natural} (\mathsf{F}_{s_1}^{\sharp}(\varepsilon^{s_2}_A))^{\otimes 2}"] \ar[rrrr, "\mathsf{p}^{s_1}"] &&&& |[alias=node32]| (\mathsf{F}_{s_1}^{\natural} \mathsf{F}_{s_1}^{\sharp} \mathsf{F}_{s_1}^{\natural}  \mathsf{T}_{s_1,s_2} \mathsf{F}_{s_2}^{\sharp} A)^{\otimes 2} \ar[dd, "(\mathsf{F}_{s_1}^{\natural} \mathsf{F}_{s_1}^{\sharp} (\varepsilon^{s_2}_A))^{\otimes 2}"'] \\
      \\
      \mathsf{F}_{s_1}^{\natural} \mathsf{F}_{s_1}^{\sharp} A \ar[rrr, "\mathsf{F}_{s_1}^{\natural} (\mathsf{copy}_{\mathsf{F}_{s_1}^{\sharp} A})"'] &&& |[alias=node23]| \mathsf{F}_{s_1}^{\natural} (\mathsf{F}_{s_1}^{\sharp} A \otimes \mathsf{F}_{s_1}^{\sharp} A ) \ar[rrrr, "\mathsf{p}^{s_1}_{\mathsf{F}_{s_1}^{\sharp} A, \mathsf{F}_{s_1}^{\sharp} A}"'] &&&& |[alias=node33]| \mathsf{F}_{s_1}^{\natural} \mathsf{F}_{s_1}^{\sharp} A \otimes \mathsf{F}_{s_1}^{\natural} \mathsf{F}_{s_1}^{\sharp} A
      \ar[from=node11, to=node21, eq=diag22]
      \ar[from=node12, to=node22, eq=diag23]
      \ar[from=node13, to=node23, eq=diag24]
      \ar[from=node31, to=node32, eq=diag25]
      \ar[from=node22, to=node33, eq=diag26]
    \end{tikzcd}
    \]
  where:
  \begin{itemize}
    \item (\ref{diag22}) holds by the definition of a $\Sigma$-bouton,
    \item (\ref{diag23}) and (\ref{diag24}) hold by the naturality of $\mathsf{copy}$,
    \item (\ref{diag25}) and (\ref{diag26}) commute since ${F}^{\natural}$ is symmetric oplax monoidal by \cite[Proposition 12]{mellies2009categorical}.
  \end{itemize}

To show that every $\Sigma$-assemblage induces a $\Sigma$-bouton, one needs to show that
every constituent comonad of a $\Sigma$-assemblage induces the corresponding adjunction.
The construction building an adjunction from an exponential comonad is standard, so we just note
that the construction by Benton or Melli{\`e}s is completely preserved since one can think of linear categories
as exponential reducts of $\Sigma$-assemblages. Dealing with the affine part is simple as the Eilenberg-Moore category of coalgebras
over an affine comonad is semicartesian itself.

In the case of relevant modalities, the key observation is that the resulting Eilenberg-Moore category is relevant, so we get
the desired relevant resource modality with the cofree and coforgetful adjunction. The proof is an adaptation of the construction described in \cite[§7.4]{mellies2009categorical} with a few modifications,
but the structure of the argument is similar. Now fix $s \in C$ and take $\bang_s : \mathcal{C} \to \mathcal{C}$. The key idea is that if we have a relevant comonad, then every cofree coalgebra 
$(\bang_s  A, \delta^s_A)$ is endowed with a cocommutative cosemigroup structure and we lift this structure
to every coalgebra $(A, h_A)$ by showing that every coalgebra of that form is a retract of the cofree coalgebra 
$(\bang_s  A, \delta^s_A)$ such that the following diagram commutes:
    \[
    \begin{tikzcd}
      A \ar[r, "h_A"] \ar[d, "h_A"'] & \bang_s  A \ar[r, "\mathsf{c}_A"] & \bang_s A \otimes \bang_s A \ar[r, "\varepsilon^s_A \otimes \varepsilon^s_A"] & A \otimes A \ar[d, "h_A \otimes h_A"] \\
      \bang_s  A \ar[rrr, "\mathsf{c}^{s}_A"'] &&& \bang_s A \otimes \bang_s A
    \end{tikzcd}
    \]

    To show that $\mathcal{C}^{\bang_s}$ is relevant, we must provide a natural transformation $\gamma : 1_{\mathcal{C}^{\bang_s}} \Rightarrow 1_{\mathcal{C}^{\bang_s}}(.) \otimes 1_{\mathcal{C}^{\bang_s}}(.)$ endowing each $A$ with the structure of a cocommutative cosemigroup $(A, \gamma_A)$.
    For $A$, let $\gamma_A$ denote the following composite:
    \[
    A \xrightarrow{h_A} \bang_s A \xrightarrow{\mathsf{c}^s_A} \bang_s A \otimes \bang_s A \xrightarrow{\varepsilon^s_A \otimes \varepsilon^s_A} A \otimes A
    \]
    It is immediate to show that $\gamma$ is a symmetric lax monoidal natural transformation. $\gamma_A$ is also a coalgebra morphism, that is, we want the following square
      \[
      \begin{tikzcd}
        A \ar[rr, "h_A"] \ar[dd, "h_A"'] && \bang_s A \ar[rr, "\mathsf{c}_A"] && \bang_s A \otimes \bang_s A \ar[rr, "\varepsilon^s_A \otimes \varepsilon^s_A"] && A \otimes A \ar[d, "h_A \otimes h_A"] \\
        &&&&&& \bang_s A \otimes \bang_s A \ar[d, "\mathsf{m}^s_{A, A}"] \\
        \bang_s A \ar[rr, "\bang_s(h_A)"'] && \bang_s^2 A \ar[rr, "\bang_s(\mathsf{c}_A)"'] && \bang_s(\bang_s A \otimes \bang_s A) \ar[rr, "\bang_s(\varepsilon^s_A \otimes \varepsilon^s_A)"'] && \bang_s (A \otimes A)
      \end{tikzcd}
      \]
      to commute in $\mathcal{C}$.
      \[
      \begin{array}{lll}
        & \bang_s(\gamma_A) \circ h_A = \bang_s(\varepsilon^s_A \otimes \varepsilon^s_A) \circ \bang_s(\mathsf{c}_A) \circ \bang_s(h_A) \circ h_A = & \\
        & \:\:\:\: \text{By the property of coalgebras} & \\
        & \bang_s(\varepsilon^s_A \otimes \varepsilon^s_A) \circ \bang_s(\mathsf{c}_A) \circ \delta_A \circ h_A = & \\
        & \:\:\:\: \text{By the definition of a relevant comonad and by the definition of lax monoidal functors} & \\
        & \bang_s(\varepsilon^s_A \otimes \varepsilon^s_A) \circ \mathsf{m}^s_{\bang_s A, \bang_s A} \circ \delta^{\otimes 2}_A \circ \mathsf{c}_A \circ h_A = \mathsf{m}^s_{A, A} \circ (\bang_s \varepsilon^s_A) \otimes (\bang_s \varepsilon^s_A) \circ \delta^s_A \otimes \delta^s_A \circ \mathsf{c}_A \circ h_A = & \\
        & \:\:\:\: \text{By the definition of a comonad and the reasoning above} & \\
        & \mathsf{m}^s_{A, A} \circ \mathsf{c}_A \circ h_A = \mathsf{m}^s_{A, A} \circ h_A \otimes h_A \circ \gamma_A &
      \end{array}
      \]

Finally, take $s_1 \preceq s_2$ such that $s_1 \in C$, so we have a symmetric lax monoidal comonad morphism $\mu_{s_1, s_2} : \bang_{s_2} \Rightarrow \bang_{s_1}$.
By \cite[Proposition 4.5.9]{borceux1994handbook:vol2}, we have a functor 
$\mathsf{T}_{s_1,s_2} : \mathcal{C}^{\bang_{s_2}} \to \mathcal{C}^{\bang_{s_1}}$ such that
$\mathsf{T}_{s_1,s_2}  : (A, h_A) \mapsto (A, {\mu_{s_1,s_2}}_A \circ h_A)$ and $\mathsf{F}_{s_2}^{\natural} = \mathsf{F}_{s_1}^{\natural} \mathsf{T}_{s_1,s_2}$. 
$\mathsf{T}_{s_1,s_2}$ is also strict symmetric monoidal:
\[
\begin{array}{lll}
  & \mathsf{T}_{s_1,s_2}(A, h_A) \otimes \mathsf{T}_{s_1,s_2}(B, h_B) = (A, {\mu_{s_1,s_2}}_A \circ h_A) \otimes (B, {\mu_{s_1,s_2}}_B \circ h_B) = & \\ 
  & \:\:\:\: (A \otimes B, \mathsf{m}^{s_1}_{A, B} \circ {\mu_{s_1,s_2}}_A \otimes {\mu_{s_1,s_2}}_B \circ h_A \otimes h_B ) = & \\
  & \:\:\:\: (A \otimes B, {\mu_{s_1,s_2}}_{A \otimes B} \circ \mathsf{m}^{s_2}_{A, B} \circ h_A \otimes h_B ) = \mathsf{T}_{s_1,s_2}(A \otimes B, \mathsf{m}^{s_2}_{A, B} \circ h_A \otimes h_B) &
\end{array}
\]

And the corresponding axiom of a $\Sigma$-bouton is satisfied as follows. First of
all, let $\mathsf{copy}^{s_1}$ and $\mathsf{copy}^{s_2}$ denote the copying operations in $\mathcal{C}^{\bang_{s_1}}$
and $\mathcal{C}^{\bang_{s_2}}$.

Let $(A, h_A)$ be a $\bang_{s_2}$-coalgebra, then the copying operation
$\mathsf{copy}^{s_2} : (A, h_A) \to (A \otimes A, \mathsf{m}^{s_2}_{A, A} \circ h_A \otimes h_A)$ is given by:
\begin{equation}~\label{norm:comult}
\begin{tikzcd}
  A \ar[rr, "h_A"] && \bang_{s_2} A \ar[rr, "\mathsf{c}^{s_2}_A"] && \bang_{s_2} A \otimes \bang_{s_2} A \ar[rr, "\varepsilon^{s_2}_A \otimes \varepsilon^{s_2}_A"] && A \otimes A
\end{tikzcd}
\end{equation}

On the one hand, we have $\mathsf{T}_{s_1,s_2}(\mathsf{copy}^{s_2}) : \mathsf{T}_{s_1,s_2}(A, h_A) \to \mathsf{T}_{s_1,s_2}(A \otimes A, \mathsf{m^{s_2}}_{A, A} \circ h_A \otimes h_A)$ with the underlying
arrow as in~(\ref{norm:comult}), but as a morphism of coalgebras
$(A, {\mu_{s_1,s_2}}_A \circ h_A)$ and $(A \otimes A, {\mu_{s_1,s_2}}_{A \otimes A} \circ \mathsf{m}^{s_2}_{A, A} \circ h_A \otimes h_A)$.

On the other hand, we have the arrow $\mathsf{copy}^{s_1}_{\mathsf{T}_{s_1,s_2}(A, h_A)} : (A, {\mu_{s_1,s_2}}_A \circ h_A) \to (A \otimes A, {\mu_{s_1,s_2}}_{A \otimes A} \circ \mathsf{m}^{s_2}_{A, A} \circ h_A \otimes h_A)$ in $\mathcal{C}^{\bang_{s_1}}$
with the underlying arrow:
\begin{equation}
  \begin{tikzcd}
    A \ar[rr, "h_A"] && \bang_{s_1} A \ar[rr, "\mathsf{c}^{s_1}_A"] && \bang_{s_1} A \otimes \bang_{s_1} A \ar[rr, "\varepsilon^{s_1}_A \otimes \varepsilon^{s_1}_A"] && A \otimes A
  \end{tikzcd}
\end{equation}
So the equality $\mathsf{T}_{s_1,s_2}(\mathsf{copy}^{s_2}_{(A, h_A)}) = \mathsf{copy}^{s_1}_{\mathsf{T}_{s_1,s_2}(A, h_A)}$ holds since:
\[
\begin{tikzcd}
  A \ar[dd, equal] \ar[rr, "h_A"] && \bang_{s_2} A \ar[dd, "{\mu_{s_1,s_2}}_A"] \ar[rr, "\mathsf{c}^{s_2}_A"] && \bang_{s_2} A \otimes \bang_{s_2} A \ar[dd, "{\mu_{s_1,s_2}}_A \otimes {\mu_{s_1,s_2}}_A"] \ar[rr, "\varepsilon^{s_2}_A \otimes \varepsilon^{s_2}_A"] && A \otimes A \ar[dd, equal] \\
  \\
  A \ar[rr, "{\mu_{s_1,s_2}}_A \circ h_A"'] && \bang_{s_1} A \ar[rr, "\mathsf{c}^{s_1}_A"'] && \bang_{s_1} A \otimes \bang_{s_1} A \ar[rr, "\varepsilon^{s_1}_A \otimes \varepsilon^{s_1}_A"'] && A \otimes A
\end{tikzcd}
\]
where the left-hand side square commutes tautologically, the right-hand side square follows from the definition of a
symmetric monoidal comonad morphism, whereas the central square is from the definition of a $\Sigma$-assemblage.
\end{proof}

Now we proceed with representing ${\bf Assemblage}_{\Sigma}$ for a fixed subexponential signature $\Sigma$, but we have to define the codomain of the desired embedding accurately. 
The 2-category of strict 2-functors ${\bf Fun}(\Sigma^{op}, {\bf SMC}_{str})$ consists of all possible choices of symmetric monoidal categories indexed
by elements from $\Sigma$ and strict symmetric monoidal functors contravariantly corresponding to the structure of the underlying preorder. Further,
we must choose a symmetric monoidal \emph{closed} category $\mathcal{C}$ and a family of symmetric strict monoidal functors $\mathsf{T}_s \to \mathcal{C}$
for $\mathsf{T} \in {\bf Fun}(\Sigma^{op}, {\bf SMC}_{str})$ having a right adjoint (which is enough for having a symmetric lax monoidal adjunction by Kelly's doctrinal adjunction theorem \cite{kelly2006doctrinal}) 
and respecting the structure of the underlying preorder in $\Sigma$, so we could reconstruct comonads from adjunctions, but also
their morphisms and comonad transformations. This would allow us to completely categorise the intuition that we have a family of symmetric lax monoidal comonads
over a symmetric monoidal category such that their relations between each other are contravariantly mirrored from the structure of $\Sigma$.

Consider the comma 2-category ${\bf Fun}(\Sigma^{op}, {\bf SMC}_{str}) \downarrow \Delta$, where $$\Delta : {\bf SMCC}_{str} \to {\bf Fun}(\Sigma^{op}, {\bf SMC}_{str})$$ is the constant diagram strict 2-functor. 
Let $\text{${\bf pre}$-${\bf Bouton}_{\Sigma}$}$ denote the comma 2-category $${\bf Fun}(\Sigma^{op}, {\bf SMC}_{str}) \downarrow \Delta$$ for brevity.
This 2-category consists of triples $(\mathcal{C}, \mathsf{T}, \mathsf{F}^{\natural} : \mathsf{T} \Rightarrow \Delta \mathcal{C})$ as 0-cells,
where $\mathcal{C}$ is a symmetric monoidal closed category, $\mathsf{T} \in {\bf Fun}(\Sigma^{op}, {\bf SMC}_{str})$ and $\mathsf{F}^{\natural} : \mathsf{T} \Rightarrow \Delta \mathcal{C}$
is a strict 2-natural transformation, actually, a strict 2-cocone. 1-cells are 
$(\mathcal{C}, \mathsf{T}, \mathsf{F}^{\natural} : \mathsf{T} \Rightarrow \Delta \mathcal{C}) \to (\mathcal{D}, \mathsf{T}',\mathsf{F}'^{\natural} : \mathsf{T}' \Rightarrow \Delta \mathcal{D})$
consisting of $\mathsf{G} : \mathcal{C} \to \mathcal{D}$, a 1-cell in ${\bf SMCC}_{str}$, $\mathsf{T} \to \mathsf{T}'$, a 1-cell in ${\bf Fun}(\Sigma^{op}, {\bf SMC}_{str})$, i.e.
a strict 2-natural transformation with components $\mathsf{G}_s : {\mathsf{T}}_s \to {\mathsf{T}'}_s$ such that the following is satisfied for any $s \in \Sigma$:
\[
\begin{tikzcd}
  {\mathsf{T}}_s \ar[d, "\mathsf{G}_s"'] \ar[rr, "{\mathsf{F}^{\natural}}_{{\mathsf{T}}_s}"] && \mathcal{C} \ar[d, "\mathsf{G}"] \\
  {\mathsf{T}}'_s \ar[rr, "{\mathsf{F}'^{\natural}}_{{\mathsf{T}'}_s}"'] && \mathcal{D}
\end{tikzcd}
\]

Further, given $(\mathsf{G}, \mathsf{G}_s), (\mathsf{G}', \mathsf{G}'_s) : (\mathcal{C}, \mathsf{T}, \mathsf{F}^{\natural}) \to (\mathcal{D}, \mathsf{T}',\mathsf{F}'^{\natural})$,
then a 2-cell $\theta : (\mathsf{G}, \mathsf{G}_s) \Rightarrow (\mathsf{G}', \mathsf{G}'_s)$ consists of a symmetric monoidal natural transformation
$\theta : \mathsf{G} \Rightarrow \mathsf{G}'$ and $\theta_s : \mathsf{G}_s \Rightarrow \mathsf{G}'_s$ which are agree with the preorder structure of $\Sigma$
such that the following conditions are satisfied:
\begin{itemize}
\item $\forall s \in \Sigma \:\: \mathsf{F}_s'^{\natural} \circ \theta_s = \theta \circ \mathsf{F}_s^{\natural}$,
\item $\forall s_1, s_2 \in \Sigma \:\: (s_1 \preceq s_2 \Rightarrow \theta_{s_1} \circ \mathsf{T}_{s_1,s_2} = \mathsf{T}'_{s_1,s_2} \circ \theta_{s_2})$.
\end{itemize}

The intuition is that, at the level of 0-cells, the 2-category $\text{${\bf pre}$-${\bf Bouton}_{\Sigma}$}$ consists of diagrams with a collection of symmetric monoidal categories
$(\mathsf{T}_s)_{s \in \Sigma}$ and $\mathcal{C}$ such that $\mathcal{C}$ is also closed, strict symmetric monoidal functors 
$\mathsf{T}_{s_2} \to \mathsf{T}_{s_1}$ whenever $s_1 \preceq s_2$ in $\Sigma$ and strict symmetric monoidal functors 
$\mathsf{F}^{\natural}_s : \mathsf{T}_{s} \to \mathcal{C}$ such that the following is satisfied for $s_1 \preceq s_2$
\[
\begin{tikzcd}
  \mathsf{T}_{s_2} \ar[rr, "\mathsf{T}_{s_1,s_2}"] \ar[dr, "\mathsf{F}^{\natural}_{s_2}"'] && \mathsf{T}_{s_1} \ar[dl, "\mathsf{F}^{\natural}_{s_1}"] \\
  & \mathcal{C}
\end{tikzcd}
\]

Note that one can simplify this construction for some cases. For example, if all the symmetric monoidal categories we consider were monoidally presentable as in Example~\ref{lafont:category} and if all the 
strict monoidal functors were accessible, then it would suffice to extend $\Sigma$ to $\Sigma_{\bullet}$ with the element $\frownie$ such that $\frownie \preceq s$ and for any $s \in \Sigma$ and $\frownie \in -(W \cap C)$, and 
then consider ${\bf Fun}(\Sigma_{\bullet}^{op}, {\bf prSMC}_{str})$ (where ${\bf prSMC}_{str}$ is the 2-category of presentably SMCs, accessible strict symmetric monoidal functors and their natural transformations) 
for the following reasons. For any $\mathsf{T} \in {\bf Fun}(\Sigma_{\bullet}^{op}, {\bf prSMC}_{str})$, $\mathsf{T}_{\frownie}$ would be automatically closed because of the Adjoint Functor theorem. Besides,
all the functors $\mathsf{F}_s^{\natural} : \mathsf{T}_s \to \mathsf{T}_{\frownie}$ for any $s \in \Sigma$ would have right adjoints as any such $\mathsf{F}_s^{\natural}$ preserves all small colimits.

Those diagrams in $\text{${\bf pre}$-${\bf Bouton}_{\Sigma}$}$ reflect how left adjoints commute with $\mathsf{T}_{s_1,s_2}$'s, but we also
expect those $\mathsf{T}_s$ to be relevant or Cartesian or semicartesian or none of them depending on what a particular $s$ is.
But such diagrams bear no information about that. Recall that if $s$ is from $W \setminus C$, then $\mathsf{T}_s$ is not really problematic since $\mathsf{T}_s$ is (semi)Cartesian iff the forgetful functor 
from the category of cocommutative comonoids (from the slice category over $\mathds{1}$) is an isomorphism. 
But if the only thing we know about $s$ is that it is from $C$, then, as we know from Theorem~\ref{coalgebra:relevant}, $\mathsf{T}_s$ is relevant iff it forms a coalgebra in the
Eilenberg-Moore category ${\bf SMC}_{\operatorname{str}}^{\bf coSem}$. So our idea is to consider
$\Sigma$-boutons as coalgebras in the Eilenberg-Moore 2-category
$\text{${\bf pre}$-${\bf Bouton}_{\Sigma}$}^{\mathsf{K}}$ for a strict 2-comonad $\mathsf{K}$
inheriting the corresponding coalgebra structures from ${\bf SMC}_{\operatorname{str}}^{\bf coSem}$ from Theorem~\ref{coalgebra:relevant} for each $\mathsf{T}_s$ whenever $s \in C$.
So, first of all, recall that:
\begin{definition}
  Let $\mathcal{C}$ be a 2-category and let $\mathsf{K} : \mathcal{C} \to \mathcal{C}$ be a strict 2-functor,
  then a (strict) 2-comonad is a triple $(\mathsf{K}, \delta, \varepsilon)$, where
  $\delta$ and $\varepsilon$ are strict 2-natural transformations satisfying the coassociative and coidentity laws.
\end{definition}
There is a more detailed description of (lax) 2-monads in \cite[§6.5]{johnson20212}. If we have a comonad $\mathsf{K}$ over a 2-category $\mathcal{C}$, then we can associate the \emph{Eilenberg-Moore 2-category}
$\mathcal{C}^{\mathsf{K}}$ of strict coalgebras $(A, h_A)$, morphisms and 2-cells. The rest is to construct the desired 2-functor $\mathsf{K}$.

Let us take any $(\mathcal{C}, \mathsf{T}, \mathsf{F}^{\natural})$ and define $(\mathcal{C}, \mathsf{S}, \mathbb{F}^{\natural}) = \mathsf{K} (\mathcal{C}, \mathsf{T}, \mathsf{F}^{\natural})$ by the case analysis 
depending on $s \in \Sigma$:
\begin{itemize}
  \item If $s \in C \setminus W$, then we put $\mathsf{S}_s := {\bf coSem}(\mathsf{T}_s)$ and $\mathbb{F}_s^{\natural} := \mathsf{F}_s^{\natural} \circ \mathsf{U}^{s}_{\bf coSem}$,
  If $s \preceq t$, then we put $\mathsf{S}_{s,t} := {\bf coSem}(\mathsf{T}_{s,t})$,
  \item If $s \in W \setminus C$, then we put $\mathsf{S}_s := \mathsf{T}_s/\mathds{1}$ and $\mathbb{F}_s^{\natural} := \mathsf{F}_s^{\natural} \circ \mathsf{U}^{s}_{\mathds{1}}$,
  If $s \preceq t$, then we put $\mathsf{S}_{s,t} = \mathsf{T}_{s,t}/\mathds{1}$,
  \item If $s \in W \cap C$, then $\mathsf{S}_{s} := {\bf coMon}(\mathsf{T}_{s})$, $\mathbb{F}_s^{\natural} := \mathsf{F}_s^{\natural} \circ \mathsf{U}^{s}_{\bf coMon}$
  and if $s \preceq t$, then we put $\mathsf{S}_{s,t} :=  {\bf coMon}(\mathsf{T}_{s,t})$.
  \item If we are given $s \in C$, $s \preceq t$, but also $t \in W$,
  then we put $\mathsf{S}_{s,t}$ as the following composite:
  \[
  {\bf coMon}(\mathsf{T}_t) \rightarrow {\bf coSem}(\mathsf{T}_t) \xrightarrow{{\bf coSem}(\mathsf{T}_{s,t})} {\bf coSem}(\mathsf{T}_s) 
  \]
  where the left-hand side functor is a strict symmetric monoidal functor dropping counit in every cocommutative comonoid.
  \item If we are given $s \in W$, $s \preceq t$, but also $t \in C$, then we put $\mathsf{S}_{s,t}$ as the following composite:
  \[
  {\bf coMon}(\mathsf{T}_t) \rightarrow {\mathsf{T}_t}/\mathds{1} \xrightarrow{\mathsf{T}_{s,t}/\mathds{1}} {\mathsf{T}_s}/\mathds{1}
  \]
  where the left-hand side functor is a strict symmetric monoidal functor dropping comultiplication in every cocommutative comonoid.
  \item Now let us arrange the transition functors for the configuration when $s \preceq t$ and $s \notin W \cup C$:
  \begin{itemize}
    \item if $t \in C$, then
    \[\mathsf{S}_{s,t} := {\bf coSem}(\mathsf{T}_t) \xrightarrow{\mathsf{U}^t_{\bf coSem}} \mathsf{T}_t \xrightarrow{\mathsf{T}_{s,t}} \mathsf{T}_s,
    \]
    \item if $t \in W$, then
    \[\mathsf{S}_{s,t} := \mathsf{T}_t/\mathds{1} \xrightarrow{\mathsf{U}^t_{\mathds{1}}} \mathsf{T}_t \xrightarrow{\mathsf{T}_{s,t}} \mathsf{T}_s,
    \]
    \item if $t \in W \cap C$, then
    \[\mathsf{S}_{s,t} := {\bf coMon}(\mathsf{T}_t) \xrightarrow{\mathsf{U}^t_{\bf coMon}} \mathsf{T}_t \xrightarrow{\mathsf{T}_{s,t}} \mathsf{T}_s.
    \]
  \end{itemize}
  \item Otherwise, we put $\mathsf{S}_s := \mathsf{T}_s$ and $\mathbb{F}_s^{\natural} := \mathsf{F}_s^{\natural}$, and if $s \preceq t$, put
  $\mathsf{S}_{s,t} = \mathsf{T}_{s,t}$.
\end{itemize}

\begin{lemma}~\label{sigma:comonad}
  The mapping $\mathsf{K}$ that takes $(\mathcal{C}, \mathsf{T}, \mathsf{F}^{\natural}) \in \text{${\bf pre}$-${\bf Bouton}_{\Sigma}$}$ 
  and sends it to $(\mathcal{C}, \mathsf{S}, \mathbb{F}^{\natural})$ via the above construction is a 2-endofunctor on $\text{${\bf pre}$-${\bf Bouton}_{\Sigma}$}$.
  Besides, $\mathsf{K}$ is a strict 2-comonad on the 2-category $\text{${\bf pre}$-${\bf Bouton}_{\Sigma}$}$.
\end{lemma}
\begin{proof} Let $(\mathcal{C}, \mathsf{T}, \mathsf{F}^{\natural}), (\mathcal{D}, \mathsf{T}', \mathsf{F}'^{\natural}) \in \text{${\bf pre}$-${\bf Bouton}_{\Sigma}$}$. 
  Let $f : (\mathcal{C}, \mathsf{T}, \mathsf{F}^{\natural}) \to (\mathcal{D}, \mathsf{T}', \mathsf{F}'^{\natural})$ be a 1-cell in $\text{${\bf pre}$-${\bf Bouton}_{\Sigma}$}$ given by strict symmetric monoidal
  functors $\mathsf{H} : \mathcal{C} \to \mathcal{D}$
  and a family $\mathsf{H}_s : \mathsf{T}_s \to \mathsf{T}'_s$ for $s \in \Sigma$ such that $\mathsf{H}_{s_1} \circ \mathsf{T}_{s_1,s_2} = \mathsf{T}'_{s_1,s_2} \circ \mathsf{H}_{s_2}$ whenever $s_1 \preceq s_2$.
  $\mathsf{K}f : \mathsf{K}(\mathcal{C}, \mathsf{T}, \mathsf{F}^{\natural}) \to \mathsf{K}(\mathcal{D}, \mathsf{T}', \mathsf{F}'^{\natural})$ is given as follows. 
  
  Let $(\mathcal{C}, \mathsf{S}, \mathbb{F}^{\natural})$ and $(\mathcal{D}, \mathsf{S}', \mathbb{F}'^{\natural})$ be 
  $\mathsf{K}(\mathcal{C}, \mathsf{T}, \mathsf{F}^{\natural})$ and $\mathsf{K}(\mathcal{D}, \mathsf{T}', \mathsf{F}'^{\natural})$
  respectively. We define $\mathbb{H} := \mathsf{K}f : (\mathcal{C}, \mathsf{S}, \mathbb{F}^{\natural}) \to (\mathcal{D}, \mathsf{S}', \mathbb{F}'^{\natural})$
  consisting of strict symmetric monoidal functors $\mathsf{H} : \mathcal{C} \to \mathcal{D}$ (the same as above)
  and $\mathbb{H}_s : \mathsf{S}_s \to \mathsf{S}'_s$ defined as follows by the case analysis on $s \in \Sigma$:
  \begin{itemize}
    \item If $s \in C \setminus W$, then we put $\mathbb{H}_s := {\bf coSem}(\mathsf{H}_s)$,
    \item If $s \in W \setminus C$, then we put $\mathbb{H}_s := \mathsf{H}_s/\mathds{1}$,
    \item If $s \in W \cap C$, then we put $\mathbb{H}_s := {\bf coMon}(\mathsf{H}_s)$,
    \item Otherwise, $\mathbb{H}_s := \mathsf{H}_s$.
  \end{itemize}
  Note that every $\mathbb{H}_s$ is strict symmetric monoidal, the argument is similar to \cite[Proposition 3.38]{aguiar2010monoidal}. 
  The functoriality condition for 1-cells is immediate.
  
  Let $f_1, f_2 : (\mathcal{C}, \mathsf{T}, \mathsf{F}^{\natural}) \to (\mathcal{D}, \mathsf{T}', \mathsf{F}'^{\natural})$ be 1-cells in $\text{${\bf pre}$-${\bf Bouton}_{\Sigma}$}$
  given by strict symmetric monoidal functors $\mathsf{G}, \mathsf{H} : \mathcal{C} \to \mathcal{D}$, $\mathsf{G}_s, \mathsf{H}_s : \mathsf{T}_s \to \mathsf{T}'_s$ 
  for $s \in \Sigma$ respecting the structure of the preorder in $\Sigma$. Let $\phi : f_1 \Rightarrow f_2$ be a 2-cell given by 
  strict symmetric monoidal natural transformations $\xi : \mathsf{G} \Rightarrow \mathsf{H}$ and $\xi_{s} : \mathsf{G}_s \Rightarrow \mathsf{H}_s$ for $s \in \Sigma$ such that
  the following is satisfied for $s_1 \preceq s_2$:
  \[
  \begin{tikzcd}
    \mathsf{T}_{s_2} \ar[dd, "\mathsf{T}_{s_1,s_2}"'] \ar[rrr, bend left = 20, "\mathsf{G}_{s_2}", ""{name=gs2}] \ar[rrr, bend right = 20, "\mathsf{H}_{s_2}"', ""{name=hs2}] &&& \mathsf{T}'_{s_2} \ar[dd, "\mathsf{T}'_{s_1,s_2}"] \\
    \\
    \mathsf{T}_{s_1} \ar[rrr, bend left = 20, "\mathsf{G}_{s_1}", ""{name=gs1}] \ar[rrr, bend right = 20, "\mathsf{H}_{s_1}"', ""{name=hs1}] &&& \mathsf{T}'_{s_1}
    \arrow[from=gs2,to=hs2,Rightarrow, "\xi_{s_2}"]
    \arrow[from=gs1,to=hs1,Rightarrow, "\xi_{s_1}"]
  \end{tikzcd}
  \]
  so $\mathsf{K}\phi : \mathsf{K}f_1 \Rightarrow \mathsf{K}f_2$ consists of strict symmetric monoidal natural transformations 
  $\xi : \mathsf{G} \Rightarrow \mathsf{H}$ and $\widetilde{\xi_{s}}$ for $s \in \Sigma$, which are defined by the case analysis on $s \in \Sigma$.
  Let us write down the clause for $s \in C$. In this case we put: $\widetilde{\xi_{s}} := {\bf coSem}({\xi_s}) : {\bf coSem}(\mathsf{G}_s) \Rightarrow {\bf coSem}(\mathsf{H}_s)$ such that the following is satisfied:
  \[
  \begin{tikzcd}
    {\bf coSem}({\mathsf{T}}_{s}) \ar[rrrr, bend left = 15, "{\bf coSem}(\mathsf{G}_{s})", ""{name=hatgs}] \ar[dd, "\mathsf{U}^{{\mathsf{T}}_{s}}_{{\bf coSem}}"'] \ar[rrrr, bend right = 15, "{\bf coSem}(\mathsf{H}_{s})"', ""{name=haths}] &&&& {\bf coSem}({\mathsf{T}'}_s)  \ar[dd, "\mathsf{U}^{{\mathsf{T}'}_s}_{{\bf coSem}}"] \\
    \\
    {\mathsf{T}}_{s} \ar[rrrr, bend left = 15, "\mathsf{G}_{s}", ""{name=gs}] \ar[rrrr, bend right = 15, "\mathsf{H}_{s}"', ""{name=hs}] &&&& {\mathsf{T}'}_{s}
    \arrow[from=hatgs,to=haths,Rightarrow, "\widetilde{\xi_{s}}"]
    \arrow[from=gs,to=hs,Rightarrow, "\xi_s"]
  \end{tikzcd}
  \]
  So the following diagram commutes for $s_1 \preceq s_2$ and $s_1 \in C$:
  \[
  \begin{tikzcd}
    {\bf coSem}(\mathsf{T}_{s_2}) \ar[dd, "{\bf coSem}(\mathsf{T}_{s_1,s_2})"']  \ar[rrrr, bend left = 13, "{\bf coSem}(\mathsf{G}_{s_2})", ""{name=hatgs2}] \ar[rrrr, bend right = 13, "{\bf coSem}(\mathsf{H}_{s_2})"', ""{name=haths2}] &&&& {\bf coSem}(\mathsf{T}'_{s_2}) \ar[dd, "{\bf coSem}(\mathsf{T}'_{s_1,s_2})"]  \\
    \\
    {\bf coSem}(\mathsf{T}_{s_1})  \ar[rrrr, bend left = 13, "{\bf coSem}(\mathsf{G}_{s_1})", ""{name=hatgs1}] \ar[rrrr, bend right = 13, "{\bf coSem}(\mathsf{H}_{s_1})"', ""{name=haths1}] &&&& {\bf coSem}(\mathsf{T}'_{s_1}) \\
    \arrow[from=hatgs2,to=haths2,Rightarrow, "\widetilde{\xi_{s_2}}"]
    \arrow[from=hatgs1,to=haths1,Rightarrow, "\widetilde{\xi_{s_1}}"]
  \end{tikzcd}
  \]
  The case of other subexponential indices is handled similarly, and after making all those constructions one can ensure $\mathsf{K}\phi : \mathsf{K}f_1 \Rightarrow \mathsf{K}f_2$ 
  is a well-defined 2-cell in $\text{${\bf pre}$-${\bf Bouton}_{\Sigma}$}$.

  We have yet to define the counit and comultiplication and check the axioms of 2-comonad. The comonad structure on $\mathsf{K}$ is inherited from 
  the comonad constructed in Lemma~\ref{cosem:relevant}.  The counit natural transformation $\varepsilon : \mathsf{K} \Rightarrow {\bf Id}$ consists of the components 
  $\varepsilon_{(\mathcal{C}, \mathsf{T}, \mathsf{F}^{\natural})} : \mathsf{K}(\mathcal{C}, \mathsf{T}, \mathsf{F}^{\natural}) \to (\mathcal{C}, \mathsf{T}, \mathsf{F}^{\natural})$ for any
  0-cell from $\text{${\bf pre}$-${\bf Bouton}_{\Sigma}$}$ that
  turns $\mathsf{K}(\mathcal{C}, \mathsf{T}, \mathsf{F}^{\natural})$ into $(\mathcal{C}, \mathsf{T}, \mathsf{F}^{\natural})$ by connecting ${\bf coSem}(\mathsf{T}_s)$ in $\mathsf{K}(\mathcal{C}, \mathsf{T}, \mathsf{F}^{\natural})$ for $s \in C$ and $\mathsf{T}_s$ in $(\mathcal{C}, \mathsf{T}, \mathsf{F}^{\natural})$
  with the forgetful functor $\mathsf{U}^{s}_{\bf coSem}$ and by making similar connections for other subexponential cases. 
  Then one checks that every $\varepsilon_{(\mathcal{C}, \mathsf{T}, \mathsf{F}^{\natural})} : \mathsf{K}(\mathcal{C}, \mathsf{T}, \mathsf{F}^{\natural}) \to (\mathcal{C}, \mathsf{T}, \mathsf{F}^{\natural})$
  is a well-defined 1-cell in $\text{${\bf pre}$-${\bf Bouton}_{\Sigma}$}$.

  We define the comultiplication similarly $\delta : \mathsf{K} \Rightarrow \mathsf{K}^2$. For $s \in C$, we connect ${\bf coSem}(\mathsf{T}_s)$ in $\mathsf{K}(\mathcal{C}, \mathsf{T}, \mathsf{F}^{\natural})$ for $s \in C$ with 
  ${\bf coSem}^2(\mathsf{T}_s)$ in $\mathsf{K}^2(\mathcal{C}, \mathsf{T}, \mathsf{F}^{\natural})$
  with the corresponding comultiplication component $\delta^{\bf coSem}_{\mathsf{T}_s} : {\bf coSem}(\mathsf{T}_s) \to {\bf coSem}^2(\mathsf{T}_s)$ from Lemma~\ref{cosem:relevant}. In other cases,
  such as the affine and the exponential ones, we already know that an SMC is Cartesian (semicartesian) iff it is isomorphic to its category of cocommutative comonoids (its slice category over $\mathds{1}$), see \cite[Corollary 19]{mellies2009categorical},
  so the corresponding ``comultiplications'' are trivial. The axioms of a 2-comonad are inherited from the comonad laws verified for ${\bf coSem}$ in Lemma~\ref{cosem:relevant}.
\end{proof}

Now we show that any $\Sigma$-assemblage forms a coalgebra in the Eilenberg-Moore 2-category over the comonad $\mathsf{K}$ constructed above.

\begin{lemma}~\label{cocteau:coalgebra}
  Any $\Sigma$-assemblage $(\mathcal{C}, (\bang_s)_{s \in \Sigma}, (\mu_{s_1,s_2})_{s_1 \preceq s_2})$ forms a coalgebra in the Eilenberg-Moore 2-category $\text{${\bf pre}$-${\bf Bouton}_{\Sigma}$}^{\mathsf{K}}$.
\end{lemma}

\begin{proof}
  By Proposition~\ref{cocteau:bouton}, we can unveil a $\Sigma$-assemblage $(\mathcal{C}, (\bang_s)_{s \in \Sigma}, (\mu_{s_1,s_2})_{s_1 \preceq s_2})$ as a 
  $\Sigma$-bouton $(\mathcal{C}, \mathsf{T}, (\mathsf{F}^{\natural}_s)_{s \in \Sigma})$, where $\mathsf{F}^{\natural}$
  are left adjoints from cofree-coforgetful decompositions $\mathsf{F}_s^{\natural} \dashv \mathsf{F}_s^{\sharp}$
  and $\mathsf{T}$ is a strict 2-functor sending every $s$ to the Eilenberg-Moore category $\mathcal{C}^{\bang_s}$
  and $s_1 \preceq s_2$ to the $\mathsf{T}_{s_1,s_2} : \mathcal{C}^{\bang_{s_2}} \to \mathcal{C}^{\bang_{s_1}}$
  transforming coalgebras via $\mu_{s_1,s_2}$ such that $\mathsf{F}_{s_2}^{\natural} = \mathsf{F}_{s_1}^{\natural} \circ \mathsf{T}_{s_1,s_2}$.
  Therefore, we have a well-defined 1-cell in $\text{${\bf pre}$-${\bf Bouton}_{\Sigma}$}$.
  The strict coalgebra coaction $h_{\mathcal{C}} : (\mathcal{C}, \mathsf{T}, \mathsf{F}^{\natural}) \to \mathsf{K}(\mathcal{C}, \mathsf{T}, \mathsf{F}^{\natural})$ is defined as follows
  with strict symmetric monoidal functors $1_{\mathcal{C}}$ and $\mathbb{h}_s$ for $s \in \Sigma$, which are defined by the case analysis.
  \begin{itemize}
    \item If $s \in C \setminus W$, then the comonad $\bang_s$ is relevant, and thus, by Proposition~\ref{cocteau:bouton}, the Eilenberg-Moore category 
    $\mathcal{C}^{\bang_s}$ is relevant, and therefore we put 
    $\mathbb{h}_s := \mathsf{V}_{\mathcal{C}^{\bang_s}} : \mathcal{C}^{\bang_s} \to {\bf coSem}(\mathcal{C}^{\bang_s})$, the strict symmetric monoidal functor
    from Theorem~\ref{coalgebra:relevant}.
    \item If $s \in W \setminus C$, then the comonad $\bang_s$ is affine, and therefore, the Eilenberg-Moore category 
    $\mathcal{C}^{\bang_s}$ is semicartesian, and thus $\mathcal{C}^{\bang_s} \cong \mathcal{C}^{\bang_s}/\mathds{1}$, so $\mathbb{h}_s$ is the isomorphism functor.
    \item If $s \in W \cap C$, then the comonad $\bang_s$ is exponential, and therefore, the Eilenberg-Moore category 
    $\mathcal{C}^{\bang_s}$ is Cartesian, and thus $\mathcal{C}^{\bang_s} \cong {\bf coMon}(\mathcal{C}^{\bang_s})$, so $\mathbb{h}_s$ is the isomorphism functor.
    \item Otherwise, we proceed with the original data.
  \end{itemize}
  The conditions of a strict coalgebra for $h_{\mathcal{C}}$ are inherited from Theorem~\ref{coalgebra:relevant}.
\end{proof}

Transforming $\Sigma$-assemblages into such strict 2-coalgebras is of interest for two reasons. First, when we decompose resource modalities with adjunctions,
we describe their characteristic properties diagrammatically with isomorphisms and strict sections without going into the content of given symmetric monoidal categories.
Second, this construction explicitly says that the coalgebraic nature of $\Sigma$-assemblages is completely inherited from the coalgebraic criterion 
for relevant categories from Theorem~\ref{coalgebra:relevant}.

Before formulating the representation theorem, let us make an extra technical remark. Note that all the functors between $\Sigma$-assemblages
we consider are strict as monoidal functors, whereas in Proposition~\ref{coalgebra:theorem:pr} all 1-cells are lax by default.
So we formulate the following proposition refining that theorem:
\begin{prop}~\label{strict:fun:comonad}
  Let $\mathcal{C}, \mathcal{D}$ be SMCs and let $\mathsf{S}$ and $\mathsf{T}$ be symmetric lax monoidal comonads on 
  $\mathcal{C}$ and $\mathcal{D}$ respectively. Let $\mathsf{F} : \mathcal{C} \to \mathcal{D}$ be a symmetric lax monoidal functor
  inducing a comonad morphism $(\mathcal{C}, \mathsf{S}) \to (\mathcal{D}, \mathsf{T})$, then
  $\mathsf{F}$ is strict if and only if $\widetilde{\mathsf{F}} : \mathcal{C}^{\mathsf{S}} \to \mathcal{D}^{\mathsf{T}}$ is strict.
\end{prop}
\begin{proof}
  The fact is almost a tautology. 
  
  $\widetilde{\mathsf{F}}(A, h_A) \otimes \widetilde{\mathsf{F}}(B, h_B)$ and 
  $\widetilde{\mathsf{F}}((A, h_A) \otimes (B, h_B))$ have carriers $F(A) \otimes F(B)$ and $F(A \otimes B)$,
  but $\widetilde{\mathsf{F}}$ is strict and strictness of $\widetilde{\mathsf{F}}$
  is strictness of $\mathsf{F}$ on carriers since the coforgetful functor is strict and faithful and reflects the monoidal structure.

  The other way round implication is obvious.
\end{proof}

Note that 0-cells of the Eilenberg-Moore 2-category $\text{${\bf pre}$-${\bf Bouton}_{\Sigma}$}^{\mathsf{K}}$ do not have enough information for embedding ${\bf Assemblage}_{\Sigma}$
fully faithfully. Let $\text{${\bf pre}$-${\bf Bouton}_{\Sigma}$}_R^{\mathsf{K}}$ denote the full 1,2-subcategory of $\text{${\bf pre}$-${\bf Bouton}_{\Sigma}$}^{\mathsf{K}}$
spanned by those coalgebras, whose carriers $(\mathcal{C}, \mathsf{T}, \mathsf{F}^{\natural})$ as 0-cells in $\text{${\bf pre}$-${\bf Bouton}_{\Sigma}$}$ have the following property. 
For any $s \in \Sigma$, each $\mathsf{F}_s^{\natural}$ has a right adjoint, so the adjunction $\mathsf{F}^{\natural}_s \dashv \mathsf{F}^{\sharp}_s$
lifts to the adjunction in ${\bf SMC}$ by \cite{mellies2009categorical}[Proposition 13].

Now we are ready to formulate and show the final theorem of this section.

\begin{theorem}~\label{cocteau:2cat}
  ${\bf Assemblage}_{\Sigma}$ embeds 1,2-fully faithfully into $\text{${\bf pre}$-${\bf Bouton}_{\Sigma}$}_R^{\mathsf{K}}$.
\end{theorem}
\begin{proof}
  Any $\Sigma$-assemblage $(\mathcal{C}, (\bang_s)_{s \in \Sigma}, (\mu_{s_1,s_2})_{s_1 \preceq s_2})$ has a $\mathsf{K}$-coalgebra 
  $(\mathcal{C}, \mathsf{T}, \mathsf{F}^{\natural}, h_{\mathcal{C}})$ by Lemma~\ref{cocteau:coalgebra}.

  Let $(\mathcal{C}, (\bang_s)_{s \in \Sigma}, (\mu_{s_1,s_2})_{s_1 \preceq s_2})$ and $(\mathcal{D}, (\bang'_s)_{s \in \Sigma}, (\mu'_{s_1,s_2})_{s_1 \preceq s_2})$ be $\Sigma$-assemblages and 
  let $(\mathcal{C}, \mathsf{T}, \mathsf{F}^{\natural}, h_{\mathcal{C}})$ and 
  $(\mathcal{D}, \mathsf{T}', \mathsf{F}'^{\natural}, h_{\mathcal{D}})$ be their $\mathsf{K}$-coalgebras from Lemma~\ref{cocteau:coalgebra}.
  Let us first show that a 1-cell $\mathsf{H} \in {\bf Assemblage}_{\Sigma}(\mathcal{C}, \mathcal{D})$, inducing a 1-cell
  $\mathsf{H}_{\bullet} : (\mathcal{C}, \mathsf{T}, \mathsf{F}^{\natural}, h_{\mathcal{C}}) \to (\mathcal{D}, \mathsf{T}', \mathsf{F}'^{\natural}, h_{\mathcal{D}})$ in 
  $\text{${\bf pre}$-${\bf Bouton}_{\Sigma}$}$,
  is also a coalgebra morphism $\mathsf{H}_{\bullet} : (\mathcal{C}, \mathsf{T}, \mathsf{F}^{\natural}, h_{\mathcal{C}}) \to (\mathcal{D}, \mathsf{T}', \mathsf{F}'^{\natural}, h_{\mathcal{D}})$.
  Recall that $\mathsf{H}_{\bullet}$ is a 1-cell in $\text{${\bf pre}$-${\bf Bouton}_{\Sigma}$}$
  given by functors $\mathsf{H} : \mathcal{C} \to \mathcal{D}$ and $(\mathsf{H}_s : \mathcal{C}^{\bang_s} \to \mathcal{D}^{\bang'_s})_{s \in \Sigma}$.
  Further, if $s \in C$, the following square commutes by Theorem~\ref{coalgebra:relevant}
  \[
  \begin{tikzcd}
     \mathcal{C}^{\bang_s} \ar[rr, "\mathsf{H}_s"] \ar[d, "\mathsf{V}_{\mathcal{C}^{\bang_s}}"'] && \mathcal{D}^{\bang'_s} \ar[d, "\mathsf{V}_{\mathcal{D}^{\bang'_s}}"] \\
    {\bf coSem}(\mathcal{C}^{\bang_s}) \ar[rr, "{\bf coSem}(\mathsf{H}_s)"'] && {\bf coSem}(\mathcal{D}^{\bang'_s})
  \end{tikzcd}
  \]
  After handling all those subexponential indices appropriately, we ensure that the following square commutes:
  \[
  \begin{tikzcd}
    (\mathcal{C}, \mathsf{T}, \mathsf{F}^{\natural}) \ar[rr, "\mathsf{H}_{\bullet}"] \ar[d, "h_{\mathcal{C}}"'] && (\mathcal{D}, \mathsf{T}', \mathsf{F}'^{\natural}) \ar[d, "h_{\mathcal{D}}"]\\
    \mathsf{K} (\mathcal{C}, \mathsf{T}, \mathsf{F}^{\natural}) \ar[rr,"\mathsf{K} \mathsf{H}_{\bullet}"'] && \mathsf{K}(\mathcal{D}, \mathsf{T}', \mathsf{F}'^{\natural})
  \end{tikzcd}
  \]

  Checking that $\mathsf{H} \mapsto \mathsf{H}_{\bullet}$ is functorial for 1-cells is immediate ($\mathsf{H}_{\bullet}$ as coalgebra morphisms, not only as morphisms of $\Sigma$-boutons).

  2-cells are handled similarly. Let $\mathsf{G}, \mathsf{H} \in {\bf Assemblage}_{\Sigma}(\mathcal{C}, \mathcal{D})$ and let
  $\theta : \mathsf{G} \Rightarrow \mathsf{H}$ be a 2-cell in ${\bf Assemblage}_{\Sigma}$. $\mathsf{G}$ and $\mathsf{H}$ induce 1-cells 
  $\mathsf{G}_{\bullet}$ and $\mathsf{H}_{\bullet}$, which are also coalgebra morphisms in $\text{${\bf pre}$-${\bf Bouton}_{\Sigma}$}_R^{\mathsf{K}}$.
  We must ensure that a 2-cell $\theta_{\bullet} : \mathsf{G}_{\bullet} \Rightarrow \mathsf{H}_{\bullet}$ is also a 2-cell in  $\text{${\bf pre}$-${\bf Bouton}_{\Sigma}$}_R^{\mathsf{K}}$. 
  The key observation is that the following diagram commutes in \text{${\bf pre}$-${\bf Bouton}_{\Sigma}$}
  for $s \in C$ by Proposition~\ref{coalgebra:theorem:pr} and Proposition~\ref{strict:fun:comonad}:
  \[
  \begin{tikzcd}
    \mathcal{C}^{\bang_s} \ar[dd, "\mathsf{V}_{\mathcal{C}^{\bang_s}}"'] \ar[rrrr, bend left=12, "\mathsf{G}_s", ""{name=gs}] \ar[rrrr, bend right=12, "\mathsf{H}_s"', ""{name=hs}] &&&& \mathcal{D}^{\bang'_s} \ar[dd, "\mathsf{V}_{\mathcal{D}^{\bang'_s}}"] \\
    \\
    {\bf coSem}(\mathcal{C}^{\bang_s}) \ar[rrrr, bend left=12, "{\bf coSem}(\mathsf{G}_s)", ""{name=cosemgs}] \ar[rrrr, bend right=12, "{\bf coSem}(\mathsf{H}_s)"', ""{name=cosemhs}] &&&& {\bf coSem}(\mathcal{D}^{\bang'_s})
    \ar[from=gs,to=hs, Rightarrow, "\theta_s"]
    \ar[from=cosemgs,to=cosemhs, Rightarrow, "{\bf coSem}(\theta_s)"]
  \end{tikzcd}
  \]
  The other subexponential indices are handled similarly.
  
  The mapping 
  \[
  \Omega : (\mathcal{C}, (\bang_s)_{s \in \Sigma}, (\mu_{s_1,s_2})_{s_1 \preceq s_2}) \mapsto (\mathcal{C}, \mathsf{T}, \mathsf{F}^{\natural}, h_{\mathcal{C}})
  \] is 1,2-fully faithful, so we must ensure that $\Omega$ is bijective on both
  1-cells and 2-cells. Let us first ensure that $\Omega$ is injective for functors and natural transformations.
  Let $\mathsf{G}, \mathsf{H} : \mathcal{C} \to \mathcal{D}$ be 1-cells in ${\bf Assemblage}_{\Sigma}$ such that 
  $\Omega\mathsf{G} = \Omega\mathsf{H}$. We must ensure that $\mathsf{G} = \mathsf{H}$.
  But this clause is tautological since the equality $\Omega\mathsf{G} = \Omega\mathsf{H}$ means the equality of constituent strict symmetric monoidal functors, in particular,
  $\mathsf{G} = \mathsf{H}$. We also have $\mathsf{G}_s = \mathsf{H}_s$ for any $s \in \Sigma$ by $\mathsf{G} = \mathsf{H}$, so the equality of the constituent symmetric monoidal comonad morphisms
  from $(\mathcal{C}, \bang_s)$ to $(\mathcal{D}, \bang'_s)$ is obtained from $\mathsf{G}_s$ and $\mathsf{H}_s$ respectively. 2-cells are handled similarly, so we omit the corresponding part of the proof.

  Now we show that $\Omega$ is full for both 1-cells and 2-cells. Assume we are given a 1-cell
  $\mathsf{H}_{\bullet} : (\mathcal{C}, \mathsf{T}, \mathsf{F}^{\natural}, h_{\mathcal{C}}) \to (\mathcal{D}, \mathsf{T}', \mathsf{F}'^{\natural}, h_{\mathcal{D}})$ in $\text{${\bf pre}$-${\bf Bouton}_{\Sigma}$}_R^{\mathsf{K}}$ 
  with a carrier $\mathsf{H} : \mathcal{C} \to \mathcal{D}$ where $(\mathcal{C}, \mathsf{T}, \mathsf{F}^{\natural}, h_{\mathcal{C}}) = \Omega(\mathcal{C}, (\bang_s)_{s \in \Sigma}, (\mu_{s_1,s_2})_{s_1 \preceq s_2})$
  and $(\mathcal{D}, \mathsf{T}', \mathsf{F}'^{\natural}, h_{\mathcal{D}}) = \Omega(\mathcal{D}, (\bang'_s)_{s \in \Sigma}, (\mu'_{s_1,s_2})_{s_1 \preceq s_2})$. We must ensure that $\mathsf{H}_{\bullet}$ forms a 1-cell between $\mathcal{C}$ and $\mathcal{D}$
  in ${\bf Assemblage}_{\Sigma}$, that is, $\mathsf{H}$ satisfies the corresponding part of Definition~\ref{2:cat:cocteau}.
  We consider the case of relevant comonads, the other cases are considered similarly. First of all, the following square commutes:
    \[
    \begin{tikzcd}
      \mathcal{C}^{\bang_s} \ar[rr, "\mathsf{H}_s"] \ar[d, "\mathsf{F}^{\natural}_{\mathcal{C}^{\bang_s}}"']  && \mathcal{D}^{\bang'_s} \ar[d, "\mathsf{F}'^{\natural}_{\mathcal{D}^{\bang'_s}}"] \\
      \mathcal{C} \ar[rr, "\mathsf{H}"'] && \mathcal{D}
    \end{tikzcd}
    \]
    so, by Proposition~\ref{coalgebra:theorem:pr}, \cite[Proposition 12]{mellies2009categorical} and Proposition~\ref{strict:fun:comonad}, we have a symmetric lax monoidal comonad morphism 
    $(\mathsf{H}, \widetilde{\mathsf{H}_s}) : (\mathcal{C}, \bang_s) \Rightarrow (\mathcal{D}, \bang'_s)$
    with the component $\widetilde{\mathsf{H}_s} : \mathsf{H} \bang_s \Rightarrow \bang'_s \mathsf{H}$ induced by $\mathsf{H}_s$. Let us check the rest of the conditions of 1-cells in ${\bf Assemblage}_{\Sigma}$.
    Further, we must ensure that~(\ref{cocteau:1cell:axiom1}) is satisfied for any $s_1 \preceq s_2$ in $\Sigma$.

    On the one hand, we have the following left morphism of adjunctions:
    \[
    \begin{tikzcd}
      \mathcal{C}^{\bang_{s_2}} \ar[rr, "\mathsf{H}_{s_2}"] \ar[dd, "\mathsf{H}^{\natural}_{\mathcal{C}^{\bang_{s_2}}}"', ""{name=Il}] && \mathcal{D}^{\bang'_{s_2}} \ar[dd, "\mathsf{H}'^{\natural}_{\mathcal{D}^{\bang'_{s_2}}}"', ""{name=Il'}]  \ar[rr, "{\mathsf{T}'}_{s_1,s_2}"] && \mathcal{D}^{\bang'_{s_1}} \ar[dd, "\mathsf{H}'^{\natural}_{\mathcal{D}^{\bang'_{s_1}}}", ""{name=Rl'}] \\
      \\
      \mathcal{C} \ar[rr, "\mathsf{H}"']  && \mathcal{D} \ar[rr, equal] && \mathcal{D}
    \end{tikzcd}
    \]
    which gives us the comonad morphism $\mu'_{s_1,s_2} \mathsf{H} \circ \widetilde{\mathsf{H}_s}$ by Proposition~\ref{coalgebra:theorem:pr}, \cite[Proposition 12]{mellies2009categorical} and Proposition~\ref{strict:fun:comonad}.
    
    On the other hand, we have:
    \[
    \begin{tikzcd}
    \mathcal{C}^{\bang_{s_2}} \ar[rr, "{\mathsf{T}}_{s_1,s_2}"] \ar[dd, "\mathsf{H}^{\natural}_{\mathcal{C}^{\bang_{s_2}}}"', ""{name=Il}] &&  \mathcal{C}^{\bang_{s_1}} \ar[dd, "\mathsf{H}^{\natural}_{\mathcal{C}^{\bang_{s_1}}}"', ""{name=Rl}] \ar[rr, "\mathsf{H}_{s_1}"] && \mathcal{D}^{\bang'_{s_1}} \ar[dd, "\mathsf{H}'^{\natural}_{\mathcal{D}^{\bang'_{s_1}}}", ""{name=Rl'}] \\
    \\
    \mathcal{C} \ar[rr, equal] && \mathcal{C} \ar[rr, "\mathsf{H}"'] && \mathcal{D}
    \end{tikzcd}
    \]
    which, in turn, induces the comonad morphism $\mathsf{H}_{s_1} \circ \mathsf{H} \mu_{s_1,s_2}$ by the same statements. Finally, the following square holds by the definition of
    a 1-cell in $\text{${\bf pre}$-${\bf Bouton}_{\Sigma}$}_R^{\mathsf{K}}$:
    \[
    \begin{tikzcd}
      \mathcal{C}^{\bang_{s_2}} \ar[dd, "{\mathsf{T}}_{s_1,s_2}"'] \ar[rr, "\mathsf{H}_{s_2}"] && \mathcal{D}^{\bang'_{s_2}} \ar[dd, "{\mathsf{T}'}_{s_1,s_2}"] \\ 
      \\
      \mathcal{C}^{\bang_{s_1}} \ar[rr, "\mathsf{H}_{s_1}"'] && \mathcal{D}^{\bang'_{s_1}}
    \end{tikzcd}
    \]
    Thus, the composites $\mathsf{H}_{s_1} \circ \mathsf{H} \mu_{s_1,s_2}$ and $\mu'_{s_1,s_2} \mathsf{H} \circ \mathsf{H}_{s_2}$ are equal.
    
    Further, we observe that the functor $\mathsf{H}_{s} : \mathcal{C}^{\bang_{s}} \to \mathcal{D}^{\bang'_{s}}$ preserves the copying operations for any $s \in C$.
    Let $(A, h_A)$ be a coalgebra in $\mathcal{C}^{\bang_s}$, then we have a coalgebra 
    $(\mathsf{H}A, {\mathsf{H}_s}_A \circ \mathsf{H}(h_A))$ in $\mathcal{D}^{\bang'_{s}}$. 
    On the other hand, we take $(A, h_A)$and endow it with 
    the copying operation $\gamma_A$ as follows:
    \[
    \begin{tikzcd}
      A \ar[r, "h_A"] & \bang_{s} A \ar[r, "\mathsf{c}^{s}_A"] & \bang_{s} A \otimes \bang_{s} A \ar[rr, "\varepsilon^s_A \otimes \varepsilon^s_A"] && A \otimes A
    \end{tikzcd}
    \]
    
    Then we endow the coalgebra $(\mathsf{H}A, {\mathsf{H}_s}_A \circ \mathsf{H}(h_A))$ with the copying operation $\gamma_{\mathsf{H}A} : \mathsf{H}A \to \mathsf{H}A \otimes \mathsf{H}A$ as follows:
    \begin{equation}~\label{comul1}
    \begin{tikzcd}
      \mathsf{H}A \ar[r, "\mathsf{H}(h_A)"] & \mathsf{H}(\bang_{s} A) \ar[r, "{\mathsf{H}_{s}}_A"] & \bang'_{s} \mathsf{H}A \ar[r, "\mathsf{c}'^s_{\mathsf{H}A}"] & \bang'_{s} \mathsf{H}A \otimes \bang'_{s} \mathsf{H}A  \ar[rr, "\varepsilon'^{s}_{\mathsf{H}A} \otimes \varepsilon'^{s}_{\mathsf{H}A}"] && \mathsf{H}A \otimes \mathsf{H}A
    \end{tikzcd}
    \end{equation}
    But those composites are equal by Theorem~\ref{coalgebra:relevant}.

    Now we establish that we have fullness for 2-cells as well. Let $\mathsf{G}, \mathsf{H} \in {\bf Assemblage}_{\Sigma}(\mathcal{C}, \mathcal{D})$, 
    so $\mathsf{G}$ and $\mathsf{H}$ induce coalgebra morphisms 
    $\mathsf{G}_{\bullet}, \mathsf{H}_{\bullet} : (\mathcal{C}, \mathsf{T}, \mathsf{F}^{\natural}, h_{\mathcal{C}}) \to (\mathcal{D}, \mathsf{T}', \mathsf{F}'^{\natural}, h_{\mathcal{D}})$
    and let $\theta : \mathsf{G} \Rightarrow \mathsf{H}$ be a 2-cell in
    $\text{${\bf pre}$-${\bf Bouton}_{\Sigma}$}_R^{\mathsf{K}}$. We must ensure that $\theta$ provides a 2-cell in ${\bf Assemblage}_{\Sigma}$.
    First of all, we have, for any $s \in \Sigma$:
    \begin{equation}
    \begin{tikzcd}
      \mathcal{C}^{\bang_s} \ar[dd, "\mathsf{F}^{\natural}_{\mathcal{C}^{\bang_s}}"'] \ar[rrr, bend left=15, "\mathsf{G}_{s}", ""{name=Fr}] \ar[rrr, bend right=15, "\mathsf{H}_{s}"', ""{name=Gr}] &&& \mathcal{D}^{\bang'_{s}} \ar[dd, "\mathsf{F}'^{\natural}_{\mathcal{D}^{\bang'_{s}}}"] \\
      \\
      \mathcal{C} \ar[rrr, bend left=15, "\mathsf{G}", ""{name=F}] \ar[rrr, bend right=15, "\mathsf{H}"', ""{name=G}] &&& \mathcal{D}
      \arrow[from=Fr, to=Gr, Rightarrow, "\theta_{s}"]
      \arrow[from=F, to=G, Rightarrow, "\theta"]
    \end{tikzcd}
    \end{equation}
    We already know that $\mathsf{G}$ and $\mathsf{H}$ are equipped with symmetric lax monoidal comonad morphisms
    $(\mathsf{G}, \widetilde{\mathsf{G}_s}), (\mathsf{H}, \widetilde{\mathsf{H}_s}) : (\mathcal{C}, \bang_s) \Rightarrow  (\mathcal{D}, \bang'_s)$
    with the components $\widetilde{\mathsf{G}_s} : \mathsf{G} \bang_s \Rightarrow \bang'_s \mathsf{G}$ and 
    $\widetilde{\mathsf{H}_s} : \mathsf{H} \bang_s \Rightarrow \bang'_s \mathsf{H}$ for $s \in \Sigma$.
    By Proposition~\ref{coalgebra:theorem:pr}, \cite[Proposition 12]{mellies2009categorical}  and Proposition~\ref{strict:fun:comonad}, therefore $\theta$ induces symmetric lax monoidal comonad morphism transformation
    $(\mathsf{G}, \widetilde{\mathsf{G}_s}) \Rightarrow (\mathsf{H}, \widetilde{\mathsf{H}_s})$.
    To finally check~(\ref{final:axiom}), observe that the following square already commutes as $\theta$ is a 2-cell in $\text{${\bf pre}$-${\bf Bouton}_{\Sigma}$}$ for any $s_1 \preceq s_2$:
    \begin{equation}~\label{base}
      \begin{tikzcd}
        \mathcal{C}^{\bang_{s_2}} \ar[dd, "{\mathsf{T}}_{s_1,s_2}"'] \ar[rrr, bend left=15, "\mathsf{G}_{s_2}", ""{name=Fi}] \ar[rrr, bend right=15, "\mathsf{H}_{s_2}"', ""{name=Gi}] &&& \mathcal{D}^{\bang'_{s_2}} \ar[dd, "{\mathsf{T}'}_{s_1,s_2}"] \\
        \\
        \mathcal{C}^{\bang_{s_1}} \ar[rrr, bend left=15, "\mathsf{G}_{s_1}", ""{name=Fr}] \ar[rrr, bend right=15, "\mathsf{H}_{s_1}"', ""{name=Gr}] &&& \mathcal{D}^{\bang'_{s_1}}
      \arrow[from=Fi, to=Gi, Rightarrow, "\theta_{s_2}"]
      \arrow[from=Fr, to=Gr, Rightarrow, "\theta_{s_1}"]
      \end{tikzcd}
    \end{equation}
    
    Now consider the following prism:
    \begin{equation}~\label{prism}
    \begin{tikzcd}
      |[alias=Ci]| \mathcal{C}^{\bang_{s_2}} &&&& |[alias=Di]| \mathcal{D}^{\bang'_{s_2}} \\
      \\
      && |[alias=C]| \mathcal{C} &&&& |[alias=D]| \mathcal{D} \\
      \\
      |[alias=Cr]| \mathcal{C}^{\bang_{s_1}} &&&& |[alias=Dr]| \mathcal{D}^{\bang'_{s_1}} 
    \arrow[from=Ci, to=Di, "\mathsf{G}_{s_2}", bend left=15, ""{name=Fi}]
    \arrow[from=Ci, to=Di, "\mathsf{H}_{s_2}"', bend right=15, ""{name=Gi}]
    \arrow[from=Fi, to=Gi, Rightarrow, "\theta_{s_2}"']
    \arrow[from=Cr, to=Dr, "\mathsf{G}_{s_1}", bend left=15, ""{name=Fr}]
    \arrow[from=Cr, to=Dr, "\mathsf{H}_{s_1}"', bend right=15, ""{name=Gr}]
    \arrow[from=Fr, to=Gr, Rightarrow, "\theta_{s_1}"']
    \arrow[from=C, to=D, bend left=15,  "\mathsf{G}\:\:\:\:\:\:\:", ""{name=F}]
    \arrow[from=C, to=D, bend right=15, "\mathsf{H}\:\:\:\:\:\:\:"', ""{name=G}]
    \arrow[from=F, to=G, Rightarrow, shift right=3, "\theta"']
    \arrow[from=Ci, to=C, "\mathsf{F}^{\natural}_{\mathcal{C}^{\bang_{s_2}}}"]
    \arrow[from=Cr, to=C, "\mathsf{F}^{\natural}_{\mathcal{C}^{\bang_{s_1}}}"]
    \arrow[from=Di, to=D, "\mathsf{F}'^{\natural}_{\mathcal{D}^{\bang'_{s_2}}}"]
    \arrow[from=Dr, to=D, "\mathsf{F}'^{\natural}_{\mathcal{D}^{\bang'_{s_1}}}"']
    \arrow[from=Ci, to=Cr, "{\mathsf{T}}_{s_1,s_2}"']
    \arrow[from=Di, to=Dr, crossing over, "{\mathsf{T}'}_{s_1,s_2}"]
    \end{tikzcd}
    \end{equation}
    where lateral triangles commute since $\mathcal{C}$ and $\mathcal{D}$ induce $\Sigma$-boutons by Proposition~\ref{cocteau:bouton}, and those triangles
    are the corresponding left adjunction morphisms.
    The front square already commutes since it is exactly the same square as~(\ref{base}). 
    The bottom square holds for $\theta_{s_1} : \mathsf{G}_{s_1} \Rightarrow \mathsf{H}_{s_1}$ between $\mathsf{G}_{s_1}, \mathsf{H}_{s_1} : \mathcal{C}^{\bang_{s_1}} \to \mathcal{D}^{\bang'_{s_1}}$
    since $\mathsf{G}_{s_1}$ and $\mathsf{H}_{s_1}$ are symmetric lax monoidal comonad morphisms $(\mathcal{C}, \bang_{s_1}) \Rightarrow (\mathcal{D}, \bang'_{s_1})$ induced by the strict symmetric monoidal functors $\mathsf{G}$
    and $\mathsf{H}$. So in the bottom square, $\theta_{s_1}$ is connected with $\theta$ with the corresponding coforgetful functors from the comonad decompositions of $\bang_{s_1}$
    and $\bang'_{s_1}$. The back square commutes similarly but for comonads $(\mathcal{C}, \bang_{s_2})$ and $(\mathcal{D}, \bang'_{s_2})$.
    So we conclude~(\ref{final:axiom}) by applying Proposition~\ref{coalgebra:theorem:pr}, \cite[Proposition 12]{mellies2009categorical} and Proposition~\ref{strict:fun:comonad}.
\end{proof}
\section{Realisability Interpretation of ${\bf SILL}(\lambda)_{\Sigma_3}$}~\label{realisability}

Realisability was initially introduced by Kleene to concretise and formalise BHK-semantics of intuitionistic logic by means of computability theory \cite{kleene1945interpretation} \cite{kleene1969formalized}.
More generally, the intuition behind realisability is to equip proofs with the computational procedures implementing the process of a proof as an algorithm.
In constructive mathematics, realisability is connected with such topics as program extraction and the disjunction and the existential properties. 
One can also categorify the idea of realisability and consider the category of assemblies, that is, the category of sets whose elements, informally, have computational evidence,
extracted from, for example, the Kleene partial combinatory algebra $\mathcal{K}_1$ or other partial or total applicative structures, see \cite{longley1995realizability} and \cite{van2008realizability}.
Besides, one can provide the realisability interpretation in this way for simply typed lambda calculus, System F, the calculus of constructions and even MLTT, 
see \cite{reus1999realizability} and \cite{bezem2018realizability}.

There are also several approaches to realisability in linear logic developing the idea of assemblies over linear combinatory algebras 
(the concept coined in \cite{abramsky2002geometry}) to give the realisability interpretation of linear logic in the appropriate categories of assemblies, modest sets and partial equivalence relations, see \cite{abramsky2005linear} and \cite{hoshino2007linear}.
In this section, we develop Hoshino's approach to linear realisability and expand the concepts of a (relational) linear combinatory algebra to the subexponential
case with two extra modalities corresponding to the relevant and affine resource policies. 
This allows for concretising the intuition of keeping multiple resource consumption strategies within a single setting with the actual computational content by 
obtaining realisability models of ${\bf SILL}(\lambda)_{\Sigma_3}$.

\subsection{Background on applicative structures}

A \emph{partial applicative structure} is a structure $(A, \cdot)$ where $A$ is a set and $\cdot : A \times A \rightharpoonup A$ is a partial binary operation.
If $\cdot$ is total, then we call $(A, \cdot)$ just an \emph{applicative structure}. 
\begin{definition} A \emph{partial combinatory algebra} (see, e.g., \cite{feferman2006language}) (PCA) is a partial applicative structure $(A, \cdot)$ such that there are $s, k \in A$ such that:
  \begin{itemize}
    \item $\forall x \in A \:\: s \cdot x \downarrow$
    \item $\forall x, y \in A \:\: s \cdot x \cdot y \downarrow$,
    \item $\forall x,y,z \in A \:\: s \cdot x \cdot y \cdot z \simeq (x \cdot z) \cdot (y \cdot z)$,
    \item $\forall x, \in A \:\: k \cdot x \downarrow $
    \item $\forall x,y \in A \:\: k \cdot x \cdot y \simeq x$.
  \end{itemize}
  where $x \cdot y \downarrow$ means that the value of $x \cdot y$ is defined and $\simeq$ means that if the left-hand side of the equation is defined,
  then the right-hand side of the equation is defined, and they are equal.

  A total PCA is simply a \emph{combinatory algebra}.
\end{definition}

The canonical example of a PCA is the Kleene PCA $\mathcal{K}_1 = (\omega, \cdot)$, where $m \cdot n$ means that a partial recursive function with a G\"{o}del number $m$ is applied to the input $n$.

\begin{definition} Let $(A, \cdot)$ be an applicative structure, then $(A, \cdot)$ is called a \emph{BCI-algebra} if it has $i,b,c \in A$ such that the following identities are 
  satisfied for any $x,y,z \in A$:
  \begin{itemize}
    \item $i \cdot x = x$,
    \item $b \cdot x \cdot y \cdot z = x \cdot (y \cdot z)$,
    \item $c \cdot x \cdot y \cdot z = (x \cdot z) \cdot y$.
  \end{itemize}
  A BCI algebra is \emph{BCK} (\emph{BCW}) if it has $k \in A$ ($w \in A$) such that $k \cdot x \cdot y = x$ ($w \cdot x \cdot y = x \cdot y \cdot y$).
\end{definition}

Let $\mathcal{A}$ be an applicative structure and let $\{ x_i \: | \: i < \omega \}$ be a set of individual variables.
A \emph{polynomial} over $\mathcal{A}$ is an expression generated by $x_i$'s, elements of $A$ and application.
$\mathcal{A}[x_1, \ldots, x_n]$ is the set of all polynomials in $x_1, \ldots, x_n$. 
Clearly, if $\mathcal{A}$ is a BCI (BCK, BCW) algebra or a PCA, so is $\mathcal{A}[x_1, \ldots, x_n]$.

\begin{prop}~\label{abstraction:lemma}
  $ $

  \begin{enumerate}
    \item Let $\mathcal{P} = (P, \cdot)$ be a PCA and let $M \in \mathcal{P}[x, x_1, \ldots, x_n]$
    be a polynomial over $\mathcal{P}$, then there exists a polynomial 
    $\langle x \rangle_{\bf i} M \in \mathcal{P}[x_1, \ldots, x_n]$ (read as ``intuitionistic abstraction'') such that for each $a \in \mathcal{P}$
    $\langle x \rangle M \cdot a \simeq M[x := a]$.
    \item Let $\mathcal{L} = (L, \cdot)$ be a BCI-algebra and let $M \in \mathcal{L}[x, x_1, \ldots, x_n]$
    be a polynomial over $\mathcal{L}$ such that $x$ occurs freely in $M$ exactly once, then
    there exists a polynomial $\langle x \rangle M$ such that for each $a \in \mathcal{L}$
    $\langle x \rangle_{\bf i} M \cdot a = M[x := a]$.
    \item Let $\mathcal{A} = (A, \cdot)$ be a BCK-algebra and let $M \in \mathcal{A}[x, x_1, \ldots, x_n]$
    be a polynomial over $\mathcal{A}$ such that $x$ occurs freely in $M$ at most once, then
    there exists a polynomial $\langle x \rangle_{\bf a} M$ (read as ``affine abstraction'')  such that for each $a \in \mathcal{A}$
    $\langle x \rangle_{\bf a} M \cdot a = M[x := a]$.
    \item Let $\mathcal{R} = (R, \cdot)$ be a BCW-algebra and let $M \in \mathcal{R}[x, x_1, \ldots, x_n]$
    be a polynomial over $\mathcal{R}$ such that $x$ occurs freely in $M$ at least once, then
    there exists a polynomial $\langle x \rangle_{\bf r} M$ (read as ``relevant abstraction'')  such that for each $a \in \mathcal{R}$ $\langle x \rangle_{\bf r} M \cdot a = M[x := a]$.
  \end{enumerate}
\end{prop}

\begin{proof} The fact is standard, but there are some nuances to discuss. Some ideas are taken from \cite[§1-2]{hindley1989bck}.
  This is what one has in a combinatory algebra $\mathcal{C}$:

  \begin{itemize}
    \item $\langle x \rangle_{\bf i} x = (s \cdot k) \cdot k$,
    \item $\langle x \rangle_{\bf i} M = k \cdot M$ if $M$ is a polynomial containing no $x$ as a free variable,
    \item $\langle x \rangle _{\bf i}(M_1 \cdot M_2) = s \cdot (\langle x \rangle_{\bf i} M_1) \cdot (\langle x \rangle_{\bf i} M_2)$.
  \end{itemize}

  In the case of BCI-algebras, we have the following modification:

  \begin{enumerate}
    \item $\langle x \rangle x := i$,
    \item $\langle x \rangle (M \cdot x) := M$ if $M$ does not contain $x$ as a free variable.
    \item 

    $\langle x \rangle (M_1 \cdot M_2) := \begin{cases}
      (b \cdot M_1) \cdot (\langle x \rangle M_2), \text{if $x \in \operatorname{FV}(M_2) \setminus \operatorname{FV}(M_1)$ and 
      $x \not\eqcirc M_2$.} \\
      (c \cdot (\langle x \rangle M_1)) \cdot M_2, \text{if $x \in \operatorname{FV}(M_1) \setminus \operatorname{FV}(M_2)$ and
      $x \not\eqcirc M_1$.}
    \end{cases}$
  \end{enumerate}

  The above defined construction is modified in the following way for the BCK case since the body of an abstraction might not contain the occurrence
  of a variable bound by abstraction:
  \begin{itemize}
    \item $\langle x \rangle_{\bf a} M := k \cdot M$ if $M$ is a polynomial containing no $x$ as a free variable,
  \end{itemize}

  In the BCW case, the abstraction $\langle x \rangle_{\bf r} M$ is defined whenever $x$ occurs in $M$ at least once.
  The clauses for $x$ itself and for applications $M_1 \cdot M_2$ are the same as in the case of BCI. When $x$ occurs in both $M_1$ and $M_2$,
  we put $\langle x \rangle_{\bf r} (M_1 \cdot M_2) := s \cdot (\langle x \rangle_{\bf r} M_1) \cdot (\langle x \rangle_{\bf r} M_2)$, as $s$ is derivable in BCW algebras.
  But there is no clause for $x \not\in \operatorname{FV}(M)$, and this is the difference between $\langle x \rangle_{\bf i}$ and $\langle x \rangle_{\bf r}$.
\end{proof}

We say that an applicative structure \emph{admits abstraction} when it satisfies the condition of Proposition~\ref{abstraction:lemma}.
We write $\hat{x}. M$ whenever we need to refer to abstraction in an applicative structure without specifying
what exactly kind of abstraction we use.

Let $\mathcal{P} = (P, \cdot)$ be an applicative structure admitting abstraction and let $m, n \in \mathcal{P}$, let
\[
m \bullet n := \hat{x}. x \cdot m \cdot n.
\]
We also write ${\bf let} \: x \bullet y = z \: {\bf in} \: F(x,y)$ for $z \cdot (\hat{x}. \hat{y}. F(x,y))$ where $F(x,y)$ is an element of $\mathcal{P}[x,y]$.
Therefore we have:
\[
{\bf let} \: x \bullet y = m \bullet n \: {\bf in} \: F(x,y) \: := F[x:=m,y:=n]
\]

\begin{prop}~\label{bck:projections} Let $\mathcal{P} = (P, \cdot)$ be an applicative structure admitting abstraction, then

  \begin{enumerate}
  \item Assume $\mathcal{P}$ is BCK. Let $m_1, m_2 \in \mathcal{P}$, then there are
  $p_1, p_2 \in \mathcal{P}$ such that $p_i \cdot (m_1 \bullet m_2) = m_i$ for $i = 1, 2$.
  \item Assume $\mathcal{P}$ is BCW. Let $m \in \mathcal{P}$, then there is $\operatorname{copy} \in \mathcal{P}$
  such that $\operatorname{copy} \cdot m = m \bullet m$.
  \end{enumerate}
\end{prop}

\begin{proof} For the BCK case, put $p_1 := \langle x \rangle_{\bf a} x \cdot k$ and $p_2 := \langle x \rangle_{\bf a} \:\: x \cdot (k \cdot i)$.
  For the BCW case, define $\operatorname{copy}$ as $\operatorname{copy} \cdot m := (w \cdot (\langle x \rangle \langle y \rangle x \bullet y)) \cdot m$ for each $m \in \mathcal{P}$.
\end{proof}

The concept of an applicative morphism was introduced by Longley in \cite[§2.1]{longley1995realizability} as a total relation between applicative structures
robustly lifting to the functor between categories of assemblies, so if we have an applicative morphism from $\mathcal{P}$ to $\mathcal{Q}$, 
then we have a witness in $\mathcal{Q}$ that tracks a morphism as well as the application operation in $\mathcal{P}$. The intuition behind 
the idea of an applicative morphism is that we read $\gamma \subseteq \mathcal{P} \times \mathcal{Q}$ as the interpretation of $\mathcal{P}$ in $\mathcal{Q}$, 
and such interpretations are in general not functions, but relations, since $p \in \mathcal{P}$ could have multiple interpretations in $\mathcal{Q}$.

\begin{definition} Let $\mathcal{P} = (P, \cdot), \mathcal{Q} = (Q, \cdot)$ be applicative structures.
  An \emph{applicative morphism} from $\mathcal{P}$ to $\mathcal{Q}$ is a total relation $\gamma \subseteq P \times Q$ such that
  \[
  \exists r \in \mathcal{Q} \: \forall a, b \in \mathcal{P} \: \forall a' \in \gamma(a) \: \forall b' \in \gamma(b) \:\: (r \cdot a') \cdot b' \in \gamma(a \cdot b).
  \]

  $r$ is said to be a \emph{realiser} for $\gamma$ or $\gamma$ is \emph{realised} by $r$. An applicative morphism $\gamma$ from $\mathcal{P}$ to $\mathcal{Q}$ is \emph{functional} if $\gamma$ is a function.
    An \emph{operator} on an applicative structure $\mathcal{P} = (P, \cdot)$ is a functional applicative morphism $\gamma : \mathcal{P} \to \mathcal{P}$.
  
    Let $\gamma_1$ and $\gamma_2$ be applicative morphisms from $\mathcal{P}$ to $\mathcal{Q}$,
  then $\gamma_1 \sqsubseteq \gamma_2$ if there exists $r \in \mathcal{Q}$ such that
  for each $a \in \mathcal{P}$ and for each $q \in \gamma_1(a)$ one has $r \cdot q \in \gamma_2(a)$.

  Let $\mathcal{P} = (P, \cdot)$ be an applicative structure and let $\gamma : \mathcal{P} \to \mathcal{P}$ be an applicative endomorphism.
  Then $\gamma$ is called \emph{comonadic} if $\gamma \sqsubseteq {\bf id}_{\mathcal{P}}$ and $\gamma \sqsubseteq \gamma^2$.
\end{definition}

The concept of an adjoint pair of applicative morphisms mirrors the ordinary concept of adjoint functors, and it was also introduced by Longley for applicative morphisms 
between PCAs \cite[§2.5]{longley1995realizability}, but it can be formulated for a broader class of applicative structures.

\begin{definition}
  Let $\mathcal{P} = (P, \cdot), \mathcal{Q} = (Q, \cdot)$ be applicative structures and
  let $\gamma_1 : \mathcal{P} \to \mathcal{Q}$ and $\gamma_2 : \mathcal{Q} \to \mathcal{P}$ be applicative morphisms.
  Then $\gamma_1$ and $\gamma_2$ form an \emph{adjoint pair}, denoted as $\gamma_1 \dashv \gamma_2$, if
  $\gamma_1 \circ \gamma_2 \sqsubseteq 1_{\mathcal{Q}}$ and $1_{\mathcal{P}} \sqsubseteq \gamma_2 \circ \gamma_1$. Notation: $\gamma_1 \dashv \gamma_2 : \mathcal{Q} \to \mathcal{P}$.
\end{definition}

Clearly, if we have applicative structures $\mathcal{P}$ and $\mathcal{Q}$ such that
there is an adjoint pair $\gamma_1 \dashv \gamma_2 : \mathcal{Q} \to \mathcal{P}$, then $\gamma_1 \circ \gamma_2$ is a comonadic applicative morphism
$\mathcal{Q} \to \mathcal{Q}$.

\begin{definition} Let $\mathcal{P} = (P, \cdot)$ be an applicative structure. 
  \begin{itemize}
  \item An \emph{assembly} on $\mathcal{P}$ is a structure
  $(X, E)$ where $X$ is a set and $E$ is a total relation from $X$ to $\mathcal{P}$.

  \item A morphism of assemblies $\mathcal{X} = (X, E_1)$ and $\mathcal{Y} = (Y, E_2)$ is a map $f : X \to Y$ 
  such that there is $a \in \mathcal{P}$ such that for each $x \in X$ and for each $b \in \mathcal{P}$ 
  such that $b \in E_1(x)$ one has $a \cdot b \in E_2(f(x))$. 
  In this case we say that $a$ \emph{tracks} $f$ and $f$ is a realisable map.
  \end{itemize}
  Let $\operatorname{Asm}(\mathcal{P})$ denote the category of assemblies and their realisable maps.
\end{definition}

Modest sets are those assemblies whose elements do not share common realisers.
\begin{definition}
Let $\mathcal{X} = (X, E)$ be an assembly over an applicative structure $\mathcal{P}$,
then $\mathcal{X}$ is called a \emph{modest set} if for any different $x, y \in X$ $E(x)$ and $E(y)$ are disjoint.

$\operatorname{Mod}(\mathcal{P})$ is the category of modest sets and their realisable maps.
\end{definition}

One can check that $\operatorname{Mod}(\mathcal{P})$, the category of modest sets,
is a full reflective subcategory of $\operatorname{Asm}(\mathcal{P})$.

\begin{prop}
  
  Let $\mathcal{P}$ be an applicative structure, then
  \begin{enumerate}
    \item If $\mathcal{P}$ is BCI, then $\operatorname{Asm}(\mathcal{P})$ is an SMCC,
    \item If $\mathcal{P}$ is BCK, then  $\operatorname{Asm}(\mathcal{P})$ is a semicartesian SMCC,
    \item If $\mathcal{P}$ is BCW, then $\operatorname{Asm}(\mathcal{P})$ is a relevant SMCC,
    \item If $\mathcal{P}$ is a PCA, then $\operatorname{Asm}(\mathcal{P})$ is Cartesian closed.
  \end{enumerate}
\end{prop}

\begin{proof}
  One can find the idea of the proof in \cite[Proposition 4]{hoshino2007linear}, 
  so we just outline the sketch for substructural applicative structures.

   In any clause $\operatorname{Asm}(\mathcal{P})$ is an SMCC, the symmetric monoidal closed structure is given as follows. 
   For $\mathcal{X} = (X, E_1), \mathcal{Y} = (Y, E_2) \in 
   \Ob{\operatorname{Asm}(\mathcal{P})}$, their tensor product is defined as $\mathcal{X} \otimes \mathcal{Y} = (X \times Y, E)$ where, for $x \in X$, $y \in Y$
      and $a_1, a_2 \in \mathcal{P}$, if $a_1 \in E_1(x)$ and $a_2 \in E_2(y)$, then $a_1 \bullet a_2 \in E(x,y)$.
      The carrier of $[\mathcal{X}, \mathcal{Y}]$ is the set $\{ f : X \to Y \: | \: \text{$f$ is realisable} \}$.
      $\mathds{1} = ( \{ *\}, E)$ where $E(*) = \{ i \}$.
    
    Assume that $\mathcal{P}$ is BCK. Let $\mathcal{X}_1, \mathcal{X}_2 \in \Ob{\operatorname{Asm}(\mathcal{P})}$, then we have
    $\mathcal{X}_1 \otimes \mathcal{X}_2 \xrightarrow{\pi_i} X_i$,
    whose underlying maps are $\pi_i : (x_1, x_2) \mapsto x_i$ and $\pi_i$ is realised by $p_i$ from Proposition~\ref{bck:projections}.
    Let $\mathcal{X} = (X, E) \in \Ob{\operatorname{Asm}(\mathcal{P})}$. $\mathds{1}$ is obviously terminal in this case.
    
    For the BCW case, the copying map $\operatorname{copy}_{\mathcal{X}} : \mathcal{X} \to \mathcal{X} \otimes \mathcal{X}$
    is realised by $\operatorname{copy} \in \mathcal{P}$ from Proposition~\ref{bck:projections}.
    
    In \cite[Theorem 1.5.2]{van2008realizability} or \cite[§3.1]{reus1999realizability}, one can find the proof of the last clause.
\end{proof}

If $\mathcal{P}$ is an applicative structure admitting abstraction, then we shall refer to 
$\operatorname{Asm}(\mathcal{P})$ as \emph{a symmetric monoidal closed category of assemblies}.

The category of modest sets in a symmetric monoidal closed category of assemblies has nice properties that we take from \cite[Proposition 5, Lemma 1]{hoshino2007linear}.
Let $\mathcal{P}$ be an applicative structure admitting abstraction, then the inclusion functor $\iota : \operatorname{Mod}(\mathcal{P}) \hookrightarrow \operatorname{Asm}(\mathcal{P})$ has a left adjoint
$\Delta : \operatorname{Asm}(\mathcal{P}) \to \operatorname{Mod}(\mathcal{P})$ such that $ \operatorname{Mod}(\mathcal{P})$ is a reflective full subcategory of $\operatorname{Asm}(\mathcal{P})$.
Besides, if $\mathcal{X} \in \Ob{\operatorname{Mod}(\mathcal{P})}$ and $\mathcal{Y} \in \Ob{\operatorname{Asm}(\mathcal{P})}$, then there is an isomorphism
$[\mathcal{Y}, \iota(\mathcal{X})] \xrightarrow{\cong} \iota \Delta[\mathcal{Y}, \iota(\mathcal{X})]$.

Further, applicative morphisms of BCI algebras (as well as their expansions and PCAs) lift to symmetric lax monoidal functors in the following way.
Let $\gamma : \mathcal{L}_1 \to \mathcal{L}_2$ be an applicative morphism of BCI algebras, 
then $\gamma_* : \operatorname{Asm}(\mathcal{L}_1) \to \operatorname{Asm}(\mathcal{L}_2)$ is a functor that sends $(X, E_X)$ in 
$\operatorname{Asm}(\mathcal{L}_1)$ to $(X, \gamma(E_X))$ in $\operatorname{Asm}(\mathcal{L}_2)$ and a morphism $f : X \to Y$ to 
$\gamma_* f : \gamma_*X \to \gamma_* Y$ with $|f|$ as the underlying map. Let $\mathcal{L}_1$ and $\mathcal{L}_2$ be BCI algebras and let $\gamma : \mathcal{L}_1 \to \mathcal{L}_2$ be an applicative morphism,
then $\gamma_* : \operatorname{Asm}(\mathcal{L}_1) \to \operatorname{Asm}(\mathcal{L}_2)$ is a symmetric lax monoidal functor 
with the components
\begin{center}
   $m_{\mathcal{X}, \mathcal{Y}} : \gamma_* \mathcal{X} \otimes \gamma_* \mathcal{Y} \to \gamma_* (\mathcal{X} \otimes \mathcal{Y})$

   $m_{\mathds{1}} : \mathds{1} \to \gamma_* \mathds{1}$
\end{center}
Let $r$ be a realiser for $\gamma$ and let $a \in \gamma(\hat{x}. \hat{y}. x \bullet y)$.
Then one can show that $\hat{x}. \hat{y}. r \cdot (r \cdot a \cdot x) \cdot y$ realises 
$m_{\mathcal{X}, \mathcal{Y}}$. One can also show that $m_{\mathcal{X}, \mathcal{Y}}$ is natural. $m_{\mathds{1}}$ is realised by 
$\hat{x}. x \cdot i'$ where $i' \in \gamma(i)$.

If $\gamma_1 \sqsubseteq \gamma_2$, we can define a natural transformation $\alpha_* : {\gamma_1}_* \Rightarrow {\gamma_2}_*$,
whose underlying maps ${\alpha_*}_X : {\gamma_1}_* X \to {\gamma_2}_* X$ are identities realised by the realiser of $\gamma_1 \sqsubseteq \gamma_2$. 
Besides, $\alpha_*$ is a symmetric lax monoidal natural transformation.

Moreover, we have:
\begin{prop}~\label{symmetric:monoidal:adjunctions:combintory}
  Let $\mathcal{L}_1$ and $\mathcal{L}_2$ be BCI algebras and let $\gamma_1 : \mathcal{L}_1 \to \mathcal{L}_2$ and $\gamma_2 : \mathcal{L}_2 \to 
  \mathcal{L}_1$ such that $\gamma_1 \dashv \gamma_2$, then the adjoint pair $\gamma_1 \dashv \gamma_2$ lifts to the symmetric monoidal adjunction ${\gamma_1}_* \dashv {\gamma_2}_*$ between the corresponding
  categories of assemblies.
\end{prop}

\begin{proof}
  As we have discussed above, both ${\gamma_1}_* : \operatorname{Asm}(\mathcal{L}_1) \to \operatorname{Asm}(\mathcal{L}_2)$
  and ${\gamma_2}_* : \operatorname{Asm}(\mathcal{L}_2) \to \operatorname{Asm}(\mathcal{L}_1)$ are symmetric lax monoidal functors.
  ${\gamma_1}_* \dashv {\gamma_2}_*$ as plain functors by \cite[Proposition 2.5.7]{longley1995realizability},
  so we have natural transformations $\eta_* : 1_{\operatorname{Asm}(\mathcal{L}_1)} \Rightarrow {\gamma_2}_* {\gamma_1}_*$ and
  $\varepsilon_* : {\gamma_1}_* {\gamma_2}_* \Rightarrow 1_{\operatorname{Asm}(\mathcal{L}_2)}$.
  $\eta_*$ and $\varepsilon_*$ are also symmetric lax monoidal natural transformations, which is evident.
\end{proof}

The concept of a relational subexponential linear combinatory algebra is a combinatory counterpart of a $\Sigma_3$-assemblage, which is also an expansion
of Hoshino's relational linear combinatory algebra with a single exponential $\bang$.

\begin{definition}
  Let $\mathcal{L} = (L, \cdot)$ be a BCI algebra and let $\bang_{\bf i}, \bang_{\bf a}, \bang_{\bf r} : \mathcal{L} \to \mathcal{L}$
  be comonadic applicative endomorphisms with realisers $r_{\bf i}$, $r_{\bf a}$ and $r_{\bf r}$, then a structure $(\mathcal{L}, \bang_{\bf i}, \bang_{\bf a}, \bang_{\bf r})$
  is called a \emph{relational subexponential linear combinatory algebra} if the following axioms are satisfied:
  \begin{multicols}{2}
  \begin{itemize}
    \item $\bang_{\bf i} \sqsubseteq \bang_{\bf i} \bullet \bang_{\bf i}$,
    \item $\bang_{\bf i} \sqsubseteq (x \mapsto \{ i \})$.
    \item $\bang_{\bf r} \sqsubseteq \bang_{\bf r} \bullet \bang_{\bf r}$,
    \item $\bang_{\bf i} \sqsubseteq \bang_{\bf r}$,
    \item $\bang_{\bf i} \sqsubseteq \bang_{\bf a}$,
    \item $\bang_{\bf a} \sqsubseteq (x \mapsto \{ i \})$.
  \end{itemize}
\end{multicols}
where $\bang_s \bullet \bang_s$ for any $s$ is an applicative morphism $(\bang_s \bullet \bang_s)(p) = \{ u \bullet v \: | \: u,v \in \bang_s(p) \}$ realised as:
\[
\langle x \rangle \langle y \rangle \:\: {\bf let} \: x_1 \bullet x_2 = x \: {\bf in} \: {\bf let} \: y_1 \bullet y_2 = y \: {\bf in} \: r_s \cdot x_1 \cdot y_1 \bullet r_s \cdot x_2 \cdot y_2.
\]
where $r_s$ is a realiser for $\bang_s$.
\end{definition}

The following statement is a combinatory counterpart of Proposition~\ref{cocteau:bouton}, in particular, that part of it that allows one forming an assemblage
from a bouton, but from adjoint pairs between applicative structures:

\begin{prop}~\label{combinatory:bouton}
  Given a PCA $\mathcal{P}$, a BCI-algebra $\mathcal{L}$, a BCK-algebra $\mathcal{A}$ and a BCW-algebra $\mathcal{R}$. Given adjoint pairs
  $f^{\natural}_{\bf i} \dashv f^{\sharp}_{\bf i} : \mathcal{L} \to \mathcal{P}$, $f^{\natural}_{\bf a} \dashv f^{\sharp}_{\bf a} : \mathcal{L} \to \mathcal{A}$ and 
  $f^{\natural}_{\bf r} \dashv f^{\sharp}_{\bf r} : \mathcal{L} \to \mathcal{R}$ such that there are applicative morphisms $\mu_{\bf ir} :  \mathcal{P} \to \mathcal{R}$
  and $\mu_{\bf ia} :  \mathcal{P} \to \mathcal{A}$ with realisers $r_{\bf ir}$ and $r_{\bf ia}$ such that $f^{\natural}_{\bf i} = f^{\natural}_{\bf r} \circ \mu_{\bf ir}$
  and $f^{\natural}_{\bf i} = f^{\natural}_{\bf a} \circ \mu_{\bf ia}$. Let $\bang_s = f^{\natural}_s f^{\sharp}_s$ for every $s$, then $(\mathcal{L}, (\bang_s)_{s \in \{ {\bf i}, {\bf r}, {\bf a}\}})$ is a relational subexponential linear combinatory algebra.
\end{prop}
\begin{proof}
  Each $\bang_s$ is a comonadic applicative morphism. Let $\varepsilon_s$ be a realiser for $f^{\natural}_s f^{\sharp}_s \sqsubseteq 1_{\mathcal{L}}$
  and let $\eta_s$ be a realiser for $1 \sqsubseteq f^{\sharp}_s f^{\natural}_s$. Now we must check the conditions of a relational subexponential linear combinatory algebra.
  $\bang_{\bf i} \sqsubseteq \bang_{\bf i} \bullet \bang_{\bf i}$ and $\bang_{\bf i} \sqsubseteq (x \mapsto \{ i \})$ are checked similarly to \cite[Proposition 8]{hoshino2007linear},
  so one can port similar proofs for $\bang_{\bf r} \sqsubseteq \bang_{\bf r} \bullet \bang_{\bf r}$ and $\bang_{\bf a} \sqsubseteq (x \mapsto \{i \})$.
  So the only claims one needs to prove are $\bang_{\bf i} \sqsubseteq \bang_{\bf r}$ and $\bang_{\bf i} \sqsubseteq \bang_{\bf a}$. We consider only the former clause
  since the latter is completely identical.

  The idea of the proof is inspired by Proposition~\ref{2:comonad:morphish:prop}. First of all, observe that $\mu_{\bf ir} \circ f_{\bf i}^{\sharp} \sqsubseteq f_{\bf r}^{\sharp}$. 
  \[
    \mu_{\bf ir} f_{\bf i}^{\sharp} \sqsubseteq f_{\bf r}^{\sharp} f_{\bf r}^{\natural} \mu_{\bf ir} f_{\bf i}^{\sharp}  = f_{\bf r}^{\sharp} f_{\bf i}^{\natural} f_{\bf i}^{\sharp} \sqsubseteq f_{\bf r}^{\sharp} \\
  \]

  $\bang_{\bf i} \sqsubseteq \bang_{\bf r}$ means $\bang_{\bf i} = f^{\natural}_{\bf i} f^{\sharp}_{\bf i} = f^{\natural}_{\bf r} \mu_{\bf ir} f^{\sharp}_{\bf i} \sqsubseteq f^{\natural}_{\bf r} f^{\sharp}_{\bf r}$,
  but $f^{\natural}_{\bf r} \mu_{\bf ir} f^{\sharp}_{\bf i} \sqsubseteq f^{\natural}_{\bf r} f^{\sharp}_{\bf r}$ as $\sqsubseteq$ respects composition 
  by \cite[Proposition 1.6.2]{van2008realizability}.
\end{proof}

Subexponential linear combinatory algebras are an important subcase of relational subexponential linear combinatory algebras that admit an equational axiomatisation, and, besides,
allow for extracting combinatory, BCK and BCW subalgebras from the corresponding subexponentials.

\begin{definition}
  Let $\mathcal{L} = (L, \cdot)$ be a BCI algebra and let $\bang_{\bf i}, \bang_{\bf a}, \bang_{\bf r} : \mathcal{L} \to \mathcal{L}$ be operators on $\mathcal{L}$,
  then a structure $(\mathcal{L}, \bang_{\bf i}, \bang_{\bf a}, \bang_{\bf r})$ is a \emph{subexponential linear combinatory algebra} if
  there are elements $k_{\bf i}$, $k_{\bf a}$, $w_{\bf i}$, $w_{\bf r}$, $(\varepsilon_s)_{s \in \{{\bf i}, {\bf a}, {\bf r} \}}$, $(\delta_s)_{s \in \{{\bf i}, {\bf a}, {\bf r} \}}$,
  $(\mathsf{m}^{s}_{s_1,s_2})_{s \preceq s_1,s_2 \in \{{\bf i}, {\bf a}, {\bf r} \}} \in \mathcal{L}$ such that the following identities are satisfied:
  \begin{multicols}{2}
    \begin{itemize}
      \item $\mathsf{m}^{s}_{s_1, s_2} \cdot \bang_{s_1} x \cdot \bang_{s_2} y = \bang_s (x \cdot y)$ for any $s \preceq s_1, s_2$,
      \item $\varepsilon_s \cdot \bang_s x = x$ for any $s$,
      \item $\delta_s \cdot \bang_s x = \bang_s \bang_s x$ for any $s$,
      \item $k_s \cdot x \cdot (\bang_s y) = x$ for $s = {\bf a}, {\bf i}$,
      \item $w_s \cdot x \cdot \bang_s y = x \cdot \bang_s y \cdot \bang_s y$ for $s = {\bf r}, {\bf i}$.
    \end{itemize}
  \end{multicols}
\end{definition}

\begin{prop}~\label{combinatory:decompose}
  Let $(\mathcal{L}, \bang_{\bf i}, \bang_{\bf a}, \bang_{\bf r})$ be a subexponential linear combinatory algebra. Put  $x \cdot_s y := x \cdot \bang_s y$ for any $s \in \{{\bf i}, {\bf r}, {\bf a} \}$. 
  Then $\mathcal{L}_{\bang_{\bf i}} = (\mathcal{L}, \cdot_{\bf i})$ is a combinatory algebra;
  $\mathcal{L}_{\bang_{\bf r}} = (\mathcal{L}, \cdot_{\bf r})$ is a BCW-algebra;
  $\mathcal{L}_{\bang_{\bf a}} = (\mathcal{L}, \cdot_{\bf a})$ is a BCK-algebra.

  Besides, $\mathcal{L}$ forms functional adjoint pairs with
  $\mathcal{L}_{\bang_{\bf i}}$, $\mathcal{L}_{\bang_{\bf r}}$ and $\mathcal{L}_{\bang_{\bf a}}$
  (denoted as $\mathsf{f}_s^{\natural} \dashv \mathsf{f}_s^{\sharp}$ for $s \in \{ {\bf i}, {\bf r}, {\bf a} \}$)
  such that there are functional applicative morphisms $\mu_{\bf ir} : (\mathcal{L}_{\bang_{\bf i}}, \cdot_{\bf i}) \to (\mathcal{L}_{\bang_{\bf r}}, \cdot_{\bf r})$
  and $\mu_{\bf ia} : (\mathcal{L}_{\bang_{\bf i}}, \cdot_{\bf i}) \to (\mathcal{L}_{\bang_{\bf a}}, \cdot_{\bf a})$
  such that $\mathsf{f}^{\natural}_{\bf i} = \mathsf{f}^{\natural}_{\bf r} \mu_{\bf ir}$ and 
  $\mathsf{f}^{\natural}_{\bf i} = \mathsf{f}^{\natural}_{\bf a} \mu_{\bf ia}$.
\end{prop}
\begin{proof}
  For $(\mathcal{L}, \cdot_{\bf i})$, see \cite[Theorem 3.7]{abramsky2002geometry} for the proof. 
  In $(\mathcal{L}, \cdot_{\bf r})$, $w$ is given as $w := \langle x \rangle \:\: w_{\bf r} \cdot (\varepsilon_{\bf r} \cdot x)$.
  In $\mathcal{L}_{\bang_{\bf a}}$, $k$ is given as $k := \langle x \rangle \:\: k_{\bf a} \cdot (\varepsilon_{\bf a} \cdot x)$.

  For each $s$, the adjoint pair $\mathsf{f}_s^{\natural} \dashv \mathsf{f}_s^{\sharp}$ is given by mappings
  $\mathsf{f}_s^{\natural} : \mathcal{L}_s \to \mathcal{L}$ and $\mathsf{f}_s^{\sharp} : \mathcal{L} \to \mathcal{L}_s$,
  where $\mathsf{f}_s^{\natural}(x) = x$ and $\mathsf{f}_s^{\sharp}(x) = \bang_s x$.
  $\mathsf{f}_s^{\natural}$ and $\mathsf{f}_s^{\sharp}$ are realised with $r^{\natural}_{s}$ and $r^{\sharp}_{s}$ given as:
  \begin{center}
    $r^{\natural}_{s} := \langle x \rangle \langle y \rangle \:\: x \cdot (\bang_s y)$

    $r^{\sharp}_s := \langle x \rangle \langle y \rangle \:\: \mathsf{m}_{s,s} \cdot (\varepsilon_s \cdot x) \cdot (\varepsilon_s \cdot y)$
  \end{center}
  Both $\mathsf{f}_s^{\natural}$ and $\mathsf{f}_s^{\sharp}$ are applicative morphisms.
  For $\mathsf{f}_s^{\natural}$, we need $a \cdot_s b = r^{\natural}_{s} \cdot a \cdot b$, but $r^{\natural}_{s} \cdot a \cdot b = a \cdot \bang_s b = a \cdot_s b$.
  For $\mathsf{f}_s^{\sharp}$, we need $!_s (a \cdot b) = r^{\sharp}_s \cdot_s \bang_s a \cdot_s \bang_s b$, but
  \[
  r^{\sharp}_s \cdot_s \bang_s a \cdot_s \bang_s b = r^{\sharp}_s \cdot \bang_s \bang_s a \cdot \bang_s \bang_s b =  \mathsf{m}_{s,s} \cdot (\varepsilon_s \cdot \bang_s \bang_s a) \cdot (\varepsilon_s \cdot \bang_s \bang_s b) = \mathsf{m}_{s,s} \cdot \bang_s a \cdot \bang_s b = \bang_s (a \cdot b)
  \]
  Now we must ensure $\mathsf{f}_s^{\natural}$ and $\mathsf{f}_s^{\sharp}$ form a comonadic operator.
  On the one hand, we have $\mathsf{f}_s^{\natural} \mathsf{f}_s^{\sharp} \sqsubseteq 1_{\mathcal{L}}$, which is witnessed by $\varepsilon_s$.
  On the other hand, we have $1_{\mathcal{L}_s} \sqsubseteq \mathsf{f}_s^{\sharp} \mathsf{f}_s^{\natural}$, which is witnessed by 
  $\langle x \rangle \:\: \varepsilon_s \cdot (\varepsilon_s \cdot x)$, as we need $r \in \mathcal{L}_s$ such that $r \cdot_s \bang_s a = a$, i.e.,
  $r \cdot \bang_s \bang_s a = a$.
  Besides, $\mathsf{f}_s^{\natural} \mathsf{f}_s^{\sharp} \sqsubseteq \mathsf{f}_s^{\natural} \mathsf{f}_s^{\sharp} \mathsf{f}_s^{\natural} \mathsf{f}_s^{\sharp}$
  is witnessed by $\delta_s$.

  Put $\mu_{\bf ir}(x) = x$ and $\mu_{\bf ia}(x) = x$. Then the required equalities become tautological. Put the following realisers:
  \begin{center}
  $r_{\bf ir} := \langle x \rangle \langle y \rangle \:\: (\varepsilon_{\bf r} \cdot x) \cdot \bang_{\bf i} (\varepsilon_{\bf r} \cdot y)$ \\

  $r_{\bf ia} := \langle x \rangle \langle y \rangle \:\: (\varepsilon_{\bf a} \cdot x) \cdot \bang_{\bf i} (\varepsilon_{\bf a} \cdot y)$
  \end{center}
  $\mu_{\bf ir}$ is also an applicative morphism since:
  \[ 
  r_{\bf ir} \cdot_{\bf r} a \cdot_{\bf r} b = (\langle x \rangle \langle y \rangle \:\: (\varepsilon_{\bf r} \cdot x) \cdot \bang_{\bf i} (\varepsilon_{\bf r} \cdot y)) \cdot \bang_{\bf r} a \cdot \bang_{\bf r} b = (\varepsilon_{\bf r} \cdot \bang_{\bf r} a) \cdot \bang_{\bf i} (\varepsilon_{\bf r} \cdot \bang_{\bf r} b) = a \cdot \bang_{\bf i} b = a \cdot_{\bf i} b 
  \]
\end{proof}

Now we finally show we can obtain $\Sigma_3$-assemblages from relational subexponential linear combinatory algebras.

\begin{theorem}~\label{combinatory:assemblage}
  Let $(\mathcal{L}, \bang_{\bf i}, \bang_{\bf r}, \bang_{\bf a})$ be a relational subexponential linear combinatory algebra,
  the categories $\operatorname{Asm}(\mathcal{L})$ and $\operatorname{Mod}(\mathcal{L})$ are endowed with the structure of a $\Sigma_3$-assemblage.
  Therefore any relational subexponential linear combinatory algebra induces a categorical model of ${\bf SILL}(\lambda)_{\Sigma_3}$.
\end{theorem}

\begin{proof}
  $\operatorname{Asm}(\mathcal{L})$ is an SMCC since $\mathcal{L}$ is BCI. We also lift 
  comonadic applicative morphisms $\bang_{\bf i}$, $\bang_{\bf r}$ and $\bang_{\bf a}$ to 
  symmetric lax monoidal comonads ${\bang_{\bf i}}_*, {\bang_{\bf r}}_*, {\bang_{\bf a}}_* : \operatorname{Asm}(\mathcal{L}) \to \operatorname{Asm}(\mathcal{L})$.
  ${\bang_{\bf i}}_*$ is equipped with natural transformations $\mathsf{c}^{\bf i}$ and $\mathsf{d}^{\bf i}$ 
  with the components $\mathsf{c}^{\bf i}_A : {\bang_{\bf i}}_* A \to {\bang_{\bf i}}_* A \otimes {\bang_{\bf i}}_* A$ and 
  $\mathsf{d}^{\bf i}_A : {\bang_{\bf i}}_* A \to \mathds{1}$ given as $\mathsf{c}^{\bf i}_A : x \mapsto (x,x)$ and $\mathsf{d}^{\bf i}_A : x \mapsto *$, whose
  realisers are realisers of $\bang_{\bf i} \sqsubseteq \bang_{\bf i} \bullet \bang_{\bf i}$ and $\bang_{\bf i} \sqsubseteq (x \mapsto \{ i \})$ respectively.
  Similarly, ${\bang_{\bf r}}_*$ and ${\bang_{\bf a}}_*$ are equipped with natural transformations
  $\mathsf{c}^{\bf r}$ and $\mathsf{d}^{\bf a}$ given as $\mathsf{c}^{\bf r}_A : x \mapsto (x,x)$ and $\mathsf{d}^{\bf a}_A : x \mapsto *$
  realised by realisers of $\bang_{\bf r} \sqsubseteq \bang_{\bf r} \bullet \bang_{\bf r}$ and  $\bang_{\bf a} \sqsubseteq (x \mapsto \{ i \})$.
  Therefore, ${\bang_{\bf i}}_*$ is an exponential comonad, ${\bang_{\bf r}}_*$ is relevant and ${\bang_{\bf a}}_*$ is affine. Checking the coherence conditions is routine.
  $\bang_{\bf i} \sqsubseteq \bang_{\bf r}$ and $\bang_{\bf i} \sqsubseteq \bang_{\bf a}$ are lifted to the symmetric lax monoidal natural transformations 
  $\mu_{\bf ir} : {\bang_{\bf i}}_* \Rightarrow {\bang_{\bf r}}_*$ and $\mu_{\bf ia} : {\bang_{\bf i}}_* \Rightarrow {\bang_{\bf a}}_*$. 
  Besides, $\mu_{\bf ir}$ and  $\mu_{\bf ia}$ are symmetric lax monoidal comonad morphisms. In particular, the following diagram commutes (the  $\mu_{\bf ia}$ clause is similar):
  \[
  \begin{tikzcd}
  {\bang_{\bf i}}_* \ar[d, "\delta^{\bf i}"'] \ar[rrrr, "\mu_{\bf ir}"] &&&& {\bang_{\bf r}}_* \ar[d, "\delta^{\bf r}"]                                                    & {\bang_{\bf i}}_* \ar[rr, "\mu_{\bf ir}"] \ar[dr, "\varepsilon^{\bf i}"'] && {\bang_{\bf r}}_*  \ar[dl, "\varepsilon^{\bf r}"] \\
  {\bang_{\bf i}}_*^2 \ar[rr, "\mu_{\bf ir} {\bang_{\bf i}}_*"'] && {\bang_{\bf r}}_* {\bang_{\bf i}}_* \ar[rr, "{\bang_{\bf r}}_* \mu_{\bf ir}"'] && {\bang_{\bf r}}_*^2  && 1_{\operatorname{Asm}(\mathcal{L})}       
  \end{tikzcd}
  \]
  Recall that $\sqsubseteq$ respects the composition of applicative morphisms, so the left-hand side diagram commutes since
  $\bang_{\bf i} \sqsubseteq {\bang_{\bf i}}^2 \sqsubseteq \bang_{\bf r} {\bang_{\bf i}} \sqsubseteq {\bang_{\bf r}}^2$ as well as $\bang_{\bf i} \sqsubseteq \bang_{\bf r} \sqsubseteq \bang_{\bf r}^{2}$.
  The right-hand side diagram holds similarly since $\bang_{\bf i} \sqsubseteq \bang_{\bf r} \sqsubseteq 1_{\mathcal{L}}$.
  The rest is
  to ensure that the following square commutes. The affine clause is considered similarly:
  \[
  \begin{tikzcd}
    {\bang_{\bf i}}_* A \ar[d, "{\mu_{\bf ir}}_A"'] \ar[rr, "\mathsf{c}^{\bf i}"] && {\bang_{\bf i}}_* A \otimes {\bang_{\bf i}}_* A \ar[d, "{\mu_{\bf ir}}_A \otimes {\mu_{\bf ir}}_A"] \\
    {\bang_{\bf r}}_* A \ar[rr, "\mathsf{c}^{\bf r}"'] && {\bang_{\bf r}}_* A \otimes {\bang_{\bf r}}_* A
  \end{tikzcd}
  \]
  But the underlying map of the right-hand side arrow is $(x, y) \mapsto (x, y)$, so the diagram becomes a tautology.

  The statement for the category of modest sets is concluded by analogy with \cite[Proposition 14]{hoshino2007linear}.
\end{proof}

Finally, we discuss some examples from computation theory giving concrete models of ${\bf SILL}(\lambda)_{\Sigma_3}$ by Theorem~\ref{combinatory:assemblage}.

\begin{example}~\label{term:combinatory}
  The simplest example of a subexponential linear combinatory algebra comes from the term calculus for ${\bf SILL}(\lambda)_{{\Sigma}_3}$.
  In particular, the syntactic subexponential linear combinatory algebra $\mathcal{L}_{{\bf SILL}(\lambda)_{{\Sigma}_3}}$ on closed terms of the ${\bf SILL}(\lambda)_{{\Sigma}_3}$
  calculus modulo conversion $\equiv$ is given by the following combinators.
  \begin{itemize}
    \item $\bang_s : M(\vec{x}) \mapsto \bang_s (\lambda \vec{x}. M) \: {\bf with} \: \underline{\:\:} = \underline{\:\:}$ 
    \footnote{As a matter of fact, this definition of $\bang_s$ encodes the Curry-Howard counterpart of the modal necessitation rule: from $\vdash A$ infer $\vdash \bang_s A$.},
    \item $i = \lambda x. x$, $b = \lambda f g x. f (g x)$, $c = \lambda f x y. f y x$,
    \item $M \cdot N := (\lambda f. \lambda x. f x) \: M \: N$,
    \item $\mathsf{m}^s_{s_1,s_2} := \lambda f. \lambda x. \bang_s \: ({\bf der}_{s_1} (f')) ({\bf der}_{s_2} (x')) \:  {\bf with} \: f', x' = f,x$ for $s \preceq s_1, s_2$,
    \item $\varepsilon_s := \lambda x. {\bf der}_s(x)$ for any $s$,
    \item $\delta_s := \lambda x. \bang_s \: x \: {\bf with} \: x = x$,
    \item $w_s := \lambda f. \lambda z. {\bf let}_s \: x, y = z \: {\bf in} \: f \: x \: y$ for $s = {\bf i}, {\bf r}$,
    \item $k_s := \lambda x. \lambda y. {\bf del}_s(y; x)$ for $s = {\bf i}, {\bf a}$.
  \end{itemize}
  The axioms of a subexponential linear combinatory algebra follow from the corresponding proof normalisation conversions. Therefore, we have a syntactic realisability model
  $\operatorname{Asm}(\mathcal{L}_{{\bf SILL}(\lambda)_{{\Sigma}_3}})$ of ${\bf SILL}(\lambda)_{{\Sigma}_3}$.
\end{example}

\begin{example}
  This example is the categorification of Example~\ref{term:combinatory}. 
  Given a $\Sigma_3$-assemblage $(\mathcal{C}, \bang_{\bf i}, \bang_{\bf r}, \bang_{\bf a}, \mu_{\bf ir}, \mu_{\bf ia})$
  with a reflexive object $V$ with retractions $v_1 : V \multimap V \triangleleft V : v_2$ and $p_s : \bang_s V \triangleleft V : q_s$ for $s \in \{ {\bf i}, {\bf r}, {\bf a} \}$. 
  Let $f, g \in \mathcal{C}(\mathds{1}, V)$, then put $f \cdot g := \operatorname{ev} ((v_1 \circ f) \otimes g)$ and 
  $\Box_s(f) := p_s \circ (\bang_s f) \circ \mathsf{m}^{s}_{\mathds{1}}$. Then $(\mathcal{C}(\mathds{1}, V), \cdot, (\Box_s)_{s \in \{ {\bf i}, {\bf r}, {\bf a} \}})$
  forms a subexponential linear combinatory algebra. The proof is transferable from \cite[Theorem 3.5]{abramsky2002geometry}.
\end{example}

\begin{example}
  This example adapts the example of a linear combinatory algebra (i.e. a BCI algebra with a single $\bang_{\bf i}$) from \cite[Example 5.3]{hoshino2007linear}. One can also think of this 
  subexponential linear combinatory algebra as the subexponential linearisation of Scott's graph model widely discussed in recursion theory and
  realisability, see \cite{scott1975lambda}.
  First of all, recall that $\mathcal{P}(\omega)$ is a BCI-algebra since $\omega$ is an object 
  in the compact closed category ${\bf Rel}$ such that $\omega \times \omega \triangleleft \omega$ and $\omega^* = \omega$.
  Let $\langle .,. \rangle : \omega \times \omega \to \omega$ denote a bijective pair enumeration function and let 
  $e : \omega \to \mathcal{P}_{fin}(\omega)$ denote a bijective encoding of finite subsets of $\omega$. 
  Then $\mathcal{P}(\omega)$ forms a subexponential linear combinatory algebra with the following structure for $A, B \subseteq \omega$:
  \begin{itemize}
    \item $A \cdot B = \{ n \: | \: \langle m, n \rangle \in A, m \in B \}$,
    \item $\bang_{\bf i}(A) = \{ n \: | \: e(n) \subseteq A \}$,
    \item $\bang_{\bf r}(A) = \{ n \: | \: \emptyset \neq e(n) \subseteq A \}$,
    \item $\bang_{\bf a}(A) = \{ n \: | \: e(n) \subseteq A \land |e(n)| \leq 1 \}$,
    \item $\mathsf{m}^{s}_{s_1,s_2} = \{ \langle m, \langle n, p \rangle \rangle \: | \: e(p) \subseteq e(m) \cdot e(n) \land \varphi_{s_1}(m) \land \varphi_{s_2}(n) \land \varphi_s(p) \}$ for $s \preceq s_1, s_2$,
    where $\varphi_{s_1}(.), \varphi_{s_2}(.), \varphi_s(.)$ are predicates corresponding to the index, in particular 
    \begin{center}
    $\varphi_{\bf i}(x) = e(x) \subseteq \omega$, \\
    $\varphi_{\bf r}(x) = e(x) \subseteq \omega \land e(x) \neq \emptyset$, \\
    $\varphi_{\bf a}(x) = e(x) \subseteq \omega \land |e(x)| \leq 1$.
    \end{center}
    \item $\delta_{\bf i} = \{ \langle m, n \rangle \: | \: \cup_{k \in e(n)} e(k) \subseteq e(m) \}$, 
    \item $\delta_{\bf r} = \{ \langle m, n \rangle \: | \: \cup_{k \in e(n)} e(k) \subseteq e(m), e(n) \neq \emptyset, \forall k \in e(n) \: e(k) \neq \emptyset \}$,
    \item $\delta_{\bf a} = \{ \langle m, n \rangle \: | \: \cup_{k \in e(n)} e(k) \subseteq e(m), |e(n)| \leq 1 \land k \in e(n) \to |e(k)| \leq 1 \}$,
    \item $\varepsilon_{s} = \{ \langle m, n \rangle \: | \: n \in e(m) \}$ for any $s$,
    \item $w_{\bf i} = \{ \langle k, \langle l, m \rangle \rangle \: | \: k = \langle i, \langle j, m \rangle \rangle \land e(i) \cup e(j) \subseteq e(l) \}$,
    \item $w_{\bf r} = \{ \langle k, \langle l, m \rangle \rangle \: | \: k = \langle i, \langle j, m \rangle \rangle \land e(i) \cup e(j) \subseteq e(l), e(i), e(j) \neq \emptyset \}$,
    \item $k_{\bf i} = k_{\bf a} = \{ \langle m, \langle n, m\rangle \rangle \: | \: n,m \in \omega \}$.
  \end{itemize}
  The above also works if we restrict $\mathcal{P}(\omega)$ to only recursively enumerable subsets of $\omega$. Let $\mathcal{P}_{re,li}(\omega)$ denote such a subexponential linear combinatory algebra.
\end{example}

\begin{example}
  As Longley showed in \cite[Proposition 3.3.7]{longley1995realizability}, there is an adjoint pair 
  $\kappa^{\natural} \dashv \kappa^{\sharp} : \mathcal{P}_{re}(\omega) \to \mathcal{K}_1$, where $\mathcal{P}_{re}(\omega)$ is Scott's graph model on recursively enumerable sets.
  Further, $\mathcal{P}_{re,li}(\omega)$ is a subexponential linear combinatory algebra, so, by Proposition~\ref{combinatory:decompose}, 
  we have adjoint pairs $f_{\bf i}^{\natural} \dashv f_{\bf i}^{\sharp}$ between $(\mathcal{P}_{re, li}(\omega), \cdot)$ and $(\mathcal{P}_{re, li}(\omega), \cdot_{\bf i})$,
  $f_{\bf r}^{\natural} \dashv f_{\bf r}^{\sharp}$ between $(\mathcal{P}_{re, li}(\omega), \cdot)$ and $(\mathcal{P}_{re, li}(\omega), \cdot_{\bf r})$ and 
  $f_{\bf a}^{\natural} \dashv f_{\bf a}^{\sharp}$ between $(\mathcal{P}_{re, li}(\omega), \cdot)$ and $(\mathcal{P}_{re, li}(\omega), \cdot_{\bf a})$
  with transition applicative morphisms $\widehat{\mu_{\bf ir}} : (\mathcal{P}_{re, li}(\omega), \cdot_{\bf i}) \to (\mathcal{P}_{re, li}(\omega), \cdot_{\bf r})$
  and $\widehat{\mu_{\bf ia}} : (\mathcal{P}_{re, li}(\omega), \cdot_{\bf i}) \to (\mathcal{P}_{re, li}(\omega), \cdot_{\bf a})$.

  But note that $(\mathcal{P}_{re}(\omega), \cdot)$ and $(\mathcal{P}_{re,li}(\omega), \cdot_{\bf i})$ are equivalent by \cite[Proposition 3.3.4]{longley1995realizability}
  since $(\mathcal{P}_{re,li}(\omega), \cdot_{\bf i})$ is a combinatory algebra by Proposition~\ref{combinatory:decompose}, so both combinatory algebras coincide up to the computable
  encodings of finite sets and tuples.
  Therefore, we have the adjoint pair $\kappa'^{\natural} \dashv \kappa'^{\sharp} : (\mathcal{P}_{re,li}(\omega), \cdot_{\bf i}) \to \mathcal{K}_1$,
  the adjoint pair composing $\kappa^{\natural} \dashv \kappa^{\sharp}$ with the adjoint pair establishing the equivalence between $ \mathcal{P}_{re}(\omega)$ and $(\mathcal{P}_{re,li}(\omega), \cdot_{\bf i})$.
  Therefore, $(\mathcal{P}_{re,li}(\omega), \bang'_{\bf i}, \bang_{\bf r}, \bang_{\bf a})$ is a 
  relational subexponential linear combinatory algebra by Proposition~\ref{combinatory:bouton},
  where $\bang'_{\bf i} := f^{\natural}_{\bf i} \kappa'^{\natural} \kappa'^{\sharp} f^{\sharp}_{\bf i} $
  and $\bang'_{\bf i} \sqsubseteq \bang_{\bf r}$ and $\bang'_{\bf i} \sqsubseteq \bang_{\bf a}$
  are realised just by the transitivity of $\sqsubseteq$.
\end{example}
\section{Conclusion}~\label{conclusion}

In this paper, we described the Curry-Howard correspondence for intuitionistic multiplicative linear logic with subexponential modalities
and suggested a possible way of specifying functional algorithms maintaining several resource usage policies within a single setting.
We expanded the concept of a linear category for a polymodal case with multiple comonadic subexponential modalities
by coining the concept of a $\Sigma$-assemblage. We not only generalised Benton's linear-non-linear models by suggesting the notion of a $\Sigma$-bouton, but 
also showed that the 2-category of all $\Sigma$-assemblages for any $\Sigma$ 1,2-fully faithfully embeds into the 2-category of $\Sigma$-boutons.
This 2-category is constructed as the Eilenberg-Moore 2-category of coalgebras over the strict 2-comonad on the full subcategory of cocones over indexed families
of symmetric monoidal categories and transition strict symmetric monoidal functors between them. The concept of a $\Sigma$-assemblage seems to be rather 
natural as there are examples of this concept in algebraic geometry, computability theory and game semantics. 
We also suggested the realisability interpretation for our term calculus by generalising the concepts of linear combinatory algebras as well as 
relational linear combinatory algebras. 

There are several questions we would like to leave for further research. First of all, non-commutative generalisations of $\Sigma$-assemblages
would provide denotational semantics for the Lambek calculus with subexponentials as it is described in \cite{kanovich2017undecidability} and \cite{kanovich2019subexponentials}.
Of course, one can give a rather natural non-symmetric generalisation, but providing concrete instances of non-commutative assemblages
from, for example, non-commutative geometry would be of a considerable interest. Also, it is a well-known fact that a realisability topos can be defined
as the completion of the category of assemblies over a PCA, see \cite{van2008realizability}. So it would be of interest as well to characterise completions of the 
category of assemblies over relational subexponential linear combinatory algebras and what sort of syntax such categories can potentially describe. 
Besides, it might be worth relating our approach to other treatments of linear realisability for dependent type theory, such as the recent realisability model 
of linear dependent type theory from \cite{speight2026impredicativity}.
\begin{appendices}
\section{Proof Normalisation Rules}~\label{appendix}

As for linear type theory with the single exponential modality $\bang$ considered in \cite{benton1993term}, one can distinguish the $\beta\eta$-reduction relation 
and the commuting conversion relation which are conversions between terms arising from the proof normalisation analysis.
Although all those commuting conversions can be transferred to ${\bf SILL}(\lambda)_{\Sigma}$, we consider a simpler approach
that combines \cite{troelstra1995natural} and \cite[§4.3]{UCAM-CL-TR-346}. 
Let us specify the term conversions for ${\bf SILL}(\lambda)_{\Sigma}$.

Before specifying the conversion rules, let us introduce some convenient shorthand notation for a tuple of terms $\vec{M} = (M_0, \ldots, M_n)$:
\begin{itemize}
  \item ${\bf del}_{s_0, \ldots, s_n}(\vec{M}; N) := {\bf del}_{s_0}(M_0; ({\bf del}_{s_1}(M_1; \ldots ({\bf del}_{s_n}(M_n; N))\ldots)))$,
  \item ${\bf let}_{s_0, \ldots, s_n} \: \vec{y}, \vec{z} = \vec{M} \: {\bf in} \: N := {\bf let}_{s_0} \: y_0, z_0 = M_0 \: {\bf in} \: ({\bf let}_{s_1} \: y_1, z_1 = M_1 \: {\bf in} \: (\ldots ({\bf let}_{s_n} \: y_n, z_n = M_n \: {\bf in} \: N) \ldots))$.
\end{itemize}

Figure~\ref{beta:conv} describes the computational $\beta$-reduction rules:

\begin{figure}[H]
  \hrule
  \centering
  \caption{$\beta$-reduction}~\label{beta:conv}
  \begin{small}
  \begin{itemize}
    \item[$(\multimap\beta)$] $(\lambda x. M) N \triangleleft M[x := N]$,
    \item[$(\mathds{1}\beta)$] ${\bf let} \: {\bf 1} = {\bf 1} \: {\bf in} \: N \triangleleft N$,
    \item[$(\otimes\beta)$] ${\bf let} \: x \otimes y = M \otimes N \: {\bf in} \: P \triangleleft P [x := M, y := N]$,
    \item[$(\bang_s\beta)$] ${\bf der}_{s} \: (\bang_{s} \: N \: {\bf with} \: \vec{x} = \vec{M}) \triangleleft N [\vec{x} := \vec{M}]$ for $s \in \Sigma$,
    \item[$({\bf W}\beta)$] Let $s \preceq s_1, \ldots, s_n$ and $s \in W$, then
      ${\bf del}_{s} ({\bang_s} N \: {\bf with} \: \vec{x} = \vec{M}; N') \triangleleft {\bf del}_{s_1, \ldots, s_n}(\vec{M}; N')$
    \item[$({\bf C}\beta)$] Let $s \in C$
    \[\begin{array}{lll}
    & {\bf let}_{s} \: y, z = (\bang_{s} M \: {\bf with} \: \vec{x} = \vec{M'}) \: {\bf in} \: N \triangleleft & \\
    & \:\: {\bf let}_{s, \vec{s'}, \vec{s''}} \: \vec{x'}, \vec{x''} = \vec{M'} \: {\bf in} \: N [ y:= \bang_s M  \: {\bf with} \: \vec{x} = \vec{x'}, z:= \bang_s M  \: {\bf with} \: \vec{x} = \vec{x''}]. &
    \end{array}
    \]
  \end{itemize}
\end{small}
  \hrule
\end{figure}

The extensional $\eta$-reduction rules from Figure~\ref{eta:conv} can be considered as \emph{simplification conversions} in the derivations
where an introduction rule was used after applying the corresponding elimination rule. 

\begin{figure}[H]
  \hrule
  \centering
  \caption{$\eta$-reduction}~\label{eta:conv}
  \begin{small}
  \begin{itemize}
    \item[$(\multimap \eta)$] $(\lambda x. M x) \triangleleft M$ if $x$ is not free in $M$,
    \item[$(\bang_s \eta)$] $\bang_s ({\bf der}_{s} \: x) \: {\bf with} \: x = M \triangleleft M$,
    \item[$(\otimes \eta)$] ${\bf let} \: x \otimes y = M \: {\bf in} \: N[z := x \otimes y] \triangleleft N [z:= M]$.
  \end{itemize}
  \end{small}
  \hrule
\end{figure}

The conversions from Figure~\ref{naturality:conv} syntactically depict that the tensor product and unit types
alongside the weakening and contraction operations
satisfy the categorical naturality property. From a proof-theoretic angle,
those conversions reflect \emph{permutation conversions} where we permute minor premises of an elimination rule in a certain way.

\begin{figure}[!htbp]
  \hrule
  \centering
  \caption{The naturality conversions for $\otimes$, $\mathds{1}$, weakening and contraction}~\label{naturality:conv}
  \begin{small}
  \begin{itemize}
    \item[$(\otimes_{\bf nat})$] $P[w := {\bf let} \: x \otimes y = M \: {\bf in} \: N] \triangleleft {\bf let} \: x \otimes y = M \: {\bf in} \: P [w := N]$,
    \item[$(\mathds{1}_{\bf nat})$] $P [w := {\bf let} \: {\bf 1} = M \: {\bf in} \: N] \triangleleft {\bf let} \: {\bf 1} = M \: {\bf in} \: P[w := N]$,
    \item[$({\bf del}_s)_{\bf nat}$] Let $s \in W$, then $P [x := {\bf del}_s(M; N)] \triangleleft {\bf del}_s(M; P [x := N])$,
    \item[$({\bf let}_s)_{\bf nat}$] Let $s \in C$, then $P [z := {\bf let}_s \: x, y = M \: {\bf in} \: N] \triangleleft {\bf let}_s \: x, y = M \: {\bf in} \: P [z := N]$.
  \end{itemize}
  \end{small}
  \hrule
\end{figure}

The conversions from the rest of the figures elaborate on how copying, deleting and promotion term constructors commute with one another.
In particular, the conversions having the form of $({\bf W}_{\operatorname{conv}})$ say that the deletion operation has a coalgebraic nature, 
but we can also view it as a proof conversion allowing us to avoid some redundant promotions when we combine
the promotion and deletion operations. Categorically, the conversions $({\bf C}_{{\operatorname{conv}_1}})$, $({\bf C}_{{\operatorname{conv}_2}})$ and $({\bf C}_{{\operatorname{conv}_3}})$ 
demonstrate the coalgebraic and cocommutative comonoid nature of the contraction operation, but, as above, 
we think of them as a separate sort of prooftree transformation that can be reduced to none of the above. The conversion $({\bf CW}_{\operatorname{conv}})$ specifies how exactly 
the exponential deletion and copying terms commute with each other. One can think of this conversion
as the syntactic encoding of one of the cocommutative comonoid axioms.

\begin{figure}[!htbp]
  \hrule
  \centering
  \caption{The conversions for weakening and contraction}~\label{structural:conv}
  \begin{small}
  \begin{itemize}
    \item[$({\bf W}_{\operatorname{conv}})$] Let $s, t \in \Sigma$, $s \in W$ and $t \preceq s$, then
    $\bang_{t}({\bf del}_{s}(x; M)) \: {\bf with} \: x, \vec{x} = N, \vec{N} \triangleleft {\bf del}_{s}(N; (\bang_{t} M \: {\bf with} \: \vec{x} = \vec{N}))$,
    \item[$({\bf C}_{{\operatorname{conv}_1}})$]~\label{contracting:conv:1} Let $s \preceq s_1$ and $s_1 \in C$

  \[\begin{array}{lll}
    & \bang_s \: ({\bf let}_{s_1} \: y, z = x_1 \: {\bf in} \: N) \: {\bf with} \: x_1, \ldots, x_n = M_1, \ldots, M_n \triangleleft & \\
    & \:\:\:\: {\bf let}_{s_1} \: y', z' = M_1 \: {\bf in} \: (\bang_s N \: {\bf with} \: y, z, x_2, \ldots, x_n = y', z', M_2, \ldots, M_n) & 
  \end{array}
  \]
    \item[$({\bf C}_{{\operatorname{conv}_2}})$]~\label{contracting:conv:2} ${\bf let}_s \: x, y = M \: {\bf in} \: N \triangleleft {\bf let}_s \: y, x = M \: {\bf in} \: N$,
    \item[$({\bf C}_{{\operatorname{conv}_3}})$]~\label{contracting:conv:3} 
    ${\bf let}_{s} \: x, w = \: M \: {\bf in} \: ({\bf let}_{s} \: y, z = w \: {\bf in} \: N) \triangleleft {\bf let}_{s} \: w, z = M \: {\bf in} \: ({\bf let}_{s} \: x,y = w \: {\bf in} \: N)$,
    \item[$({\bf CW}_{\operatorname{conv}})$] ${\bf let}_{s} \: x, y = M \: {\bf in} \: {\bf del}_{s}(x; N) \triangleleft N[y := M]$ for $s \in W \cap C$.
  \end{itemize}
  \end{small}
  \hrule
\end{figure}

The conversion $({\bang_s}_{\bf conv})$ was introduced by Troelstra for intuitionistic
linear logic, see \cite[X1]{troelstra1995natural}. We introduce the version of that conversion for all subexponentials in a more general form. The conversion 
looks counterintuitive, but it encodes the fact that every subexponential, as a symmetric lax monoidal functor, respects associators.

\begin{figure}[H]
  \hrule
  \caption{Extra $\bang_s$-conversion}~\label{box:conv}
  \begin{itemize}
    \item[$(\bang_{t})_{\bf conv}$] $\begin{array}{lll}
        & \bang_{t} M \: {\bf with} \: \vec{y'}, y, \vec{y''} = \vec{M'}, (\bang_{s} N \: {\bf with} \: \vec{x} = \vec{M}), \vec{M''} \triangleleft & \\
        & \:\:\:\: \bang_{t} (M [y := \bang_{s} N \: {\bf with} \: \vec{x'} = \vec{x}]) \: {\bf with} \: \vec{y'}, \vec{x'}, \vec{y''} = \vec{M'}, \vec{M}, \vec{M''} &
      \end{array}$
  \end{itemize}
  \hrule
\end{figure}
\end{appendices}

\bibliographystyle{alpha}
\bibliography{main}

\end{document}